\documentclass[12pt,reqno]{amsart} 

\usepackage{float} 

\usepackage{soul}

\usepackage{hhline}

\usepackage{epsfig,amssymb,amsmath,version}
\usepackage{amssymb,version,graphicx,fancybox,mathrsfs}

\usepackage{enumitem}
 \usepackage{bbm}
\usepackage{url,hyperref}
\usepackage{subfigure}
\usepackage{color}
\usepackage{stmaryrd}
\usepackage{multirow}
\usepackage{booktabs,siunitx}
\usepackage{booktabs}
\usepackage{multicol,epstopdf}
\usepackage{xypic}
\usepackage[]{graphicx}
\usepackage[all]{xy}
\usepackage{tikz,ifthen}
\usetikzlibrary{positioning, math, decorations.markings, arrows.meta}
\usetikzlibrary{calc,shapes,cd}
\usetikzlibrary{knots} 
\usepgfmodule{decorations}
\allowdisplaybreaks
\tikzset{knotarrow/.pic={ \draw[edge, <-] (0,0) -- +(-.001,0);}}

\tikzset{edge/.style={line width=0.8}}
\tikzset{wall/.style={very thick}}

\tikzset{->-/.style n args={2}{decoration={markings, mark=at position #1 with {\arrow{#2}}}, postaction={decorate}}} 

\tikzset{-o-/.code 2 args={\ifstreqF{#2}{} 
{\ifstreqTF{#2}{>}
   {\pgfkeysalso{decoration={markings,mark=at position #1 with {\arrow[scale=0.8]{#2}}}
                    ,postaction={decorate}}
    }
   {\ifstreqTF{#2}{<}
       {\pgfkeysalso{decoration={markings,mark=at position #1 with {\arrow[scale=0.8]{#2}}}
                    ,postaction={decorate}}
        }
       {\pgfkeysalso{decoration={markings,
                    mark=at position #1 with
                    {\draw[black, fill={#2}] circle[radius=2pt];}}
                    ,postaction={decorate}}
        }
     }
  }}}

\allowdisplaybreaks[4]

\newtheorem{theorem}{Theorem}[section]
\newtheorem{lemma}[theorem]{Lemma}
\newtheorem{definition}[theorem]{Definition}
\newtheorem{corollary}[theorem]{Corollary}
\newtheorem{proposition}[theorem]{Proposition}
\newtheorem{remark}[theorem]{Remark}

\newcommand{\bp}{\begin{proposition}}
\newcommand{\ep}{\end{proposition}}
\newcommand{\bpr}{\begin{proof}}
\newcommand{\epr}{\end{proof}}

\newcommand{\bt}{\begin{theorem}}
\newcommand{\et}{\end{theorem}}

\newcommand{\bl}{\begin{lemma}}
\newcommand{\el}{\end{lemma}}

\newcommand{\bcr}{\begin{corollary}}
\newcommand{\ecr}{\end{corollary}}

\newcommand{\be}{\begin{equation}}
\newcommand{\ee}{\end{equation}}

\newcommand{\bes}{\begin{equation*}}
\newcommand{\ees}{\end{equation*}}

\newcommand{\ba}{\begin{align}}
\newcommand{\ea}{\end{align}}

\newcommand{\bas}{\begin{align*}}
\newcommand{\eas}{\end{align*}}

\DeclareMathOperator{\im}{\mathrm{Im}}

\DeclareMathOperator{\Int}{\mathrm{Int}}

\DeclareMathOperator{\sgn}{\mathrm{sgn}}

\graphicspath{{./Figures/} } 

\newcommand{\vs}[0]{\vspace{2mm}}

\newcommand{\red}[1]{#1}

\begin{document}

\title{Naturality of ${\rm SL}_n$ quantum trace maps for surfaces}

\author[Hyun Kyu Kim]{Hyun Kyu Kim}
\address{School of Mathematics, Korea Institute for Advanced Study (KIAS), 85 Hoegi-ro, Dongdaemun-gu, Seoul 02455, Republic of Korea}
\email{hkim@kias.re.kr}

\author[Zhihao Wang]{Zhihao Wang}
\address{School of Mathematics, Korea Institute for Advanced Study (KIAS), 85 Hoegi-ro, Dongdaemun-gu, Seoul 02455, Republic of Korea}
\email{zhihaowang@kias.re.kr}

\keywords{${\rm SL}_n$-skein algebras, ${\rm SL}_n$-quantum trace, quantum cluster varieties, quantum higher Teichm{\"u}ller theory}

 \maketitle

\begin{abstract}
The ${\rm SL}_n$-skein algebra of a punctured surface $\mathfrak{S}$, studied by Sikora, is an algebra generated by isotopy classes of $n$-webs living in the thickened surface $\mathfrak{S} \times (-1,1)$, where an $n$-web is a union of framed links and framed oriented $n$-valent graphs satisfying certain conditions. For each ideal triangulation $\lambda$ of $\mathfrak{S}$, L\^e and Yu constructed an algebra homomorphism, called the ${\rm SL}_n$-quantum trace, from the ${\rm SL}_n$-skein algebra of $\mathfrak{S}$ to a so-called balanced subalgebra of the $n$-th root version of Fock and Goncharov's quantum torus algebra associated to $\lambda$. We show that the ${\rm SL}_n$-quantum trace maps for different ideal triangulations are related to each other via a balanced $n$-th root version of the quantum coordinate change isomorphism, which extends Fock and Goncharov's isomorphism for quantum cluster varieties. We avoid heavy computations in the proof, by using the splitting homomorphisms of L\^e and Sikora, and a network dual to the $n$-triangulation of $\lambda$ studied by Schrader and Shapiro.
\end{abstract}

\tableofcontents

\newcommand{\ca}{{\cev{a}  }}

\def\BZ{\mathbb Z}
\def\Id{\mathrm{Id}}
\def\Mat{\mathrm{Mat}}
\def\BN{\mathbb N}

\def \cb {\color{blue}}
\def \cred {\color{red}}
\def \cbf {\color{blue}\bf}
\def \credf {\color{red}\bf}
\definecolor{ligreen}{rgb}{0.0, 0.3, 0.0}
\def \cg {\color{ligreen}}
\def \cgf {\color{ligreen}\bf}
\definecolor{darkblue}{rgb}{0.0, 0.0, 0.55}
\def \dbf {\color{darkblue}\bf}
\definecolor{anti-flashwhite}{rgb}{0.55, 0.57, 0.68}
\def \afw {\color{anti-flashwhite}}
\def\cF{\mathcal F}
\def\cP{\mathcal P}
\def\embed{\hookrightarrow}
\def\pr{\mathrm{pr}}
\def\cV{\mathcal V}
\def\ot{\otimes}
\def\buu{{\mathbf u}}


\def \ri {{\rm i}}
\newcommand{\bs}[1]{\boldsymbol{#1}}
\newcommand{\cev}[1]{\reflectbox{\ensuremath{\vec{\reflectbox{\ensuremath{#1}}}}}}
\def\bS{\bar \fS}
\def\cE{\mathcal E}
\def\fB{\mathfrak B}
\def\cR{\mathcal R}
\def\cY{\mathcal Y}
\def\cS{\mathscr S}
\def\rS{\overline{\cS}_\omega}

\def\fS{\mathfrak{S}}

\def\MN {(M)}
\def\cN {\mathcal{N}}
\def\SL{{\rm SL}_n}

\def\bP{\mathbb P}
\def\bR{\mathbb R}

\def\SS{\cS_{\omega}(\fS)}
\def\rdS{\overline \cS_{\omega}(\fS)}
\def\rdP{\overline \cS_{\omega}(\mathbb{P}_4)}

\newcommand{\beq}{\begin{equation}}
	\newcommand{\eeq}{\end{equation}}

\section{Introduction}

\subsection{Overview}\label{overview}

Let $\fS$ be a {\em punctured surface}, that is, a closed oriented surface minus a non-empty finite set of points called punctures such that every component of $\fS$ contains at least one puncture. For an algebraic group ${\rm G}$, one considers the moduli space of ${\rm G}$-local systems on $\fS$, which is identified with the space
$$
\mathscr{L}_{{\rm G},\fS} = {\rm Hom}(\pi_1(\fS), {\rm G})/{\rm G},
$$
which is the quotient of the space ${\rm Hom}(\pi_1(\frak{S}),{\rm G})$ of group homomorphisms $\pi_1(\fS) \to {\rm G}$ by the action of ${\rm G}$ given by  conjugation by an element of ${\rm G}$. 
Here we regard this space as a stack, with the algebra of functions $\mathcal{O}(\mathscr{L}_{{\rm G},\fS})$ being the algebra of ${\rm G}$-invariant functions $\mathcal{O}({\rm Hom}(\pi_1(\frak{S}),{\rm G}))^{\rm G}$; one might call this space the ${\rm G}$-character stack for $\fS$. 
This space is equipped with a natural symplectic structure per each choice of an invariant bilinear form on ${\rm G}$ \cite{Goldman}, hence the problem of quantization along this symplectic structure arises. Here we focus on the algebraic aspect of quantization, for which one would like to deform the algebra of regular functions $\mathcal{O}(\mathscr{L}_{{\rm G},\fS})$ by a one-parameter family of non-commutative algebras, in an appropriate sense. When ${\rm G} = {\rm SL}_n$, $n\ge 2$, it is well-known \red{\cite{Procesi}} that $\mathcal{O}(\mathscr{L}_{{\rm SL}_n,\fS})$ is generated by the {\em trace-of-monodromy} functions $f_\gamma$ associated to isotopy classes of oriented loops $\gamma$ in $\fS$; the value of $f_\gamma$ at the equivalence class of a group homomorphism $\rho : \pi_1(\fS) \to {\rm SL}_n$ is ${\rm trace}(\rho([\gamma]))$, where $[\gamma]$ is the element of $\pi_1(\fS)$ represented by $\gamma$. 
When $n\ge 3$, a complete set of algebraic relations among $f_\gamma$ is rather complicated \red{(\cite{Procesi})}, and becomes much simpler when one also takes into consideration more functions $f_W$, each of which is associated to a combinatorial object $W$ called a {\em web} or an {\em ${\rm SL}_n$-web}; we call it an {\em $n$-web}. Here we consider a model developed by Sikora \cite{Sikora01}, in which an $n$-web is a union of oriented loops and oriented $n$-valent graphs in $\fS$, where each vertex of the graph is required to be either a source or a sink. See \cite{CKM14} for another model of Cautis, Kamnitzer and Morrison, based on 3-valent graphs with edges labeled by fundamental representations of ${\rm SL}_n$, and \cite{Poudel} for the equivalence of the two models. 

A quantum deformed algebra for $\mathcal{O}(\mathscr{L}_{{\rm SL}_n,\fS})$ is constructed with the help of 3-dimensional quantum topology. The {\em ${\rm SL}_n$-skein algebra} $\mathscr{S}_\omega(\fS) = \mathscr{S}_\omega^{{\rm SL}_n}(\fS)$ is an associative algebra generated by the isotopy classes of framed $n$-webs living in the thickened surface $\fS \times (-1,1)$, modulo the {\it skein relations} 
\eqref{w.cross}--\eqref{wzh.four}, where the product of two framed $n$-webs is given by superposition, i.e. stacking one on top of the other; see \S\ref{subsec.reduced_stated_SLn-skein_algebras}. Here $\omega$ appearing in the relations is a formal quantum parameter, which is related to another quantum parameter $q$ by
$$
q = \omega^{n^2}.
$$
There are various versions of this algebra; the case when $n=2$ is first studied by Turaev \cite{Turaev} and Przytycki \cite{P99}, and for general $n$ by Sikora \cite{Sikora05} and Cautis-Kamnitzer-Morrison \cite{CKM14}, building on the work of Kuperberg \cite{Kuperberg} on invariants of tensor product representations of Lie algebras and quantum groups.

When $n=2$, the ${\rm SL}_2$-character stack $\mathscr{L}_{{\rm SL}_2,\fS}$ is closely related to the Teichm\"uller space of $\fS$, which is the space of isotopy classes of complex structures (or complete hyperbolic metrics) on $\fS$. Quantization of various versions of Teichm\"uller spaces has been studied since late 1980's, where the first batch of fundamental results were obtained by Kashaev \cite{Kashaev} and by Chekhov and Fock \cite{F97,CF99}; the latter was generalized by Fock and Goncharov \cite{FG09a,FG09b} to quantization of cluster Poisson varieties, which include ordinary and higher Teichm\"uller spaces of $\fS$ \cite{FG06}. To say briefly, Fock and Goncharov studied a certain stack $\mathscr{X}_{{\rm PGL}_n,\fS}$ that parametrizes ${\rm PGL}_n$-local systems on $\fS$ together with certain extra discrete data at punctures \cite{FG06}. Per each choice of an {\em ideal triangulation} $\lambda$ of $\fS$, which is a collection of unoriented paths in $\fS$ running between punctures that divide $\fS$ into (ideal) triangles, considered up to isotopy, Fock and Goncharov constructed in \cite{FG06} a special coordinate system on $\mathscr{X}_{{\rm PGL}_n,\fS}$. The coordinates are parametrized by the vertices of a special quiver $\Gamma_\lambda$ obtained by gluing the `$n$-triangulation' of each ideal triangle of $\lambda$ (see \S\ref{subsec.quantum_trace_maps} and Figure \ref{Fig;coord_ijk}). It is also shown that $\mathscr{X}_{{\rm PGL}_n,\fS}$ possesses a natural Poisson structure whose Poisson brackets among these special coordinates are encoded by the signed adjacency matrix $Q_\lambda$ of the quiver $\Gamma_\lambda$ which counts the number of arrows between vertices, and that the transition map between the special charts for different ideal triangulations $\lambda$ and $\lambda'$ are given by a specific composition of the formulas called the {\em cluster mutations}. These charts, called the {\em cluster} charts, together with the transition maps, are quantized by certain non-commutative algebras and maps between them. To each cluster chart for $\lambda$ is associated the quantum torus algebra $\mathcal{X}_q(\fS,\lambda) = \mathcal{X}_q^{{\rm PGL}_n}(\fS,\lambda)$ constructed from $Q_\lambda$, which is a free associative algebra generated by the symbols $X_v$ and their inverses $X_v^{-1}$, with $v$ being the vertices of the quiver $\Gamma_\lambda$, modulo the relations $X_v X_w = q^{2Q_\lambda(v,w)} X_w X_v$. Per each pair of triangulations $\lambda$ and $\lambda'$, 
Fock and Goncharov constructed an isomorphism \cite{FG09a,FG09b}
\begin{align}
    \label{intro.Phi_q}
    \Phi^q_{\lambda\lambda'} : {\rm Frac}(\mathcal{X}_q(\fS,\lambda')) \to {\rm Frac}(\mathcal{X}_q(\fS,\lambda))
\end{align}
between the skew-fields of fractions of the quantum torus algebras, as a composition of a special sequence of quantum cluster mutations; 
see \cite{BZ,FG09a,FG09b} for the notion of a quantum cluster mutation, and \cite{FG06} for an explicit description of this sequence.  
These quantum coordinate change isomorphisms recover the classical transition maps as $q\to 1$, and they satisfy the consistency relation $\Phi^q_{\lambda \lambda'} \Phi^q_{\lambda' \lambda''} = \Phi^q_{\lambda \lambda''}$. One can interpret this as having a consistent system of quantum deformed algebras for the Poisson space $\mathscr{X}_{{\rm PGL}_n,\fS}$.

\vs

For $n=2$, the ${\rm SL}_2$-skein algebra $\mathscr{S}_\omega^{{\rm SL}_2}(\fS)$\red{,} which quantizes the space $\mathscr{L}_{{\rm SL}_2,\fS}$\red{,} and the above system of quantum torus algebras $\mathcal{X}^{{\rm PGL}_2}_q(\fS,\lambda)$\red{,} which quantize the space $\mathscr{X}_{{\rm PGL}_2,\fS}$\red{,} were expected to be closely related to each other. It was Bonahon and Wong \cite{BW11} who proved this expectation. One important subtlety is that one needs to consider a square-root version of the quantum torus algebra, $\mathcal{Z}^{{\rm PGL}_2}_\omega(\fS,\lambda)$, generated by $Z_v^{\pm 1}$ for vertices $v$ of $\Gamma_\lambda$, with relations
$Z_v Z_w = \omega^{2 Q_\lambda(v,w)} Z_w Z_v$, in which $\mathcal{X}^{{\rm PGL}_2}_q(\fS,\lambda)$ is embedded as $X_v = Z_v^2$ and $q = \omega^4$. Note that, for $n=2$, the vertices of $\Gamma_\lambda$ is in bijection with the edges of the ideal triangulation $\lambda$. A Laurent monomial $\prod_{v \in \lambda} Z_v^{a_v} \in \mathcal{Z}^{{\rm PGL}_2}_\omega(\fS,\lambda)$, for $a_v \in \mathbb{Z}$, where the product is taken with respect to some chosen order, is said to be {\em ${\rm SL}_2$-balanced} if for each ideal triangle of $\lambda$, the sum of $a_v$ for the edges $v$ forming this triangle is even. Let $\mathcal{Z}^{{\rm bl}; {\rm PGL}_2}_\omega(\fS,\lambda)$ be the subalgebra of $\mathcal{Z}^{{\rm PGL}_2}_\omega(\fS,\lambda)$ generated by the ${\rm SL}_2$-balanced Laurent monomials; it is easy to see that this balanced square-root subalgebra $\mathcal{Z}^{{\rm bl};{\rm PGL}_2}_\omega(\fS,\lambda)$ contains $\mathcal{X}_q^{{\rm PGL}_2}(\fS,\lambda)$. Hiatt \cite{Hiatt,BW11,KLS} constructed an isomorphism
$$
    \Theta^{\omega;{\rm PGL}_2}_{\lambda\lambda'} : {\rm Frac}(\mathcal{Z}^{{\rm bl};{\rm PGL}_2}_\omega(\fS,\lambda')) \to {\rm Frac}(\mathcal{Z}^{{\rm bl};{\rm PGL}_2}_\omega(\fS,\lambda))
$$
between the skew-fields of fractions of the ${\rm SL}_2$-balanced square-root quantum torus algebras, that extends the previously known $\Phi^q_{\lambda\lambda'}$ which is a composition of quantum cluster mutations, and that satisfies the consistency relation $\Theta^{\omega;{\rm PGL}_2}_{\lambda\lambda'}\Theta^{\omega;{\rm PGL}_2}_{\lambda'\lambda''}=\Theta^{\omega;{\rm PGL}_2}_{\lambda\lambda''}$. Bonahon and Wong \cite{BW11} constructed an algebra homomorphism
\begin{align}
    \label{intro.eq1}
{\rm tr}_\lambda^{{\rm SL}_2} : \mathscr{S}_\omega^{{\rm SL}_2}(\fS) \to \mathcal{Z}^{{\rm bl};{\rm PGL}_2}_\omega(\fS,\lambda)
\end{align}
called the {\em ${\rm SL}_2$-quantum trace}, per each ideal triangulation $\lambda$ of $\fS$, that is compatible with the balanced square-root quantum coordinate change maps $\Theta^{\omega;{\rm PGL}_2}_{\lambda\lambda'}$, i.e.
\begin{align}
    \label{intro.eq2}
    {\rm tr}^{{\rm SL}_2}_\lambda = \Theta^{\omega;{\rm PGL}_2}_{\lambda\lambda'} \circ {\rm tr}^{{\rm SL}_2}_{\lambda'};
\end{align}
this compatibility can be interpreted as the {\em naturality} under change of ideal triangulations. The Bonahon-Wong ${\rm SL}_2$-quantum trace has been widely used ever since, in several important applications including the construction of Fock-Goncharov quantum duality map \cite{AK}, and the study of representations of ${\rm SL}_2$-skein algebras \cite{BW16}.

\vs

One major direction of research after Bonahon and Wong's work has been on the ${\rm SL}_n$-generalization, i.e. on the {\em ${\rm SL}_n$-quantum trace maps}
\begin{align}
    \label{intro.tr_SLn}
    {\rm tr}^{{\rm SL}_n}_\lambda : \mathscr{S}^{{\rm SL}_n}_\omega(\fS) \to \mathcal{Z}^{{\rm bl};{\rm PGL}_n}_\omega(\fS,\lambda) \subset \mathcal{Z}^{{\rm PGL}_n}_\omega(\fS,\lambda).
\end{align}
Here $\mathcal{Z}^{{\rm PGL}_n}_\omega(\fS,\lambda)$ is the $n$-th root version of the quantum torus algebra $\mathcal{X}^{{\rm PGL}_n}_q(\fS,\lambda)$, generated by $Z_v^{\pm 1}$ for vertices $v$ of the $n$-triangulation quiver $\Gamma_\lambda$ for an ideal triangulation $\lambda$, modulo the relation $Z_v Z_w = \omega^{2Q_\lambda(v,w)} Z_w Z_v$. The usual Fock-Goncharov quantum (cluster) torus algebra $\mathcal{X}^{{\rm PGL}_n}_q(\fS,\lambda)$ embeds into it as $X_v = Z_v^n$, $q = \omega^{n^2}$. The ${\rm SL}_n$-balanced subalgebra $\mathcal{Z}^{{\rm bl};{\rm PGL}_n}_\omega(\fS,\lambda)$ of $\mathcal{Z}^{{\rm PGL}_2}_\omega(\fS,\lambda)$ is generated by $\prod_v Z_v^{a_v}$ for $(a_v)_v \in \mathbb{Z}^{\mbox{\tiny \red{$\{$}vertices of $\Gamma_\lambda\red{\}}$}}$ that satisfies a suitable ${\rm SL}_n$-analog of the ${\rm SL}_2$-balancedness condition, so that in particular we have $\mathcal{X}_q^{{\rm PGL}_n}(\fS,\lambda) \subset \mathcal{Z}^{{\rm bl};{\rm PGL}_n}_\omega(\fS,\lambda)$. The first developments were for the case $n=3$ studied by Douglas \cite{Douglas} and the first author \cite{Kim20}, and later L\^e and Yu constructed ${\rm SL}_n$-quantum trace maps for all $n \ge 2$ \cite{LY23}.

\vs

One of the most important sought-for properties of ${\rm tr}_\lambda^{{\rm SL}_n}$ is the naturality under change of ideal triangulations. The naturality should fit into the framework of quantum cluster algebras, in the sense that ${\rm tr}_\lambda^{{\rm SL}_n}$ for different ideal triangulations must be compatible with each other via a balanced $n$-th root quantum coordinate change isomorphism
\begin{align}
\label{intro.Theta_SLn}
\Theta^{\omega;{\rm SL}_n}_{\lambda\lambda'} : {\rm Frac}(\mathcal{Z}^{{\rm bl};{\rm PGL}_n}_\omega(\fS,\lambda')) \to {\rm Frac}(\mathcal{Z}^{{\rm bl};{\rm PGL}_n}_\omega(\fS,\lambda))
\end{align}
that extends Fock and Goncharov's map $\Phi^q_{\lambda\lambda'}$ in \eqref{intro.Phi_q} which is a composition of 
quantum cluster mutations. So one should construct such a map $\Theta^{\omega;{\rm SL}_n}_{\lambda\lambda'}$ and show the naturality
\begin{align}
\label{intro.SLn_naturality}
{\rm tr}_\lambda^{{\rm SL}_n} = \Theta^{\omega;{\rm SL}_n}_{\lambda\lambda'} \circ {\rm tr}^{{\rm SL}_n}_{\lambda'}.
\end{align}
This task is accomplished for $n=3$ by the first author in \cite{Kim21}. For general $n$, L\^e and Yu have a partial result in \cite{LY23}, where they construct a skew-field isomorphism
$$
\overline{\Psi}^X_{\lambda\lambda'} : {\rm Frac}(\mathcal{Z}^{{\rm bl};{\rm PGL}_n}_\omega(\fS,\lambda')) \to {\rm Frac}(\mathcal{Z}^{{\rm bl};{\rm PGL}_n}_\omega(\fS,\lambda))
$$
and show   the naturality
$$
{\rm tr}_\lambda^{{\rm SL}_n} = \overline{\Psi}^X_{\lambda\lambda'} \circ {\rm tr}^{{\rm SL}_n}_{\lambda'}.
$$
However, L\^e and Yu's isomorphism $\overline{\Psi}^X_{\lambda\lambda'}$ is essentially derived from the naturality equation itself, while its relationship with the quantum cluster isomorphism $\Phi^q_{\lambda\lambda'}$ is not studied, but only suggested \cite[\S14.4]{LY23}.

\vs

Here is the main result of the current paper\red{, stated for punctured surfaces so as to align with the focus of this subsection. As we will explain in \S\ref{subsec.main_ideas}, Theorem~\ref{thm.main.intro} is in fact proved in a more general setting.}

\begin{theorem}[main theorem]
\label{thm.main.intro}
Let $\fS$ be a punctured surface admitting an ideal triangulation. 

\begin{enumerate}[label={\rm (\arabic*)}]
    \item (\S\ref{subsec:coordinate_change_isomorphisms_for_change_of_triangulations}\red{, and \S\ref{subsec:restriction_of_Theta_to_balanced}}) For any two ideal triangulations $\lambda$ and $\lambda'$ of $\fS$, there exists a balanced $n$-th root quantum coordinate change map $\Theta^{\omega;{\rm SL}_n}_{\lambda\lambda'}$ as in \eqref{intro.Theta_SLn}, which is a skew-field isomorphism that extends Fock and Goncharov's map $\Phi^q_{\lambda\lambda'}$ in \eqref{intro.Phi_q}, which in turn is constructed as a composition of 
    quantum cluster mutations.

    \item (Proposition \ref{prop:Theta_omega_consistency}) The consistency equation $\Theta^{\omega;{\rm SL}_n}_{\lambda\lambda'}\Theta^{\omega;{\rm SL}_n}_{\lambda'\lambda''} = \Theta^{\omega;{\rm SL}_n}_{\lambda\lambda''}$ holds for any three ideal triangulations $\lambda$, $\lambda'$ and $\lambda''$ of $\fS$.

    \item (Theorems \ref{thm-main-compatibility} and \ref{thm-main-compatibility_bl}) The naturality equation ${\rm tr}_\lambda^{{\rm SL}_n} = \Theta^{\omega;{\rm SL}_n}_{\lambda\lambda'} \circ {\rm tr}^{{\rm SL}_n}_{\lambda'}$ in \eqref{intro.SLn_naturality} holds.

    \item (Theorem \ref{thm-main-comparision}) The map $\Theta^{\omega;{\rm SL}_n}_{\lambda\lambda'}$ coincides with L\^e-Yu's map $\overline{\Psi}^X_{\lambda\lambda'}$.
\end{enumerate}

\end{theorem}

Consequently, L\^e and Yu's ${\rm SL}_n$-quantum trace \red{\cite{LY23}} is now proved to fit well in the framework of quantum cluster algebras. We expect that this comprises an important step towards an attempt of using the L\^e-Yu ${\rm SL}_n$-quantum trace to approach the Fock-Goncharov duality conjectures for gauge groups ${\rm SL}_n$ and ${\rm PGL}_n$, classical and quantum \cite{FG06,FG09a}. See \cite{AK} for $n=2$ and \cite{Kim20} for $n=3$. 

\red{We mention here that we first construct a `mutable-balanced' version of a general quantum cluster mutation (\S\ref{subsec:quantum_mutations_for_mutable-balanced_n-th_root_algebras}), then apply this to our specific setting to construct a mutable-balanced version of $\Theta^{\omega;\mathrm{SL}_n}_{\lambda\lambda'}$ (\S\ref{subsec:coordinate_change_isomorphisms_for_change_of_triangulations}), prove the main theorem for this version, and then show that the same result also holds for a `balanced' version of $\Theta^{\omega;\mathrm{SL}_n}_{\lambda\lambda'}$. We do not distinguish between mutable-balanced and balanced version in the introduction, to ease the discussion.}

\subsection{Main ideas}
\label{subsec.main_ideas}

A crucial strategy taken by Bonahon and Wong \cite{BW11} in their study of ${\rm SL}_2$-quantum trace maps is to investigate the compatibility under cutting the surface $\fS$ along an ideal arc, i.e. an edge of an ideal triangulation $\lambda$. This yields a more general kind of surface, called a {\em punctured bordered surface} or a {\em pb surface}, which is obtained from a compact oriented surface with possibly-empty boundary by removing a non-empty finite set of points called punctures, such that each boundary component of the resulting surface is diffeomorphic to an open interval. The {\em stated ${\rm SL}_2$-skein algebra} $\mathscr{S}^{{\rm SL}_2}_\omega(\fS)$ (\S\ref{subsec.reduced_stated_SLn-skein_algebras}) is defined using isotopy classes of framed links and stated framed tangles, which are arcs ending at the boundary walls of $\fS \times (-1,1)$, where the endpoints lying in each boundary wall are required to have distinct heights, i.e. distinct $(-1,1)$-coordinates, equipped with a {\em state} which assigns a value in $\{1,2\}$ to each endpoint. \red{One} mod\red{s} out by the skein relations and also by certain extra relations  regarding the boundary walls, studied by L\^e and collaborators \cite{Le18, CL22}. The Fock-Goncharov space $\mathscr{X}_{{\rm PGL}_2,\fS}$ has a counterpart space $\mathscr{P}_{{\rm PGL}_2,\fS}$ studied by Goncharov and Shen \cite{GS19}, equipped with cluster charts and corresponding quantum algebras, with quantum coordinate change isomorphisms. The statements \eqref{intro.eq1}--\eqref{intro.eq2} still hold for pb surfaces admitting ideal triangulations. Moreover, the compatibility under the cutting of a surface also holds in the following sense. Suppose $\lambda$ is a triangulation of a pb surface $\fS$, and let $\fS_e$ be the pb surface obtained from $\fS$ by cutting along an edge $e$ of $\lambda$ that does not lie in the boundary of $\fS$. Let $\lambda_e$ be the induced ideal triangulation of $\fS_e$. The cutting process yields a natural algebra embedding $\mathcal{Z}_\omega(\fS,\lambda) \to \mathcal{Z}_\omega(\fS_e,\lambda_e)$, which restricts to the map between balanced subalgebras. The cutting also yields a natural map between the stated ${\rm SL}_2$-skein algebras $\mathscr{S}^{{\rm SL}_2}_\omega(\fS) \to \mathscr{S}^{{\rm SL}_2}_\omega(\fS_e)$, called the {\it splitting homomorphism} which makes the following diagram commute \cite{Le18,CL22}
$$
\xymatrix{
\mathscr{S}^{{\rm SL}_2}_\omega(\fS) \ar[r] \ar[d] & \mathcal{Z}^{{\rm bl};{\rm PGL}_2}_\omega(\fS,\lambda) \ar[d] \\ \mathscr{S}^{{\rm SL}_2}_\omega(\fS_e) \ar[r] & \mathcal{Z}^{{\rm bl};{\rm PGL}_2}_\omega(\fS_e,\lambda_e),
}
$$
which more or less completely characterizes the ${\rm SL}_2$-quantum trace maps \cite{Le18}.

\vs

We take this strategy using cutting as well. We use the {\em reduced stated ${\rm SL}_n$-skein algebra} $\overline{\mathscr{S}}^{{\rm SL}_n}_\omega(\fS)$ of L\^e-Yu \cite{LY23}
(see \S\ref{subsec.reduced_stated_SLn-skein_algebras} of the present paper), defined using framed oriented links and stated framed $n$-webs in $\fS \times (-1,1)$ which are graphs having $1$-valent vertices lying in boundary walls equipped with state values in $\{1,2,\ldots,n\}$ and sink or source $n$-valent vertices lying in the interior, with the skein relations \eqref{w.cross}--\eqref{wzh.four} and boundary relations \eqref{wzh.five}--\eqref{wzh.eight}; here `reduced' means that we quotient out a certain ideal. In \cite{LY23,LS21} the splitting homomorphism of these algebras are studied, and the L\^e-Yu ${\rm SL}_n$-quantum trace is proved to be compatible with the splitting homomorphisms \cite{LY23}. 

\vs

It is well known that any two ideal triangulations are connected by a sequence of elementary moves called {\em flips}, where a flip changes exactly one edge of an ideal triangulation. If $\lambda$ and $\lambda'$ is related by flip at an edge $e$, we focus on the two ideal triangles in $\lambda$ adjacent to $e$, and those in $\lambda'$ adjacent to $e$. Using the splitting homomorphisms and the corresponding compatibility, through cutting along all edges not equal to $e$, the problem of proving the sought-for naturality \eqref{intro.SLn_naturality} boils down to proving it in the case when the surface $\fS$ is a 4-gon, i.e. a closed disc with 4 punctures on the boundary, which admits exactly two ideal triangulations. However, as seen in \cite{Kim21}, proving \eqref{intro.SLn_naturality} for 4-gon directly would still require a heavy computation, already for $n=3$. In particular, a flip involves $\frac{1}{6}(n^3-3)$ mutations \cite{FG06} (see our \S\ref{subsec:coordinate_change_isomorphisms_for_change_of_triangulations}).

\vs

To bypass the difficult computation, we take advantage of another development. L\^e and Yu showed in \cite[Theorem 10.5]{LY23} that the value under their ${\rm SL}_n$-quantum trace ${\rm tr}^{{\rm SL}_n}_\lambda$ of a stated arc, which has only 1-valent vertices but no $n$-valent vertices, in the case when the surface $\fS$ is a triangle, a 3-gon, can be expressed as a sum of Laurent monomials enumerated by paths in a certain directed network (a directed graph) given by dual of the $n$-triangulation. We note that this network is studied first by Schrader and Shapiro \cite{SS17} (see also \cite{CS23}), who also studied a certain compatibility of this sum under a single `mutation' of the network, which in turn is directly related to the cluster mutation of the $n$-triangulation quiver $\Gamma_\lambda$ of Fock and Goncharov (and Shen). We extend L\^e and Yu's result to 4-gon (Lemma \ref{lem.quantum_trace_as_sum_over_paths}), and use Schrader and Shapiro's compatibility repeatedly to show the sought-for naturality \eqref{intro.SLn_naturality} (Proposition \ref{prop-P4-compatibility_new}).

\vs

As a result, we show Theorem \ref{thm.main.intro} for the reduced stated ${\rm SL}_n$-skein algebras for all triangulable pb surfaces. This in particular recovers the results of \cite{Kim21} by a more conceptual proof.

\vs

An interesting future research topic is on the generalization of this non-computational proof of the naturality for more general stated $n$-webs living in the thickened 4-gon, not just stated arcs. It will become an important step in studying the sought-for quantum Fock-Goncharov duality map for ${\rm SL}_n$ and ${\rm PGL}_n$\red{.} 

\vs

{\bf Acknowledgments.} H.K. (the first author) has been supported by KIAS Individual Grant (MG047204) at Korea Institute for Advanced Study. Z.W.  (the second author)
\red{has been supported by a KIAS Individual Grant (MG104701) at the Korea Institute for Advanced Study. }
The authors thank Thang L\^e for helpful discussions. The authors thank the anonymous referee for helpful comments.

\section{Quantum trace maps for stated $\SL$-skein algebras}

\red{In this section we review the stated $\mathrm{SL}_n$-skein algebras of a pb surface, and its reduced version. We review the splitting homomorphisms between these algebras for different pb surfaces related to each other by cutting and gluing. We finally review the the quantum trace map from the skein algebra going to a quantum torus algebra. We mostly follow \cite{LS21} and \cite{LY23}.}

\red{We fix notations for some parameters.} Suppose that $R$ is a commutative domain with an invertible element
$\omega^{\frac{1}{2}}$. Set
$$
q = \omega^{n^2}
$$
with
$q^{\frac{1}{2n^2}} = \omega^{\frac{1}{2}}$.  
Define the following constants:
\begin{align*}
\mathbbm{c}_{i}= (-q)^{n-i} q^{\frac{n-1}{2n}},\quad
\mathbbm{t}= (-1)^{n-1} q^{\frac{n^2-1}{n}},\quad 
\mathbbm{a} =   q^{\frac{n+1-2n^2}{4}}.
\end{align*}

\subsection{Reduced stated $\SL$-skein algebras}
\label{subsec.reduced_stated_SLn-skein_algebras}

\def\Si{\fS}

A {\bf punctured bordered surface} (or {\bf pb surface} for simplicity) $\fS$ is obtained from a compact oriented surface $\overline{\Si}$ with possibly-empty boundary by removing finitely many points, which are called punctures, such that every boundary component of $\fS$ is diffeomorphic to an open interval. 

For any positive integer $k$, we use $\mathbb P_k$ to denote the pb surface obtained from the closed disk by removing $k$ punctures from the boundary. 
We refer to $\mathbb{P}_k$ as the (ideal) $k$-gon. In particular, $\mathbb{P}_1$, $\mathbb{P}_2$ and $\mathbb{P}_3$ will be called the {\bf monogon}, the {\bf bigon} and the {\bf triangle}.

Consider the 3-manifold $\Si \times (-1,1)$, the thickened surface. For a point $(x,t) \in \Si \times (-1,1)$, the value $t$ is called the {\bf height} of this point. An {\bf $n$-web} $\alpha$ in $\Si\times(-1,1)$ is a disjoint union of oriented closed curves and a directed finite graph properly embedded in $\Si\times(-1,1)$, satisfying the following requirements:
\begin{enumerate}
    \item $\alpha$ only contains $1$-valent or $n$-valent vertices. Each $n$-valent vertex is a source or a  sink. The set of $1$-valent vertices is denoted as $\partial \alpha$, which are called \textbf{endpoints} of $\alpha$. For any boundary component $c$ of $\Si$, we require that the points of  $\partial\alpha\cap (c\times(-1,1))$ have mutually distinct heights.
    \item Every edge of the graph is an embedded oriented  closed interval  in $\Si\times(-1,1)$.
    \item $\alpha$ is equipped with a continuous transversal \textbf{framing}. 
    \item The set of half-edges at each $n$-valent vertex is equipped with a  cyclic order. 
    \item $\partial \alpha$ is contained in $\partial\Si\times (-1,1)$ and the framing at these endpoints is given by the positive direction of $(-1,1)$.
\end{enumerate}
We will consider $n$-webs up to (ambient) \textbf{isotopy} which are continuous deformations of $n$-webs in their class. 
The empty $n$-web, denoted by $\emptyset$, is also considered as an $n$-web, with the convention that $\emptyset$ is only isotopic to itself. 

A {\bf state} for $\alpha$ is a map $s\colon\partial\alpha\rightarrow \{1,2,\cdots,n\}$. A {\bf stated $n$-web} in $\Si\times(-1,1)$ is an $n$-web equipped with a state.

We say the (stated) $n$-web $\alpha$ is in {\bf vertical position} if 
\begin{enumerate}
    \item the framing at everywhere is given by the positive direction of $(-1,1)$,
    \item $\alpha$ is in general position with respect to the projection  $\text{pr}\colon \Si\times(-1,1)\rightarrow \Si\times\{0\}$,
    \item at every $n$-valent vertex, the cyclic order of half-edges as the image of $\text{pr}$ is given by the positive orientation of $\Si$ (drawn counter-clockwise in pictures).
\end{enumerate}

By regarding $\fS$ as $\fS\times\{0\}$, there is a projection $\text{pr}\colon\fS\times(-1,1)\rightarrow \fS.$
For every (stated) $n$-web $\alpha$, we can isotope $\alpha$ to be in vertical position. For each boundary component $c$ of $\fS$, the heights of $\partial\alpha\cap (c\times(-1,1))$ determine a linear order on  $c\cap \text{pr}(\alpha)$.
Then a {\bf (stated) $n$-web diagram} of $\alpha$ is $\text{pr}(\alpha)$ equipped with the usual over/underpassing information at each double point (called a crossing) and a linear order on $c\cap \text{pr}(\alpha)$ for each boundary component $c$ of $\fS$.

Let $S_n$ denote the permutation group on the set $\{1,2,\cdots,n\}$.

\def\M {M,\cN}

The \textbf{stated $\SL$-skein algebra} $\cS_{\omega}(\fS)$ of $\fS$ is
the quotient module of the $R$-module freely generated by the set 
 of all isotopy classes of stated 
$n$-webs in $\fS\times (-1,1)$ subject to the relations \eqref{w.cross}-\eqref{wzh.eight}:

\beq\label{w.cross}
q^{\frac{1}{n}} 
\raisebox{-.20in}{

\begin{tikzpicture}
\tikzset{->-/.style=

{decoration={markings,mark=at position #1 with

{\arrow{latex}}},postaction={decorate}}}
\filldraw[draw=white,fill=gray!20] (-0,-0.2) rectangle (1, 1.2);
\draw [line width =1pt,decoration={markings, mark=at position 0.5 with {\arrow{>}}},postaction={decorate}](0.6,0.6)--(1,1);
\draw [line width =1pt,decoration={markings, mark=at position 0.5 with {\arrow{>}}},postaction={decorate}](0.6,0.4)--(1,0);
\draw[line width =1pt] (0,0)--(0.4,0.4);
\draw[line width =1pt] (0,1)--(0.4,0.6);
\draw[line width =1pt] (0.4,0.6)--(0.6,0.4);
\end{tikzpicture}
}
- q^{-\frac {1}{n}}
\raisebox{-.20in}{
\begin{tikzpicture}
\tikzset{->-/.style=

{decoration={markings,mark=at position #1 with

{\arrow{latex}}},postaction={decorate}}}
\filldraw[draw=white,fill=gray!20] (-0,-0.2) rectangle (1, 1.2);
\draw [line width =1pt,decoration={markings, mark=at position 0.5 with {\arrow{>}}},postaction={decorate}](0.6,0.6)--(1,1);
\draw [line width =1pt,decoration={markings, mark=at position 0.5 with {\arrow{>}}},postaction={decorate}](0.6,0.4)--(1,0);
\draw[line width =1pt] (0,0)--(0.4,0.4);
\draw[line width =1pt] (0,1)--(0.4,0.6);
\draw[line width =1pt] (0.6,0.6)--(0.4,0.4);
\end{tikzpicture}
}
= (q-q^{-1})
\raisebox{-.20in}{

\begin{tikzpicture}
\tikzset{->-/.style=

{decoration={markings,mark=at position #1 with

{\arrow{latex}}},postaction={decorate}}}
\filldraw[draw=white,fill=gray!20] (-0,-0.2) rectangle (1, 1.2);
\draw [line width =1pt,decoration={markings, mark=at position 0.5 with {\arrow{>}}},postaction={decorate}](0,0.8)--(1,0.8);
\draw [line width =1pt,decoration={markings, mark=at position 0.5 with {\arrow{>}}},postaction={decorate}](0,0.2)--(1,0.2);
\end{tikzpicture}
},
\eeq 
\beq\label{w.twist}
\raisebox{-.15in}{
\begin{tikzpicture}
\tikzset{->-/.style=
{decoration={markings,mark=at position #1 with
{\arrow{latex}}},postaction={decorate}}}
\filldraw[draw=white,fill=gray!20] (-1,-0.35) rectangle (0.6, 0.65);
\draw [line width =1pt,decoration={markings, mark=at position 0.5 with {\arrow{>}}},postaction={decorate}](-1,0)--(-0.25,0);
\draw [color = black, line width =1pt](0,0)--(0.6,0);
\draw [color = black, line width =1pt] (0.166 ,0.08) arc (-37:270:0.2);
\end{tikzpicture}}
= \mathbbm{t}
\raisebox{-.15in}{
\begin{tikzpicture}
\tikzset{->-/.style=
{decoration={markings,mark=at position #1 with
{\arrow{latex}}},postaction={decorate}}}
\filldraw[draw=white,fill=gray!20] (-1,-0.5) rectangle (0.6, 0.5);
\draw [line width =1pt,decoration={markings, mark=at position 0.5 with {\arrow{>}}},postaction={decorate}](-1,0)--(-0.25,0);
\draw [color = black, line width =1pt](-0.25,0)--(0.6,0);
\end{tikzpicture}}
,  
\eeq
\beq\label{w.unknot}
\raisebox{-.20in}{
\begin{tikzpicture}
\tikzset{->-/.style=
{decoration={markings,mark=at position #1 with
{\arrow{latex}}},postaction={decorate}}}
\filldraw[draw=white,fill=gray!20] (0,0) rectangle (1,1);
\draw [line width =1pt,decoration={markings, mark=at position 0.5 with {\arrow{>}}},postaction={decorate}](0.45,0.8)--(0.55,0.8);
\draw[line width =1pt] (0.5 ,0.5) circle (0.3);
\end{tikzpicture}}
= (-1)^{n-1} [n]\ 
\raisebox{-.20in}{
\begin{tikzpicture}
\tikzset{->-/.style=
{decoration={markings,mark=at position #1 with
{\arrow{latex}}},postaction={decorate}}}
\filldraw[draw=white,fill=gray!20] (0,0) rectangle (1,1);
\end{tikzpicture}}
,\ \text{where}\ [n]={\textstyle \frac{q^n-q^{-n}}{q-q^{-1}}}=q^{-n+1}+q^{-n+3}+\cdots+q^{n-1},
\eeq
\beq\label{wzh.four}
\raisebox{-.30in}{
\begin{tikzpicture}
\tikzset{->-/.style=
{decoration={markings,mark=at position #1 with
{\arrow{latex}}},postaction={decorate}}}
\filldraw[draw=white,fill=gray!20] (-1,-0.7) rectangle (1.2,1.3);
\draw [line width =1pt,decoration={markings, mark=at position 0.5 with {\arrow{>}}},postaction={decorate}](-1,1)--(0,0);
\draw [line width =1pt,decoration={markings, mark=at position 0.5 with {\arrow{>}}},postaction={decorate}](-1,0)--(0,0);
\draw [line width =1pt,decoration={markings, mark=at position 0.5 with {\arrow{>}}},postaction={decorate}](-1,-0.4)--(0,0);
\draw [line width =1pt,decoration={markings, mark=at position 0.5 with {\arrow{<}}},postaction={decorate}](1.2,1)  --(0.2,0);
\draw [line width =1pt,decoration={markings, mark=at position 0.5 with {\arrow{<}}},postaction={decorate}](1.2,0)  --(0.2,0);
\draw [line width =1pt,decoration={markings, mark=at position 0.5 with {\arrow{<}}},postaction={decorate}](1.2,-0.4)--(0.2,0);
\node  at(-0.8,0.5) {$\vdots$};
\node  at(1,0.5) {$\vdots$};
\end{tikzpicture}}=(-q)^{\frac{n(n-1)}{2}}\cdot \sum_{\sigma\in S_n}
(-q^{\frac{1-n}n})^{\ell(\sigma)} \raisebox{-.30in}{
\begin{tikzpicture}
\tikzset{->-/.style=
{decoration={markings,mark=at position #1 with
{\arrow{latex}}},postaction={decorate}}}
\filldraw[draw=white,fill=gray!20] (-1,-0.7) rectangle (1.2,1.3);
\draw [line width =1pt,decoration={markings, mark=at position 0.5 with {\arrow{>}}},postaction={decorate}](-1,1)--(0,0);
\draw [line width =1pt,decoration={markings, mark=at position 0.5 with {\arrow{>}}},postaction={decorate}](-1,0)--(0,0);
\draw [line width =1pt,decoration={markings, mark=at position 0.5 with {\arrow{>}}},postaction={decorate}](-1,-0.4)--(0,0);
\draw [line width =1pt,decoration={markings, mark=at position 0.5 with {\arrow{<}}},postaction={decorate}](1.2,1)  --(0.2,0);
\draw [line width =1pt,decoration={markings, mark=at position 0.5 with {\arrow{<}}},postaction={decorate}](1.2,0)  --(0.2,0);
\draw [line width =1pt,decoration={markings, mark=at position 0.5 with {\arrow{<}}},postaction={decorate}](1.2,-0.4)--(0.2,0);
\node  at(-0.8,0.5) {$\vdots$};
\node  at(1,0.5) {$\vdots$};
\filldraw[draw=black,fill=gray!20,line width =1pt]  (0.1,0.3) ellipse (0.4 and 0.7);
\node  at(0.1,0.3){$\sigma_{+}$};
\end{tikzpicture}},
\eeq
where the ellipse enclosing $\sigma_+$  is the minimum crossing positive braid representing a permutation $\sigma\in S_n$ and $\ell(\sigma)=\#\{(i,j)\mid 1\leq i<j\leq n,\ \sigma(i)>\sigma(j)\}$ is the length of $\sigma\in S_n$.

\beq\label{wzh.five}
   \raisebox{-.30in}{
\begin{tikzpicture}
\tikzset{->-/.style=
{decoration={markings,mark=at position #1 with
{\arrow{latex}}},postaction={decorate}}}
\filldraw[draw=white,fill=gray!20] (-1,-0.7) rectangle (0.2,1.3);
\draw [line width =1pt](-1,1)--(0,0);
\draw [line width =1pt](-1,0)--(0,0);
\draw [line width =1pt](-1,-0.4)--(0,0);
\draw [line width =1.5pt](0.2,1.3)--(0.2,-0.7);
\node  at(-0.8,0.5) {$\vdots$};
\filldraw[fill=white,line width =0.8pt] (-0.5 ,0.5) circle (0.07);
\filldraw[fill=white,line width =0.8pt] (-0.5 ,0) circle (0.07);
\filldraw[fill=white,line width =0.8pt] (-0.5 ,-0.2) circle (0.07);
\end{tikzpicture}}
   = 
   \mathbbm{a} \sum_{\sigma \in S_n} (-q)^{\ell(\sigma)}\,  \raisebox{-.30in}{
\begin{tikzpicture}
\tikzset{->-/.style=
{decoration={markings,mark=at position #1 with
{\arrow{latex}}},postaction={decorate}}}
\filldraw[draw=white,fill=gray!20] (-1,-0.7) rectangle (0.2,1.3);
\draw [line width =1pt](-1,1)--(0.2,1);
\draw [line width =1pt](-1,0)--(0.2,0);
\draw [line width =1pt](-1,-0.4)--(0.2,-0.4);
\draw [line width =1.5pt,decoration={markings, mark=at position 1 with {\arrow{>}}},postaction={decorate}](0.2,1.3)--(0.2,-0.7);
\node  at(-0.8,0.5) {$\vdots$};
\filldraw[fill=white,line width =0.8pt] (-0.5 ,1) circle (0.07);
\filldraw[fill=white,line width =0.8pt] (-0.5 ,0) circle (0.07);
\filldraw[fill=white,line width =0.8pt] (-0.5 ,-0.4) circle (0.07);
\node [right] at(0.2,1) {$\sigma(n)$};
\node [right] at(0.2,0) {$\sigma(2)$};
\node [right] at(0.2,-0.4){$\sigma(1)$};
\end{tikzpicture}},
\eeq
\beq \label{wzh.six}
\raisebox{-.20in}{
\begin{tikzpicture}
\tikzset{->-/.style=
{decoration={markings,mark=at position #1 with
{\arrow{latex}}},postaction={decorate}}}
\filldraw[draw=white,fill=gray!20] (-0.7,-0.7) rectangle (0,0.7);
\draw [line width =1.5pt,decoration={markings, mark=at position 1 with {\arrow{>}}},postaction={decorate}](0,0.7)--(0,-0.7);
\draw [color = black, line width =1pt] (0 ,0.3) arc (90:270:0.5 and 0.3);
\node [right]  at(0,0.3) {$i$};
\node [right] at(0,-0.3){$j$};
\filldraw[fill=white,line width =0.8pt] (-0.5 ,0) circle (0.07);
\end{tikzpicture}}   = \delta_{\bar j,i }\,  \mathbbm{c}_{i} \raisebox{-.20in}{
\begin{tikzpicture}
\tikzset{->-/.style=
{decoration={markings,mark=at position #1 with
{\arrow{latex}}},postaction={decorate}}}
\filldraw[draw=white,fill=gray!20] (-0.7,-0.7) rectangle (0,0.7);
\draw [line width =1.5pt](0,0.7)--(0,-0.7);
\end{tikzpicture}},
\eeq
\beq \label{wzh.seven}
\raisebox{-.20in}{
\begin{tikzpicture}
\tikzset{->-/.style=
{decoration={markings,mark=at position #1 with
{\arrow{latex}}},postaction={decorate}}}
\filldraw[draw=white,fill=gray!20] (-0.7,-0.7) rectangle (0,0.7);
\draw [line width =1.5pt](0,0.7)--(0,-0.7);
\draw [color = black, line width =1pt] (-0.7 ,-0.3) arc (-90:90:0.5 and 0.3);
\filldraw[fill=white,line width =0.8pt] (-0.55 ,0.26) circle (0.07);
\end{tikzpicture}}
= \sum_{i=1}^n  (\mathbbm{c}_{\bar i})^{-1}\, \raisebox{-.20in}{
\begin{tikzpicture}
\tikzset{->-/.style=
{decoration={markings,mark=at position #1 with
{\arrow{latex}}},postaction={decorate}}}
\filldraw[draw=white,fill=gray!20] (-0.7,-0.7) rectangle (0,0.7);
\draw [line width =1.5pt,decoration={markings, mark=at position 1 with {\arrow{>}}},postaction={decorate}](0,0.7)--(0,-0.7);
\draw [line width =1pt](-0.7,0.3)--(0,0.3);
\draw [line width =1pt](-0.7,-0.3)--(0,-0.3);
\filldraw[fill=white,line width =0.8pt] (-0.3 ,0.3) circle (0.07);
\filldraw[fill=black,line width =0.8pt] (-0.3 ,-0.3) circle (0.07);
\node [right]  at(0,0.3) {$i$};
\node [right]  at(0,-0.3) {$\bar{i}$};
\end{tikzpicture}} \quad \mbox{where $\bar{i} = n+1-i$,}
\eeq
\beq\label{wzh.eight}
\raisebox{-.20in}{

\begin{tikzpicture}
\tikzset{->-/.style=

{decoration={markings,mark=at position #1 with

{\arrow{latex}}},postaction={decorate}}}
\filldraw[draw=white,fill=gray!20] (-0,-0.2) rectangle (1, 1.2);
\draw [line width =1.5pt,decoration={markings, mark=at position 1 with {\arrow{>}}},postaction={decorate}](1,1.2)--(1,-0.2);
\draw [line width =1pt](0.6,0.6)--(1,1);
\draw [line width =1pt](0.6,0.4)--(1,0);
\draw[line width =1pt] (0,0)--(0.4,0.4);
\draw[line width =1pt] (0,1)--(0.4,0.6);
\draw[line width =1pt] (0.4,0.6)--(0.6,0.4);
\filldraw[fill=white,line width =0.8pt] (0.2 ,0.2) circle (0.07);
\filldraw[fill=white,line width =0.8pt] (0.2 ,0.8) circle (0.07);
\node [right]  at(1,1) {$i$};
\node [right]  at(1,0) {$j$};
\end{tikzpicture}
} =q^{-\frac{1}{n}}\left(\delta_{{j<i} }(q-q^{-1})\raisebox{-.20in}{

\begin{tikzpicture}
\tikzset{->-/.style=

{decoration={markings,mark=at position #1 with

{\arrow{latex}}},postaction={decorate}}}
\filldraw[draw=white,fill=gray!20] (-0,-0.2) rectangle (1, 1.2);
\draw [line width =1.5pt,decoration={markings, mark=at position 1 with {\arrow{>}}},postaction={decorate}](1,1.2)--(1,-0.2);
\draw [line width =1pt](0,0.8)--(1,0.8);
\draw [line width =1pt](0,0.2)--(1,0.2);
\filldraw[fill=white,line width =0.8pt] (0.2 ,0.8) circle (0.07);
\filldraw[fill=white,line width =0.8pt] (0.2 ,0.2) circle (0.07);
\node [right]  at(1,0.8) {$i$};
\node [right]  at(1,0.2) {$j$};
\end{tikzpicture}
}+q^{\delta_{i,j}}\raisebox{-.20in}{

\begin{tikzpicture}
\tikzset{->-/.style=

{decoration={markings,mark=at position #1 with

{\arrow{latex}}},postaction={decorate}}}
\filldraw[draw=white,fill=gray!20] (-0,-0.2) rectangle (1, 1.2);
\draw [line width =1.5pt,decoration={markings, mark=at position 1 with {\arrow{>}}},postaction={decorate}](1,1.2)--(1,-0.2);
\draw [line width =1pt](0,0.8)--(1,0.8);
\draw [line width =1pt](0,0.2)--(1,0.2);
\filldraw[fill=white,line width =0.8pt] (0.2 ,0.8) circle (0.07);
\filldraw[fill=white,line width =0.8pt] (0.2 ,0.2) circle (0.07);
\node [right]  at(1,0.8) {$j$};
\node [right]  at(1,0.2) {$i$};
\end{tikzpicture}
}\right),
\eeq
where   
$\delta_{j<i}= 
\begin{cases}
1  & j<i\\
0 & \text{otherwise}
\end{cases},\ 
\delta_{i,j}= 
\begin{cases} 
1  & i=j\\
0  & \text{otherwise}
\end{cases}$, 
and small white dots represent an arbitrary orientation of the edges (left-to-right or right-to-left), consistent for the entire equation. The black dot represents the opposite orientation. When a boundary edge of a shaded area is directed, the direction indicates the height order of the endpoints of the diagrams on that directed line, where going along the direction increases the height, and the involved endpoints are consecutive in the height order. The height order outside the drawn part can be arbitrary.

 \def \Sv{\cS_n(\Si,\mathbbm{v})}

 The algebra structure for $\cS_{\omega}(\fS)$ is given by stacking the stated $n$-webs, i.e. for any two stated $n$-webs $\alpha,\alpha'
 \subset\fS\times(-1,1)$, the product $\alpha\alpha'$ is defined by stacking $\alpha$ above $\alpha'$. That is, if $\alpha \subset \fS\times(0,1)$ and $\alpha' \subset \fS\times (-1,0)$, we have $\alpha \alpha' = \alpha \cup \alpha'$.

\red{Some historical remarks can be found in \S\ref{overview}, which we summarize here again. The first version studied in the literature is the case when $n=2$ and $\fS$ has no boundary, in which case $\cS_\omega(\fS)$ recovers the so-called Kauffman bracket skein algebra \cite{Turaev, P99}, which quantizes the $\mathrm{SL}_2$ character variety (or, stack) of $\fS$. Later, generalizations to $\mathrm{SL}_n$ were studied in \cite{Sikora01, Sikora05, CKM14}, partly based on \cite{Kuperberg}. The stated version when $\fS$ has non-empty boundary is first studied in \cite{Le18, CL22} for $n=2$, in \cite{Higgins} for $n=3$, and for \cite{LS21} in general. The internal relations \eqref{w.cross}--\eqref{wzh.four} can be viewed as relations among morphisms in the category of representations of the quantum group $\mathcal{U}_q(\fS)$, while the boundary relations \eqref{wzh.five}--\eqref{wzh.eight} are justified by the construction to be explained in the next subsection.}

\red{We now review a slightly modified version of the stated $\mathrm{SL}_n$-skein algebra that fits better for the purpose of our paper.} For a boundary puncture $p$ of a pb surface $\fS$, 
the {\bf corner arcs} $C(p)_{ij}$ and $\cev{C}(p)_{ij}$ are the stated arcs depicted as in Figure \ref{Fig;badarc}.
For a boundary puncture $p$ which is not on a monogon component of $\fS$, set 
$$C_p=\{C(p)_{ij}\mid i<j\},\quad\cev{C}_p=\{\cev{C}(p)_{ij}\mid i<j\}.$$  
Each element of $C_p\cup \cev{C}_p$ is called a \emph{bad arc} at $p$. 
\begin{figure}[h]
    \centering
    \includegraphics[width=150pt]{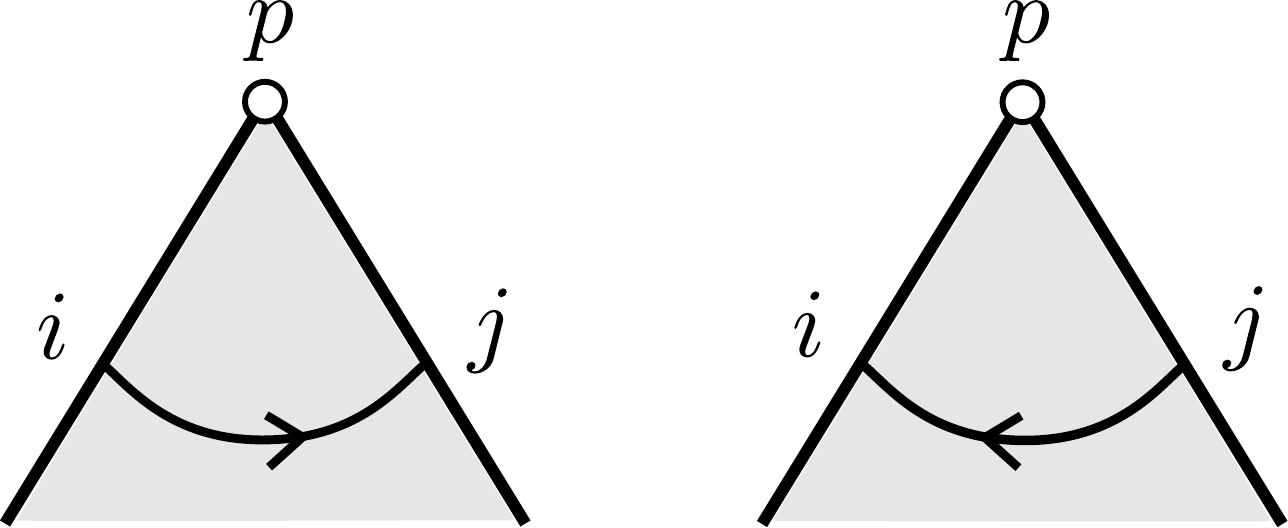}
    \caption{The left is $C(p)_{ij}$ and the right is $\cev{C}(p)_{ij}$.}\label{Fig;badarc}
\end{figure}

For a pb surface $\fS$, \begin{align}
    \label{eq:cS_red}
    \overline \cS_{\omega}(\fS) = \cS_{\omega}(\fS)/I^{\text{bad}}
\end{align}
is called the \textbf{reduced stated $\SL$-skein algebra}, defined in \cite{LY23}, where $I^{\text{bad}}$ is the two-sided ideal of $\cS_{\omega}(\fS)$ generated by all bad arcs.  \red{This algebra is what will appear in our main theorem later.}

\subsection{The splitting homomorphism}\label{sub-splitting}
\def\cut{\mathsf{Cut}}
\def\pr{{\bf pr}}

\red{One of the most crucial tools developed in the skein theory is the `splitting homomorphism' between stated $\mathrm{SL}_n$ skein algebras for surfaces related to each other by cutting and gluing, which we review here.} 

Let $\fS$ be a pb surface. An embedding $e:(0,1)\rightarrow \fS$ is called an {\bf ideal arc} if both
$\bar e(0)$ and $\bar e(1)$ are punctures, where $\bar e\colon [0,1] \to \overline{\fS}$ is the `closure' of $e$; we often consider an ideal arc $e$ up to isotopy, and often identify $e$ with its image.
Let $e$ be an 
ideal arc of a pb surface $\fS$ such that it is contained in the interior of $\fS$. After cutting $\fS$ along $e$, we get a new pb surface $\cut_e(\fS)$, which has two copies $e_1,e_2$ for $e$ such that 
${\fS}= \cut_e(\fS)/(e_1=e_2)$. We use $\pr_e$ to denote the projection from $\cut_e(\fS)$ to $\fS$.  Suppose that $\alpha$ is  a stated $n$-web diagram in $\fS$, which is transverse to $e$.
Let $s$ be a map from $e\cap\alpha$ to $\{1,2,\cdots,n\}$, and let $h$ be a linear order on $e\cap\alpha$. Then there is a lift diagram $\alpha(h,s)$ for a stated $n$-web in $\cut_e(\fS)$. 
 For $i=1,2$, the heights of the endpoints of $\alpha(h,s)$ on $e_i$ are induced by $h$ (via $\pr_e$), and the states of the endpoints of $\alpha(h,s)$ on $e_i$ are induced by $s$ (via $\pr_e$).
Then the {\bf splitting 
homomorphism} (see \cite{Le18, CL22} for $n=2$, \cite{Higgins} for $n=3$, and \cite{LS21} for general $n$)
$$
\mathbb{S}_e : \mathscr{S}_\omega(\fS) \to \mathscr{S}_\omega(\cut_e(\fS))
$$
is defined by 
\begin{align}\label{eq-def-splitting}
    \mathbb S_e(\alpha) =\sum_{s\colon \alpha \cap e \to \{1,\cdots, n\}} \alpha(h, s). 
\end{align}
Furthermore $\mathbb S_e$ is an $R$-algebra homomorphism.

\red{The splitting homomorphism, which is proven to be injective in many situtations, enables us to recover results about $\cS_\omega(\fS)$ from those about $\cS_\omega(\cut_e(\fS))$, where the latter is often easier to study. Note that the boundary relations \eqref{wzh.five}--\eqref{wzh.eight} are precisely the ones that guarantee the well-definedness of the splitting homomorphism.}

We can observe that $\mathbb S_e$ sends bad arcs to bad arcs.
So it induces an algebra homomorphism from $\rdS$
to $\rS(\cut_e(\fS))$ (see \cite{LY23}), which is still denoted as $\mathbb S_e$.
When there is no confusion we will omit the subscript $e$ 
from $\mathbb S_e$.

\subsection{Quantum trace maps}
\label{subsec.quantum_trace_maps}

In this subsection we review the materials in \cite{LY23} about the quantum trace map, \red{another pivotal result in the skein theory which in particular connects to the theory of quantum cluster 
varieties \cite{FG09a,FG09b}. It} is a homomorphism from \red{a version of} the stated $\SL$-skein algebra of a `triangulable' pb surface $\fS$ to a quantum torus algebra associated to a special quiver built from any choice of an `(ideal) triangulation' of $\fS$.

A pb surface $\fS$ is {\bf triangulable} if each component of $\fS$ has at least one puncture, and 
no component of 
$\fS$ is 
one the following cases:  
\begin{enumerate}[label={\rm (\arabic*)}]
    \item the monogon $\mathbb{P}_1$, 

    \item the bigon $\mathbb{P}_2$,

    \item once or twice punctured sphere.
\end{enumerate}

A {\bf triangulation} $\lambda$ of $\fS$ is a collection of disjoint ideal arcs in $\fS$ with the following properties: (1) any two arcs in $\lambda$ are not isotopic; (2) $\lambda$ is maximal under condition (1); (3) 
the valence of $\lambda$ at every puncture is at least $2$.
We will call each ideal arc in $\lambda$ an {\bf edge} of $\lambda$. 
If an edge is isotopic to a boundary component of $\fS$, we call such an edge a {\bf boundary edge}.
It is well known that any triangulable surface admits a triangulation. \red{We regard a triangulation up to simultaneous isotopy of the constituent edges, and we usually assume that each boundary edge lies in the boundary of $\fS$. By cutting $\fS$ along all non-boundary edges of $\lambda$ we have a disjoint union of ideal triangles. Each triangle is called a {\bf face} of $\lambda$, and we denote by $\mathbb{F}_\lambda$ the set of all faces of $\lambda$.}

\red{To describe the quantum torus algebra which is the codomain of the sought-for quantum trace map, we now explain a special quiver (\cite{FG06}) associated to any chosen triangulation $\lambda$ of $\fS$. As this quiver can be obtained by gluing quivers associated to the faces of $\lambda$, we first recall the quiver for a single triangle.}

\red{C}onsider barycentric coordinates for an ideal triangle $\bP_3$ so that
\begin{equation}
\bP_3=\{(i,j,k)\in\bR^3\mid i,j,k\geq 0,\ i+j+k=n\}\setminus\{(0,0,n),(0,n,0),(n,0,0)\}, 
\end{equation}
where $(i,j,k)$ (or $ijk$ for simplicity) are the barycentric coordinates. 
Let $v_1=(n,0,0)$, $v_2=(0,n,0)$, $v_3=(0,0,n)$. 
Let $e_i$ denote the edge on $\partial \bP_3$ whose endpoints are $v_i$ and $v_{i+1}$, where $v_4:=v_1$. See Figure \ref{Fig;coord_ijk}. 

\begin{figure}[h]
    \centering
    \includegraphics[width=260pt]{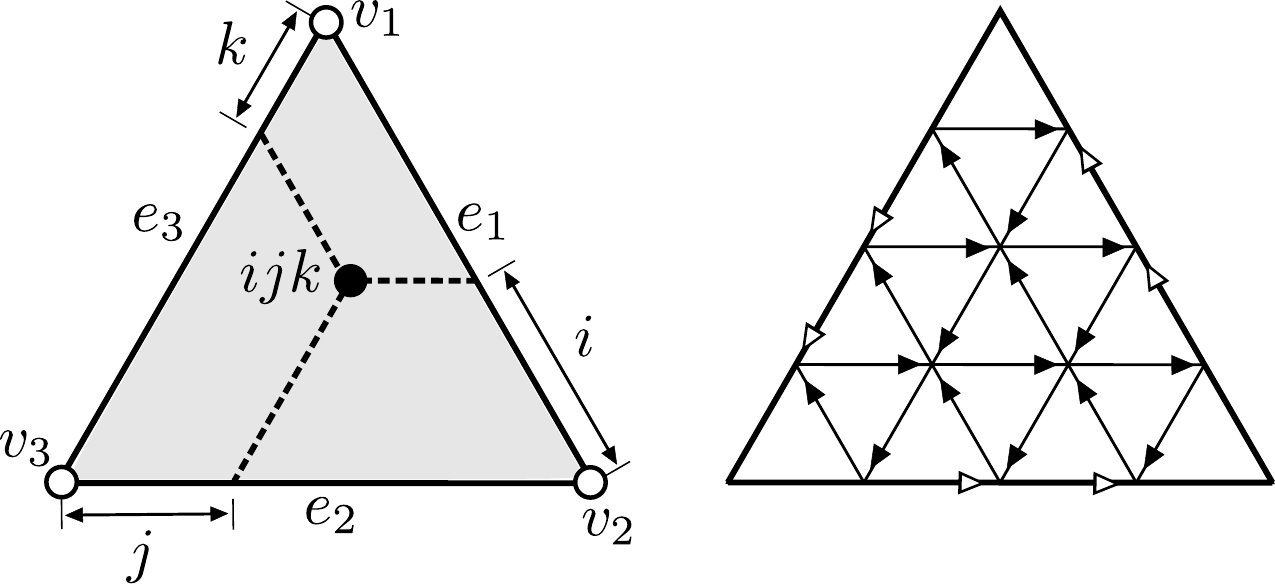}
    \caption{Barycentric coordinates $ijk$ and a $4$-triangulation with its quiver}\label{Fig;coord_ijk}
\end{figure}

The \textbf{$n$-triangulation} of $\bP_3$ is the subdivision of $\bP_3$ into $n^2$ small triangles using lines $i,j,k=\text{constant integers}$. 
For the $n$-triangulation of $\bP_3$, the vertices and edges of all small triangles, except for $v_1,v_2,v_3$ and the small edges adjacent to them, form a quiver $\Gamma_{\bP_3}$.
An \textbf{arrow} is 
a small edge 
equipped with an orientation prescribed as follows. If a small edge $e$ is in the boundary $\partial\bP_3$ then orientation of $e$ 
coincides with the counterclockwise direction of $\partial \bP_3$. If $e$ is 
not in the boundary of $\partial \mathbb{P}_3$ 
then its 
orientation is 
`parallel' to that of a boundary edge of $\bP_3$ parallel to $e$. Assign weight $\frac{1}{2}$ to any boundary arrow and weight $1$ to any interior arrow.

\def\bZ{\mathbb Z}

Let $V_{\bP_3}$ be the set of all
points with integer barycentric coordinates of $\bP_3$:
\begin{align}
V_{\bP_3} = \{ijk \in \bP_3 \mid i, j, k \in \bZ\}.
\end{align}
Elements of $V_{\bP_3}$ are called \textbf{small vertices}, and small vertices on the boundary of $\bP_3$ are called the \textbf{edge vertices}.

\red{Now} suppose that $\fS$ is a triangulable pb surface with a triangulation $\lambda$\red{; recall that $\mathbb{F}_\lambda$ is the set of all faces of $\lambda$}. Then
\begin{equation}\label{eq.glue}
\fS = \Big( \bigsqcup_{\tau\in\mathbb F_\lambda} \tau \Big) /\sim,
\end{equation}
where each face $\tau$ is regarded as a copy of $\bP_3$, and $\sim$ is the identification of edges of the faces to recover $\lambda$. 
Each face $\tau$ is characterized by a \textbf{characteristic map} $f_\tau\colon \mathbb P_3 \to \fS$, which is a homeomorphism when we restrict $f_\tau$ to $\Int\mathbb P_3$ or the interior of each edge of $\mathbb P_3$.

An \textbf{$n$-triangulation} of $\lambda$ is a collection of $n$-triangulations of the faces $\tau$ which are compatible with the gluing $\sim$,  where compatibility means, for any edges $b$ and $b'$ glued via $\sim$, the vertices on $b$ and $b'$ are identified. Define
$$V_\lambda=\bigcup_{\tau\in\mathbb F_\lambda} V_\tau, \quad V_\tau=f_\tau(V_{\mathbb P_3}).$$
The images of the weighted quivers $\Gamma_{\mathbb P_3}$ by $f_\tau$ form a quiver $\Gamma_\lambda$ on $\fS$ obtained by gluing. 
For any two $v,v'\in V_\lambda$, define
$$
a_\lambda(v,v') = \begin{cases} w \quad & \text{if there is an arrow from $v$ to $v'$ of weight $w$},\\
0 &\text{if there is no arrow between $v$ and $v'$.} 
\end{cases}$$
Let $Q_\lambda\colon V_\lambda\times V_\lambda \to \frac{1}{2}\bZ$ be the signed adjacency matrix of the weighted quiver $\Gamma_\lambda$ defined by 
\begin{equation}
\label{eq:Q_lambda}
Q_\lambda(v,v') = a_\lambda(v,v') - a_\lambda(v',v).
\end{equation}
What matters about the quiver $\Gamma_\lambda$ is only the signed adjacency matrix $Q_\lambda$, so we may tidy up the quiver by canceling or adding the weighted arrows in $\Gamma_\lambda$ while maintaining $Q_\lambda$\red{.} 

We use $Q_{\bP_3}$ to denote $Q_\lambda$ when $\fS=\bP_3$.  

\def\bT{\mathbb T}

\def\tr{\overline{\rm tr}_{\lambda}}
\def\X{\mathcal X_{\omega}(\fS,\lambda)}
\def\Y{\mathcal Y_{q}(\fS,\lambda)}

\def\bX{\mathcal Z_{\omega}^{\rm bl}(\fS,\lambda)}
\def\bk{{\bf k}}
\def\V{V_\lambda}

\red{In general, a quantum torus algebra can be constructed for any given skew-symmetric integer matrix.}  
The \textbf{Fock-Goncharov algebra} \red{\cite{FG09a,FG09b,FG06} associated to each triangulation $\lambda$ of $\fS$} is the quantum torus algebra associated to $Q_\lambda$, i.e.:
\begin{equation}
\mathcal{X}_{q}(\fS,\lambda)= \bT_q(Q_\lambda) = R \langle 
X_v^{\pm 1}, v \in V_\lambda \rangle / (
X_v 
X_{v'}= q^{\, 2 Q_\lambda(v,v')} 
X_{v'} 
X_v \text{ for } v,v'\in V_\lambda ),
\end{equation}
\red{which is the basic algebra in the theory of quantum cluster varieties, which will be dealt with in the next section in more generality.}

\red{Actually, the image of quantum trace map does not lie in $\mathcal{X}_q(\fS,\lambda)$, but in a slightly bigger algebra, namely} the \textbf{$n$-th root Fock-Goncharov algebra}, \red{which} is the quantum torus algebra defined as follows:
\begin{equation}
\mathcal{Z}_{\omega}(\fS,\lambda)= \bT_{\omega}(Q_\lambda) = R \langle 
Z_v^{\pm 1}, v \in V_\lambda \rangle / (
Z_v 
Z_{v'}= \omega^{\, 2 Q_\lambda(v,v')} 
Z_{v'} 
Z_v \text{ for } v,v'\in V_\lambda ),
\end{equation}
\red{into which the usual Fock-Goncharov algebra $\mathcal{X}_q(\fS,\lambda)$ embeds, as
$$
X_v \mapsto Z_v^n, \quad \mbox{for each $v\in V_\lambda$.}
$$}
\red{One may recall $q =\omega^{n^2}$.} We will regard 
$\mathcal{X}_q(\fS,\lambda)$ as a subalgebra of 
$\mathcal{Z}_\omega(\fS,\lambda)$.
For any integer $k\in\bZ$ and $v\in\V$, we will regard the symbol
$(
X_v)^{\frac{k}{n}}$ as referring to $
(Z_v)^k$.

\begin{remark}\label{rmk:ground_rings}
    \red{The usual Fock-Goncharov algebra $\mathcal{X}_q(\fS,\lambda)$ can be formulated over a ground ring with just $q^{\pm \frac{1}{2}}$, instead of $\omega^{\pm \frac{1}{2}}$. In principle, it might be better to take different ground rings for $\mathcal{X}_q(\fS,\lambda)$ and $\mathcal{Z}_\omega(\fS,\lambda)$, but it is not necessary to do so in the present paper.}
\end{remark}

\red{Before recalling the statement about the quantum trace map, we need one last ingredient, which is used to pin down a special subalgebra of $\mathcal{Z}_\omega(\fS,\lambda)$, where the image of the quantum trace map lies in, and which the main result of the current paper will be formulated on in a later section.}

\red{We begin with a triangle $\mathbb{P}_3$.} One can view an element of $\mathbb{Z}^{V_{\mathbb{P}_3}}$ either as a tuple of integers, or as a function $V_{\mathbb{P}_3} \to \mathbb{Z}$.  
Define the elements ${\bf k}_i \in \mathbb{Z}^{V_{\mathbb{P}_3}}$ ($i=1,2,3$) as the functions ${\bf k}_i\colon V_{\bP_3} \to\bZ\ 
$ 
given by
\begin{equation}
{\bf k}_1(ijk)=i,\quad {\bf k}_2(ijk)=j,\quad {\bf k}_3(ijk)=k.\label{def:proj}
\end{equation} 
Let $\mathcal B_{\bP_3}\subset\bZ^{V_{\bP_3}}$ be the subgroup generated by ${\bf k}_1,{\bf k}_2,{\bf k}_3$ and $(n\bZ)^{V_{\bP_3}}$. Elements in $\mathcal{B}_{\bP_3}$ are called \textbf{balanced}. 

\red{Now, for a triangulation $\lambda$ of a triangulable pb surface $\fS$,} a vector $\bk\in\bZ^{\V}$ is \textbf{balanced} if its pullback ${\bf k}_\tau:=f_\tau^\ast\bk$ to ${\bP_3}$ is balanced for every face of $\lambda$, where for every face $\tau$ and its characteristic map  $f_\tau\colon\mathbb P_3\to\fS$, the pullback $f_\tau^\ast\bk$ is a vector $\V\to\bZ$ given by $f_\tau^\ast\bk(v)=\bk(f_\tau(v))$. 
Let 
\begin{align}
    \label{B_lambda}
    \mathcal{B}_\lambda = \mbox{the subgroup of $\mathbb{Z}^{V_\lambda}$ generated by all the balanced vectors.}
\end{align}

\def\lt{\text{lt}}

\red{At last,} the \textbf{balanced Fock-Goncharov algebra} (\cite[\S11.5]{LY23})\red{, which is the promised special subalgebra of $\mathcal{Z}_\omega(\fS,\lambda)$,} is defined as 
\begin{align}
    \label{Z_bl_omega}
\bX=\text{span}_R\{Z^\bk\mid \bk\in\mathcal B_\lambda\}\subset
\mathcal{Z}_\omega(\fS,\lambda),
\end{align}
where $Z^\bk$ is as defined in \eqref{eq-weyl-order-Z}; if ${\bf k} = (k_v)_{v\in V_\lambda}\in \mathbb{Z}^{V_\lambda}$, then $Z^{\bf k}$ is given by $\omega^{m/2} \prod_{v\in V_\lambda} Z^{k_v}$, where the product is taken with respect to any chosen ordering on $V_\lambda$, and $m$ is a suitable integer. 
It is easy to verify that
$$
\mathcal{X}_q(\fS,\lambda) \subset \bX \red{\subset \mathcal{Z}_\omega(\fS,\lambda)},
$$
and that $\{Z^\bk\mid \bk\in\mathcal B_\lambda\}$ is an $R$-basis of $\bX$. \red{One could say that while $\mathcal{Z}_\omega(\fS,\lambda)$ is the span of all Laurent monomials $Z^{\mathbf{k}}$ (with $\mathbf{k} \in \mathbb{Z}^{V_\lambda}$), the balanced subalgebra $\mathcal{Z}_\omega^{\mathrm{bl}}(\fS,\lambda)$ is the span of all Laurent monomials $Z^{\mathbf{k}}$ that are balanced, i.e. with $\mathbf{k} \in \mathcal{B}_\lambda$.}

\red{We are now ready to recall the result on the quantum trace map.}

\begin{theorem}\cite{BW11,Le18,CL22,Douglas,Kim20,LY22,LY23}\label{thm.quantum_trace}
    Suppose that $\fS$ is a triangulable pb surface with a triangulation $\lambda$. \red{Let $\overline{\cS}_\omega(\fS)$ be the reduced stated $\mathrm{SL}_n$-skein algebra in \eqref{eq:cS_red}.} Then the following hold.

    \begin{enumerate}[label={\rm (\alph*)}]
    \item There is an algebra homomorphism
    $$\tr \colon \rdS\rightarrow
    \mathcal{Z}_\omega(\fS,\lambda),$$
    called the {\bf (reduced) quantum trace map}, such that for each framed oriented knot $\alpha$ in $\fS\times (-1,1)$, when $\omega^{\frac{1}{2}}$ is specialized at $1$, the Laurent polynomial $\tr(\alpha)$ recovers the trace-of-monodromy function along the loop ${\rm pr}(\alpha)$ in $\fS$ on Fock and Goncharov's moduli space $\mathscr{X}_{{\rm SL}_n,\fS}$ written in terms of their special coordinates \cite{FG06} (see e.g. \cite{Douglas} for this polynomial at the classical setting).

    \item 
    We have
    \begin{align}
    \label{eq.tr_image_in_bl}
    \im\tr\subset \bX.
    \end{align}

    \item  
    $\tr$ is injective when $n=2$ or when $n=3$.

    \item 
    $\tr$ is injective when $\fS$ is a polygon \red{$\mathbb{P}_k$ (with $k\ge 3$)}.
    \end{enumerate}
\end{theorem}

\red{This Theorem was established for $n=2$ fully in \cite{BW11, Le18, CL22, LY22}, for $n=3$ partially in \cite{Douglas} and fully in \cite{Kim20}, and for general $n\ge 2$ (parts (a) and (c)) in \cite{LY23}.}

\red{As mentioned, our main result of the paper will involve the balanced subalgebra $\mathcal{Z}_\omega^{\mathrm{bl}}(\fS,\lambda)$. For certain arguments, it will become more convenient to have a more invariant interpretation of the balancedness condition for a vector $\mathbf{k}$ in $\mathbb{Z}^{V_\lambda}$, which we now recall from \cite{LY23}.}

\red{We first need a matrix $H_\lambda$ obtained by slightly modifying the matrix $Q_\lambda$ as follows.} Suppose that $v,v'\in V_\lambda$, If $v$ and $v'$ are not contained in the same boundary edge, define $H_\lambda(v,v')=Q_\lambda(v,v')$.
If $v$ and $v'$ are contained in the same boundary edge, define $$H_\lambda(v,v')=\begin{cases}
    -1 & \text{when $v=v'$},\\
    1 & \text{when there is an arrow from $v$ to $v'$},\\
    0 & \text{otherwise.}
\end{cases}$$

\begin{remark}
    Note that $Q_\lambda$ equals $\frac{1}{2}\overline{\mathsf{Q}}_\lambda$ in \cite{LY23} and $H_\lambda$ equals $-\overline{\mathsf{H}}_\lambda$ defined in \cite[\S11.7]{LY23}.
\end{remark}

\red{The following states a convenient necessary condition for the balancedness, which we will use later.}
\begin{lemma}\cite[Proposition 11.10]{LY23}\label{lem-balanced-H}
    If a vector ${\bf k}\in\mathbb Z^{V_\lambda}$, considered as a row vector, is balanced, then ${\bf k} H_\lambda\in (n\mathbb Z)^{V_\lambda}$\red{, where $H_\lambda$ is viewed as a matrix}.
\end{lemma}

\begin{remark}\label{rem.k_balancedness_condition}
Proposition 11.10 in \cite{LY23}  
assumes the condition that $\fS$ contains no interior punctures because this proposition involves a matrix $\overline{\mathsf{K}}_\lambda$ (see \cite[\S11.6]{LY23}), which is defined only when $\fS$ contains no interior punctures.
But the arguments in \cite[Proposition 11.10]{LY23} yield the above
Lemma \ref{lem-balanced-H}  
for general $\fS$. 
In \cite[Proposition 11.10]{LY23}, they also show that $${\bf k} H_\lambda\in (n\mathbb Z)^{V_\lambda}\implies {\bf k}\text{ 
is balanced}$$ using the matrix $\overline{\mathsf{K}}_\lambda$ when 
$\fS$ contains no interior punctures.
\end{remark}

\section{\red{Quantum cluster mutations and mutable-balanced quantum coordinate change maps}}

We recall basics of quantum cluster \red{varieties}, obtained by gluing quantum torus algebras associated to cluster seeds along quantum mutations. We refer the readers to \cite{BZ,FG09a,FG09b}. Then we extend the known quantum mutations to `\red{mutable-}balanced' subalgebras for $n$-th root variables. \red{We then apply these extensions to our case relevant to the $\mathrm{SL}_n$-skein algebras of pb surfaces $\fS$, and construct a mutable-balanced quantum coordinate change map associated to each change of triangulations of $\fS$.} 
\red{This coordinate change map} was studied in \cite{Hiatt,BW11} for $n=2$, in \cite{Kim21} for $n=3$\red{, and for general $n$ for the first time here}. \red{It will be crucial in the formulation of our main result in a later section.}

\subsection{Classical cluster $\mathcal X$-mutations}

Fix a non-empty set $\mathcal{V}$, and a non-empty subset $\mathcal{V}_{\rm mut}$ of $\mathcal{V}$. The elements of $\mathcal{V}$, $\mathcal{V}_{\rm mut}$ and $\mathcal{V}\setminus \mathcal{V}_{\rm mut}$ are called {\bf vertices}, {\bf mutable vertices} and {\bf frozen vertices}, respectively. \red{Generalizing what we saw in \S\ref{subsec.quantum_trace_maps}, here} by a \red{(weighted)} quiver $\Gamma$ we mean a directed graph whose set of vertices is $\mathcal{V}$ 
and equipped with weight on the edges so that the weight on each edge is $1$ unless the edge connects two frozen vertices, in which case the weight is $\frac{1}{2}$. 
\red{Similarly as in \eqref{eq:Q_lambda}}, we denote by $Q = (Q(u,v))_{u,v\in \mathcal{V}}$ to denote the signed adjacency matrix of $\Gamma$. A quiver is considered up to equivalence, where two quivers are equivalent if they yield the same signed adjacency matrix. 
Let $\mathcal{F}$ be the field of rational functions over $\mathbb{Q}$ with the set of algebraically independent generators enumerated by $\mathcal{V}$, say $\mathcal{F} = \mathbb{Q}(\{y_v\}_{v\in \mathcal{V}})$. A {\bf cluster $\mathcal{X}$-seed} is a pair $\mathcal{D} = (\Gamma,(X_v)_{v\in \mathcal{V}})$, where $\Gamma$ is a quiver and $\{X_v\}_{v\in \mathcal{V}}$ forms an algebraically independent generating set of $\mathcal{F}$ over $\mathbb{Q}$. The signed adjacency matrix $Q$ of $\Gamma$ is called the {\bf exchange matrix} of the seed $\mathcal{D}$.

\def\sgn{\text{sgn}}

Suppose that $k\in\mathcal V_{
\rm mut}$. The {\bf mutation}
$\mu_k$ at the 
mutable vertex $k\in\mathcal{V}_{\rm mut}$ is defined to be the process that transforms a seed $\mathcal{D} = (\Gamma,(X_v)_{v\in \mathcal{V}})$ into a new seed
$\mu_k(\mathcal D) = \mathcal D' = 
(\Gamma', (
X_v')_{v\in 
\mathcal{V}})
$.
Here 
$$
X_v'= \begin{cases}
    X_v^{-1} & v=k,\\
    X_v (1+ 
    X_k^{-\text{sgn}(Q(v,k)) })^{-Q(v,k)} & v\neq k,
\end{cases}
$$
where 
$$
\sgn(a)=
\begin{cases}
    1 & a>0,\\
    0 & a=0,\\
    -1 & a<0,
\end{cases}
$$
for $a\in\mathbb R$, and $\Gamma'$ is obtained from 
$\Gamma$ by the following procedures: (1)
 reverse all the arrows incident to the vertex $k$,
(2) for each pair of arrows $k\rightarrow i$ and $j\rightarrow k$ draw an arrow $i\rightarrow j$,
(3) delete pairs of arrows $i\rightarrow j$ and $j\rightarrow i$ going in the opposite directions (more precisely, tidy up the quiver while maintaining the signed adjacency matrix).
We use $Q' = (Q'(u,v))_{u,v\in\mathcal V}$ to denote the signed adjacency matrix of $\Gamma'$. Then we have 
\begin{align}\label{eq-mutation-Q}
Q'(u,v) = \begin{cases}
    - Q(u,v) & k\in\{u,v\},\\
    Q(u,v) +\frac{1}{2}(Q(u,k)|Q(k,v)|+
    |Q(u,k)|Q(k,v)) & k\notin\{u,v\}.
\end{cases}
\end{align}

It is well-known that $\mu_k(\mu_k(\mathcal D)) = \mathcal D$.
\def\Fr{{\rm Frac}}

\subsection{Quantum mutations for Fock-Goncharov algebras}\label{subsec:quantum_mutations_X}

Suppose that 
$\mathcal{D} = (\Gamma,(X_v)_{v\in \mathcal{V}})$ is an $\mathcal X$-seed 
whose 
exchange matrix 
is $Q$.
\red{Generalizing the algebra $\mathcal{X}_q(\fS,\lambda)$ which appeared in \S\ref{subsec.quantum_trace_maps},} define the {\bf Fock-Goncharov algebra} \red{\cite{FG09a,FG09b}} associated to $\mathcal D$ to be the quantum torus algebra
\begin{align*}
\mathcal{X}_q(\mathcal{D}) & = \bT_q(Q) \\
& = R\langle X_v^{\pm 1}, v\in \mathcal{V} \rangle / (X_v X_{v'} = q^{2Q(v,v')} X_{v'} X_v \mbox{ for } v,v' \in \mathcal{V}).
\end{align*}
It is equipped with the $*$-structure, whose $*$-map $\mathcal{X}_q(\mathcal{D}) \to \mathcal{X}_q(\mathcal{D})$ is a ring anti-homomorphism sending each generator $X_v^{\pm 1}$ to itself $X_v^{\pm 1}$, and $q^{\pm 1}$ to its inverse $q^{\mp 1}$.
We understand $\mathcal{X}_q(\mathcal{D})$ as the quantum algebra associated to the seed $\mathcal{D}$, which yields the classical algebra associated to $\mathcal{D}$ as $q\to 1$. In this subsection we recall the quantum version of the mutation formula.

The quantum mutation shall be a non-commutative rational map. So we need to use the skew-field of fractions of $\mathcal{X}_q(\mathcal{D})$, which we denote by ${\rm Frac}(\mathcal{X}_q(\mathcal{D}))$; for the existence of ${\rm Frac}(\mathcal{X}_q(\mathcal{D}))$, see \cite{Cohn}. 

A useful notion is the Weyl-ordered product, which is a certain normalization of a Laurent monomial defined as follows: for any $v_1,\ldots,v_r \in \mathcal{V}$ and $a_1,\ldots,a_r \in \mathbb{Z}$,
\begin{align}
\label{Weyl_ordering}
\left[ X_{v_1}^{a_1} X_{v_2}^{a_2} \cdots X_{v_r}^{a_r} \right] := q^{-\sum_{i<j}a_i a_j Q(v_i,v_j)} X_{v_1}^{a_1} X_{v_2}^{a_2} \cdots X_{v_r}^{a_r}
\end{align}
In particular, for ${\bf t} = (t_v)_{v\in \mathcal{V}} \in \mathbb{Z}^\mathcal{V}$, the following notation for the corresponding Weyl-ordered Laurent monomial will become handy:
\begin{align}
    \label{Weyl_ordering_X}
    X^{\bf t} := \left[ \prod_{v\in \mathcal{V}} X_v^{t_v}\right].
\end{align}
Due to the property of Weyl-ordered products, note that this element is independent of the choice of the order of the product inside the bracket, and that it is invariant under the $*$-map.

The following special function is a 
\red{useful} ingredient \red{for the understanding of quantum mutations}.
\begin{definition}[compact quantum dilogarithm \cite{F95,FK94}]
    The quantum dilogarithm for a quantum parameter
$q$ is the function
$$\Psi^q(x) = \prod_{r=1}^{+\infty} (1+ q^{2r
-1} x)^{-1}.$$
\end{definition}
For our purposes, the infinite product can be understood formally, as what matters for us is only the following:

\def\sgn{{\rm sgn}}

\begin{lemma}\label{lem-F-P}
    Suppose that $xy = q^{2m} yx$ for some $m\in \mathbb{Z}$. Then we have 
    $${\rm Ad}_{\Psi^q(x)}(y) = y F^q(x,m),$$
    where
\begin{align}
    \label{eq:Fq}
    F^q(x,m)=\prod_{r=1}^{|m|} (1+ q^{(2r-1){\rm sgn}(m)}x)^{\sgn(m)}.
\end{align}
\end{lemma}
For example, $F^q(x,0) = 1$, $F^q(x,1) = 1+qx$, $F^q(x,2) = (1+qx)(1+q^3x)$, $F^q(x,-1) = (1+q^{-1} x)^{-1}$, $F^q(x,-2) = (1+q^{-1}x)^{-1}(1+q^{-3}x)^{-1}$, etc.

\red{Instead of understanding Lemma \ref{lem-F-P} as something which one can prove for certain specific setup where one can make a rigorous sense of the infinite products, one can understand it just heuristically, and interpret the expression ${\rm Ad}_{\Psi^q(x)}$ as a symbol for a map that sends $y$ to $y \cdot F^q(x,m)$.} 
In particular, if $x = X_k$ and $y = X_v$ are generators of the skew-field ${\rm Frac}(\mathcal{X}_q(\mathcal{D}))$, consider the expression
\begin{align}
    \label{eq:Ad_map_on_generator}
    {\rm Ad}_{\Psi^q(X_k)}(X_v) = X_v \cdot F^q(X_k, Q(k,v)).
\end{align}

It is easy to check that ${\rm Ad}_{\Psi^q(X_k)}$ extends to a unique skew-field isomorphism from ${\rm Frac}(\mathcal{X}_q(\mathcal{D}))$ to itself, which we denote by ${\rm Ad}_{\Psi^q(X_k)}$. \red{More concretely, to check this, one should verify that ${\rm Ad}_{\Psi^q(X_k)}$ preserves the defining relations $X_v X_{v'} = q^{2Q(v,v')} X_{v'} X_v$; one should check ${\rm Ad}_{\Psi^q(X_k)}(X_v) {\rm Ad}_{\Psi^q(X_k)}(X_{v'}) = q^{2Q(v,v')} {\rm Ad}_{\Psi^q(X_k)}(X_{v'}) {\rm Ad}_{\Psi^q(X_k)}(X_v)$, which is a straightforward exercise.}

\begin{definition}[quantum $\mathcal X$-mutation for Fock-Goncharov algebras \cite{BZ,FG09a,FG09b}]\label{def.quantum_X-mutation}
     Suppose 
     that $\mathcal{D} = (\Gamma,(X_v)_{v\in \mathcal{V}})$ is an $\mathcal X$-seed, 
     $k$$\in \mathcal{V}_{\rm mut}$ is a 
     mutable vertex of $\Gamma$, and $\mathcal D'=\mu_k(\mathcal D)$.
      The {\em quantum mutation map} 
      is the 
      isomorphism between the 
      skew-fields
      $$\mu_{
      \mathcal{D},\mathcal{D}'}^q=\mu_k^q \, \colon \, \Fr(
      \mathcal{X}_q(\mathcal D'))\rightarrow
      \Fr(
      \mathcal{X}_q(\mathcal D))$$
      given by the composition
      $$\mu_k^q= \mu_k^{\sharp q}\circ \mu_k',$$
      where $\mu'_k$ is the isomorphism of skew-fields
      $$\mu_k'\colon \Fr(
      \mathcal{X}_q(\mathcal D'))\rightarrow
      \Fr(
      \mathcal{X}_q(\mathcal D))$$ 
      given by 
      \begin{align}\label{eq-quantum-mutation}
      \mu_k'(
      X_v') = 
      \begin{cases}
          X_k^{-1}& \mbox{if } v=k,\\
          \left[ X_v X_k^{[Q(v,k)]_+}\right]
          & \mbox{if } v\neq k.
      \end{cases}
      \end{align}
      where $[a]_+$ stands for the positive part of a real number $a$
      $$
      [a]_+ := \max\{a,0\} = {\textstyle \frac{1}{2}(a+|a|)},
$$
and $[\sim]$ is the Weyl-ordered product as in \eqref{Weyl_ordering},
      while
       $\mu^{\sharp q}_k$ is the following skew-field automorphism
       $$\mu_k^{\sharp q} = {\rm Ad}_{\Psi^q(
       X_k)}\colon \Fr(
       \mathcal{X}_q(\mathcal D))\rightarrow
      \Fr(
      \mathcal{X}_q(\mathcal D)),$$
      \red{defined on the generators by the formula \eqref{eq:Ad_map_on_generator}.}
\end{definition}

Suppose that $\mathcal{D} = (\Gamma,(X_v)_{v\in \mathcal{V}})$ 
is an $\mathcal X$-seed, 
     $k\in \mathcal{V}_{\rm mut}$ is a mutable vertex, and $\mathcal D'=\mu_k(\mathcal D)$. We have the following\red{, which relates the kernels of the exchange matrix $Q$ of $\Gamma$ and the mutated one $Q' = \mu_k(Q)$}.

\begin{lemma}\label{lem-commute}
Let $\mathbb{F} = \mathbb{Z}$ or $\mathbb{Z}/m\mathbb{Z}$ for some positive integer $m\ge 2$.
    Let $\mathcal V_0\subset\mathcal V$ contain $k$, and let ${\bf t}'=(t_v')_{v\in\mathcal V}\in\bZ ^{\mathcal V}$ 
    satisfy $\sum_{v\in\mathcal V} Q'(u,v) t_v' = 0\in 
    \mathbb{F}$ for all $u\in\mathcal V_0$\red{, where $Q'=\mu_k(Q)$}.
    Suppose that  $\mu_k'((
    X')^{{\bf t}'}) = 
    X^{\bf t}$, where
    ${\bf t}=(
    t_v)_{v\in\mathcal V}\in\bZ ^{\mathcal V}$.
    Then we have 
    $\sum_{v\in\mathcal V} Q(u,v) t_v = 0\in 
    \mathbb{F}$ for all $u\in\mathcal V_0$.
\end{lemma}

This lemma, which will be used later in our proofs, can be verified by straightforward computations using definitions. It can also be viewed as a corollary of \cite[Lem.2.7]{FG09b} or \cite[Lem.3.7(1)]{Kim_irreducible}, where the core of the necessary computation can be found. Hence we omit a proof of the above lemma.

\subsection{Quantum mutations for mutable-balanced $n$-th root Fock-Goncharov algebras}\label{subsec:quantum_mutations_for_mutable-balanced_n-th_root_algebras}
Suppose 
that $\mathcal{D} = (\Gamma,(X_v)_{v\in \mathcal{V}})$ is an $\mathcal X$-seed 
whose exchange matrix is $Q$.
\red{Generalizing the algebra $\mathcal{Z}_\omega(\fS,\lambda)$ which appeared in \S\ref{subsec.quantum_trace_maps}, we} define the {\bf $n$-th root Fock-Goncharov algebra} associated to $\mathcal D$ to be
\begin{align*}
\mathcal{Z}_\omega(\mathcal{D}) & = \mathbb{T}_\omega(Q) \\
& = R\langle Z_v^{\pm 1}, v\in \mathcal{V} \rangle / (Z_v Z_{v'} = \omega^{2Q(v,v')} Z_{v'} Z_v \mbox{ for } v,v' \in \mathcal{V}),
\end{align*}
equipped with the $*$-map $\mathcal{Z}_\omega(\mathcal{D}) \to \mathcal{Z}_\omega(\mathcal{D})$ sending each $Z_v^{\pm 1}$ to $Z_v^{\pm 1}$ and $\omega^{\pm 1}$ to $\omega^{\mp 1}$, 
into which the Fock-Goncharov algebra $\mathcal{X}_q(\mathcal{D})$ \red{(defined in \S\ref{subsec:quantum_mutations_X})} embeds as
$$
X_v \mapsto Z_v^n, \quad \forall v\in \mathcal{V}.
$$
\red{One may recall $q=\omega^{n^2}$, and also Remark \ref{rmk:ground_rings}.}

We use $\Fr(
\mathcal{Z}_\omega(\mathcal{D}))$ to denote the 
skew-field of fractions of $
\mathcal{Z}_\omega(\mathcal D)$. 
Naturally, ${\rm Frac}(\mathcal{X}_q(\mathcal{D}))$ embeds into ${\rm Frac}(\mathcal{Z}_\omega(\mathcal{D}))$. The Weyl-ordered product can be defined similarly as before: for any $v_1,\ldots,v_r \in \mathcal{V}$ and $a_1,\ldots,a_r \in \mathbb{Z}$,
\begin{align}
\label{Weyl_ordering_Z}
\left[ Z_{v_1}^{a_1} \cdots Z_{v_r}^{a_r} \right] := \omega^{-\sum_{i<j}a_i a_jQ(v_i,v_j)} Z_{v_1}^{a_1} \cdots Z_{v_r}^{a_r}.
\end{align}
For ${\bf t} =(t_v)_{v\in \mathcal{V}} \in \mathbb{Z}^\mathcal{V}$ one defines the Weyl-ordered Laurent monomial:
\begin{align}\label{eq-weyl-order-Z}
Z^{\bf t} := \left[ \prod_{v\in \mathcal{V}}Z_v^{t_v} \right],
\end{align}
which does not depend on the order of the product taken, and which is invariant under the $*$-map.

The goal of the present subsection is to extend the quantum mutation maps $\mu^q_k$ defined on ${\rm Frac}(\mathcal{X}_q(\mathcal{D}))$ to the $n$-\red{th} root versions on ${\rm Frac}(\mathcal{Z}_\omega(\mathcal{D}))$. We will soon see that the extension\red{, which is novel in the present paper,} should be defined on a certain sub-skew-field of ${\rm Frac}(\mathcal{Z}_\omega(\mathcal{D}))$, rather than the entire ${\rm Frac}(\mathcal{Z}_\omega(\mathcal{D}))$.

Let $k$ be a 
mutable vertex of $\Gamma$ and let 
$$
\mathcal D' =
\mu_k(\mathcal D).
$$
Recall that the quantum mutation map $\mu^q_k : {\rm Frac}(\mathcal{X}_q(\mathcal{D}')) \to {\rm Frac}(\mathcal{X}_q(\mathcal{D}))$ is presented as the composition $\mu^q_k = \mu^{\sharp q}_k \circ \mu'_k$ of the monomial transformation part $\mu'_k : {\rm Frac}(\mathcal{X}_q(\mathcal{D}')) \to {\rm Frac}(\mathcal{X}_q(\mathcal{D}))$ and the automorphism part $\mu^{\sharp q}_k : {\rm Frac}(\mathcal{X}_q(\mathcal{D})) \to {\rm Frac}(\mathcal{X}_q(\mathcal{D}))$. We shall extend each of $\mu'_k$ and $\mu^{\sharp q}_k$ to its $n$-\red{th} root versions.

     First, define the 
     skew-field isomorphism $$\nu_k'\colon \Fr(
     \mathcal{Z}_\omega(\mathcal D'))\rightarrow
      \Fr(
      \mathcal{Z}_\omega(\mathcal D))$$
      such that the action of $\nu_k'$ on the generators is given by 
      the formula \eqref{eq-quantum-mutation} for $\mu'_k$ with $
      X_{v}$ replaced by $
      Z_{v}$:
     \begin{align}\label{eq-quantum-mutation_Z}
      \nu_k'(Z_v') = 
      \begin{cases}
          Z_k^{-1}& \mbox{if } v=k,\\
          \left[ Z_v Z_k^{[Q(v,k)]_+}\right]
          & \mbox{if } v\neq k.
      \end{cases}
      \end{align}
      It is easy to check that $\nu'_k$ extends $\mu'_k$. We recall a useful lemma, which is straightforward to verify:
      \begin{lemma}[{\cite[Lem.3.19]{Kim21}}]
      \label{lem.nu_k_prime_preserves_Weyl-ordering}
          The map $\nu'_k$ sends a Weyl-ordered Laurent monomial to a Weyl-ordered Laurent monomial.
      \end{lemma}

      For an $n$-\red{th} root extension of the automorphism part $\mu^{\sharp q}_k$, we consider the following subalgebra of $\mathcal{Z}_\omega(\mathcal{D})$.

   \begin{definition}\label{def.mbl}
   Define  
       $$\mathcal B_{\mathcal D}=\{{\bf t}=(
       t_v)_{v\in\mathcal V}\in\mathbb Z^{
       \mathcal{V}}\mid \sum_{v\in 
       \mathcal{V}} Q(u,v)
       t_v=0\in
       \mathbb{Z}/n \mathbb{Z}\text{ for all }u\in\mathcal V_{
       {\rm mut}}\},$$
    and define 
    $\mathcal{Z}^{\rm mbl}_\omega(\mathcal{D})$
    to be the $R$-submodule of $
    \mathcal{Z}_\omega(\mathcal D)$ spanned by $Z^{{\bf t}}$ for ${\bf t}\in \mathcal B_{\mathcal D}$. We call $\mathcal{Z}^{\rm mbl}_\omega(\mathcal{D})$ the {\bf mutable-balanced subalgebra} of $\mathcal{Z}_\omega(\mathcal{D})$.
   \end{definition}

Note that $\mathcal B_{\mathcal D}$ is a subgroup of $\mathbb Z^{
\mathcal{V}}$. So 
$\mathcal{Z}^{\rm mbl}_\omega(\mathcal{D})$ is 
    indeed an $R$-subalgebra of $
    \mathcal{Z}_\omega(\mathcal D)$.
    Since $n \mathbb Z^{
    \mathcal{V}}\subset\mathcal B_{\mathcal D}\subset \mathbb Z^{
    \mathcal{V}}$, we have 
    $$\mathcal{X}_q(\mathcal{D}) \subset \mathcal{Z}^{\rm mbl}_\omega(\mathcal{D}) \subset \mathcal{Z}_\omega(\mathcal{D}).$$
    The mutable-balanced subalgebra $\mathcal{Z}^{\rm mbl}_\omega(\mathcal{D})$ is precisely the subalgebra of the $n$-th root Fock-Goncharov algebra $\mathcal{Z}_\omega(\mathcal{D})$ to which we can extend the automorphism-part map $\mu^{\sharp q}_k$ for $\mathcal{X}_q(\mathcal{D})$. \red{J}ust as for $\mathcal{X}_q(\mathcal{D})$, one needs to consider the skew-field of fractions ${\rm Frac}(\mathcal{Z}^{\rm mbl}_\omega(\mathcal{D}))$, which does make sense because $\mathcal{Z}^{\rm mbl}_\omega(\mathcal{D})$ is a quantum torus algebra  \cite{Cohn}, which in turn follows from the fact that $\mathcal{B}_\mathcal{D} \subset \mathbb{Z}^\mathcal{V}$ is a finitely generated free $\mathbb{Z}$-module.

\begin{lemma}\label{lem.nu_sharp_well-defined}
There exists a unique skew-field isomorphism
    $$ \nu_k^{\sharp \omega}:={\rm Ad}_{\Psi^q(
    X_k)}\colon \Fr(
    \mathcal{Z}^{\rm mbl}_\omega(\mathcal D))\rightarrow
      \Fr(
      \mathcal{Z}^{\rm mbl}_\omega(\mathcal D))$$
      such that for each ${\bf t} = (t_v)_{v\in \mathcal{V}} \in \mathcal{B}_\mathcal{D}$, we have
      \begin{align}
          \label{eq:nu_sharp_on_generator}
          \nu^{\sharp \omega}_k(Z^{\bf t}) = Z^{\bf t} \, F^q(X_k, m),
      \end{align}
where $m = \frac{1}{n} \sum_{v\in \mathcal{V}} Q(k,v) t_v$\red{, and $F^q$ is as in \eqref{eq:Fq}}.
\end{lemma}

\begin{proof}

\red{For  ${\bf t} = (t_v)_{v\in \mathcal{V}} \in \mathcal{B}_\mathcal{D}$, we have $X_k Z^{\bf t} = q^{2m}Z^{\bf t} X_k$, where $m = \frac{1}{n} \sum_{v\in \mathcal{V}} Q(k,v) t_v$, with $m \in \mathbb{Z}$. So the right hand side of \eqref{eq:nu_sharp_on_generator} is well-defined; recall that the symbol ${\rm Ad}_{\Psi^q(X_k)}$ should just be taken as a symbol that represents the formula \eqref{eq:nu_sharp_on_generator}. As $\mathcal{Z}^{\mathrm{mbl}}_\omega(\mathcal{D})$ is a quantum torus algebra, similarly as in the discussion after Lemma~\ref{lem-F-P}, in order to show that \eqref{eq:nu_sharp_on_generator} uniquely extends to an algebra homomorphism, it suffices to check that the defining commutation relations $Z^{\mathbf{t}} Z^{\mathbf{t}'} = \omega^{2 m(\mathbf{t},\mathbf{t}')} Z^{\mathbf{t}'} Z^{\mathbf{t}}$ are preserved, for each $\mathbf{t}=(t_v)_{v\in \mathcal{V}}$ and $\mathbf{t}'=(t_v')_{v\in \mathcal{V}}$ in $\mathcal{B}_\mathcal{D}$, where $m(\mathbf{t},\mathbf{t}') := \sum_{v,v'\in\mathcal{V}} t_v t'_{v'} Q(v,v')$. That is, one should check $\nu^{\sharp \omega}_k(Z^{\bf t}) \nu^{\sharp \omega}_k(Z^{{\bf t}'}) = \omega^{2m(\mathbf{t},\mathbf{t}')} \nu^{\sharp \omega}_k(Z^{{\bf t}'})\nu^{\sharp \omega}_k(Z^{\bf t})$, which is a straightforward exercise.
}

Similar arguments show that there is a unique skew-field homomorphism
$$
{\rm Ad}_{\Psi^q(X_k)^{-1}} : {\rm Frac}(\mathcal{Z}^{\rm mbl}_\omega(\mathcal{D})) \to {\rm Frac}(\mathcal{Z}^{\rm mbl}_\omega(\mathcal{D}))
$$
such that ${\rm Ad}_{\Psi^q(X_k)^{-1}}(Z^{\bf t}) = Z^{\bf t} F^q(X_k,m)^{-1}$ for each ${\bf t} = (t_v)_{v\in \mathcal{V}} \in \mathcal{B}_\mathcal{D}$, with $m = \frac{1}{n}\sum_{v\in \mathcal{V}} Q(k,v) t_v$. It is easy to see that $\nu^{\sharp \omega}_k = {\rm Ad}_{\Psi^q(X_k)}$ and ${\rm Ad}_{\Psi^q(X_k)^{-1}}$ are inverses to each other.
\end{proof}

One can observe that ${\rm Frac}(\mathcal{X}_q(\mathcal{D})) \subset {\rm Frac}(\mathcal{Z}^{\rm mbl}_\omega(\mathcal{D}))$, and that $\nu^{\sharp \omega}_k$ indeed extends $\mu^{\sharp q}_k$ defined in Def.\ref{def.quantum_X-mutation}, 
as already hinted by the fact that we denoted both by the symbol ${\rm Ad}_{\Psi^q(X_k)}$ in eq.\eqref{lem-F-P}.

The following lemma, which follows from Lemma \ref{lem-commute}, says that the skew-fields of fractions of the mutable-balanced subalgebras are compatible with the monomial-transformation isomorphism $\nu'_k$.

\begin{lemma}\label{lem.nu_prime_restricts}
    The map $\Fr(
    \mathcal{Z}_\omega(\mathcal D'))\xrightarrow{\nu_k'}
      \Fr(
      \mathcal{Z}_\omega(\mathcal D))$ defined in \eqref{eq-quantum-mutation_Z} restricts to a  
      skew-field isomorphism
      $\Fr(
      \mathcal{Z}^{\rm mbl}_\omega(\mathcal D'))\xrightarrow{\nu_k'}
      \Fr(
      \mathcal{Z}^{\rm mbl}_\omega(\mathcal D)).$ \qed
\end{lemma}

We define the quantum mutation $\nu_k^{\omega}$ for the $n$-th root Fock-Goncharov algebras, or more precisely for the skew-fields of fractions of their mutable-balanced subalgebras,  
to be the composition 
$$
\nu^\omega_k := \nu^{\sharp\omega}_k \circ \nu'_k ~:~ 
\Fr(
\mathcal{Z}^{\rm mbl}_\omega(\mathcal D'))\xrightarrow{\nu_k'}
      \Fr(
      \mathcal{Z}^{\rm mbl}_\omega(\mathcal D))\xrightarrow{\nu_k^{\sharp \omega}}
      \Fr(
      \mathcal{Z}^{\rm mbl}_\omega(\mathcal D)),$$
\red{which is well defined due to Lemmas \ref{lem.nu_sharp_well-defined} and \ref{lem.nu_prime_restricts}. It is easy to see:}

\begin{lemma}\label{lem:nu_extends_mu}
    The map $\nu^\omega_k$ extends $\mu^q_k$. \qed
\end{lemma}

One useful property is:
\begin{lemma}[{\cite[Lem.3.27]{Kim21}}]\label{lem:nu_omega_preserves_star}
    The map $\nu^\omega_k$ preserves the $*$-structures. 
\end{lemma}

    In order to apply to our situation, consider a triangulable pb surface $\frak{S}$ with a triangulation $\lambda$.
    Letting $\mathcal{V}$ be $V_\lambda$ and $\mathcal{V}_{\rm mut}$ be the subset of $V_\lambda$ consisting of the vertices contained in the interior of $\frak{S}$, one obtains a cluster $\mathcal{X}$-seed $\mathcal{D}_\lambda := (\Gamma_\lambda,(X_v)_{v\in V_\lambda})$\red{; see \S\ref{subsec.quantum_trace_maps} for description of the vertex set $V_\lambda$ and the quiver $\Gamma_\lambda$}.
    We denote the $n$-root Fock-Goncharov algebra and its mutable-balanced subalgebra as
    \begin{align}
        \label{eq.Z_algebras_for_lambda}
        \mathcal{Z}_\omega(\fS,\lambda) = \mathcal{Z}_\omega(\mathcal{D}_\lambda), \qquad
    \mathcal{Z}_\omega^{\rm mbl}(\fS,\lambda) = \mathcal{Z}^{\rm mbl}_\omega(\mathcal{D}_\lambda).
    \end{align}
On the other hand, recall the balanced subalgebra $\mathcal{Z}^{\rm bl}_\omega(\fS,\lambda)$ of $\mathcal{Z}_\omega(\frak{S}, \lambda)$ defined in \eqref{Z_bl_omega}.

\begin{lemma}
\label{lem:bl_in_mbl}
  We have $\mathcal{Z}^{\rm bl}_\omega(\fS,\lambda) \subset \mathcal{Z}^{\rm mbl}_\omega(\fS,\lambda)$.
\end{lemma}
\begin{proof}
It suffices to show $\mathcal B_\lambda\subset \mathcal B_{
\mathcal{D}_\lambda}$.
    Suppose that ${\bf t}=(
    t_v)_{v\in V_\lambda}\in\mathcal B_\lambda$ and $k\in\mathcal V_{
    \rm mut}$. Lemma \ref{lem-balanced-H} implies that
    $\sum_{v\in V_\lambda} 
    t_v H_\lambda(v,k) = 0\in 
    \mathbb{Z}/n\mathbb{Z}$. From the definition of $H_\lambda$, it follows that $H_\lambda(v,k)= 
    Q_\lambda(v,k)= - Q_\lambda(k,v)$ for any $v\in V_\lambda$. Then we have 
    $$\sum_{v\in V_\lambda} Q_\lambda(k,v) 
    t_v = 0\in
    \mathbb{Z}/n\mathbb{Z}.$$
    This shows that ${\bf t}\in  \mathcal B_{
    \mathcal{D}_\lambda}.$
\end{proof}

      \begin{remark}\label{rem-balanced-mutable}
           One can try to come up with a modified version of $\mathcal{B}_\mathcal{D}$ for a general cluster $\mathcal{X}$-seed $\mathcal{D}$ that involves additional equations about the frozen vertices, so that it matches $\mathcal{B}_\lambda$ in the case of a seed \red{$\mathcal{D}_\lambda$} coming from a triangulation $\lambda$ of a triangulable pb surface $\fS$ with boundary as well. We do not undertake 
          these tasks in the present paper. See \cite[Rem.3.25]{Kim21} for a related question.
      \end{remark}

Eventually, in the next section \S\ref{sec:compatibility} we will be applying the $n$-th root quantum mutations $\nu^\omega_k$ to the balanced subalgebras $\mathcal{Z}^{\rm bl}_\omega(\fS,\lambda)\subset \mathcal{Z}^{\rm mbl}_\omega(\fS,\lambda)$, or more precisely to their skew-fields of fractions ${\rm Frac}(\mathcal{Z}^{\rm bl}_\omega(\fS,\lambda))$, which make sense since $\mathcal{Z}^{\rm bl}_\omega(\fS,\lambda)$ is a quantum torus algebra, which in turn holds because $\mathcal{B}_\lambda \subset \mathcal{B}_{D_\lambda}$ is a finitely generated free $\mathbb{Z}$-module.

Still in this section, we will work on the mutable-balanced algebras $\mathcal{Z}^{\rm mbl}_\omega(\fS,\lambda)$ and the skew-fields of fractions thereof, for triangulations $\lambda$ of a triangulable pb surface $\fS$. In the following subsection, we will compose the maps $\nu^\omega_k$ studied in the present subsection to construct a coordinate change isomorphism between ${\rm Frac}(\mathcal{Z}^{\rm mbl}_\omega(\fS,\lambda))$ for different triangulations $\lambda$.

\subsection{\red{Mutable-balanced quantum} coordinate change isomorphisms 
for change of triangulations}\label{subsec:coordinate_change_isomorphisms_for_change_of_triangulations}

Let $\fS$ be a triangulable pb surface with a triangulation $\lambda$. In the last subsection we considered a particular cluster $\mathcal{X}$-seed $\mathcal{D}_\lambda = (\Gamma_\lambda, (X_v)_{v\in V_\lambda})$, associated to a special quiver $\Gamma_\lambda$ \red{(see \S\ref{subsec.quantum_trace_maps})}, with certain vertex sets $\mathcal{V} = V_\lambda$ and $\mathcal{V}_{\rm mut}$. This cluster $\mathcal{X}$-seed comes from a special coordinate chart on a certain moduli space $\mathscr{P}_{{\rm PGL}_n,\fS}$ studied by Fock and Goncharov \cite{FG06}, and by Goncharov and Shen \cite{GS19}. This space parametrizes ${\rm PGL}_n$-local systems on $\fS$ together with some extra data at the punctures and the boundary of $\fS$. The stated ${\rm SL}_n$-skein algebra $\mathscr{S}_\omega(\fS)$ plays a crucial role in understanding the deformation quantization of this space. In \cite{FG06,GS19}, the authors construct a coordinate chart for this space per each triangulation $\lambda$ of $\fS$, whose coordinate variables are denoted by $X_v$, enumerated by the vertices $v$ of the quiver $\Gamma_\lambda$. They further show that these special charts for different triangulations are related to each other by a special sequence of mutations. In the present paper we do not recall the definition of the moduli space $\mathscr{P}_{{\rm PGL}_n,\fS}$ nor the construction of the coordinates, but only recall the sequence of mutations.

We begin with the elementary change of triangulations, namely with the case when $\lambda$ and $\lambda'$ are triangulations of $\fS$ that differ exactly by one edge, say $e$. We say that each of these triangulations are obtained from each other by a {\bf flip} at the edge $e$. Consider the ideal quadrilateral formed by the two ideal triangles of $\lambda$ sharing $e$ (it may not actually be a quadrilateral, but that doesn't matter here). Let $V_{\lambda;e}$ be the set of all vertices of $V_\lambda$ that either lie in the interior of one of these two triangles or lie on $e$; so $V_{\lambda;e}$ consists of $(n-1)^2$ vertices. For each $i=0,1,2,\ldots,n-2$, we will define a subset $V_{\lambda;e}^{(i)}$ of $V_{\lambda;e}$. Draw this ideal quadrilateral for $\lambda$ as in Figure \ref{Fig;mutation_sequence_for_flip}, so that $e$ is the `middle vertical line', and after flipping at $e$, the flipped $e$ would be the `middle horizontal line'. As in Figure \ref{Fig;mutation_sequence_for_flip}, for each vertex $v$ in $V_{\lambda,e}$ one can define the vertical distance $d_{\rm vert}(v) \in \mathbb{N}$ from the middle horizontal line, and the horizontal distance $d_{\rm hori}(v) \in \mathbb{N}$ from the middle vertical line. Define
$$
V_{\lambda;e}^{(i)} := \left\{v\in V_{\lambda;e} \, \left| \, \begin{array}{ll} d_{\rm vert}(v) \le i, &  d_{\rm vert}(v) \equiv i \, (\mbox{mod } 2), \\ d_{\rm hori}(v) \le n-2-i, & d_{\rm hori}(v) \equiv n-2-i (\mbox{mod } 2) \end{array} \right. \right\}.
$$
See Figure \ref{Fig;mutation_sequence_for_flip}; so each $V^{(i)}_{\lambda;e}$ can be viewed as forming the grid points of a rectangle, consisting of $(i+1)(n-i-1)$ points. Notice that these sets for different $i$ are not necessarily disjoint with each other. For example, if $n\ge 4$, then $V_{\lambda;e}^{(0)}$ and $V_{\lambda;e}^{(2)}$ have $n-3$ vertices in common.

\begin{figure}[h]
    \centering
    \scalebox{1.0}{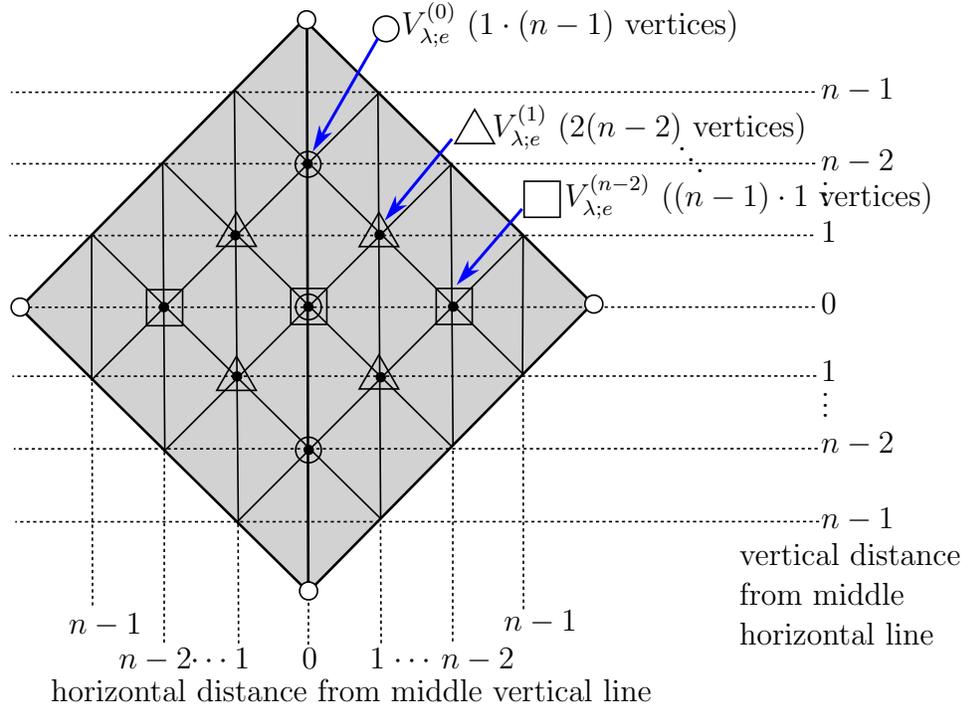}
    \caption{Vertices of the $n$-triangulation quiver involved in the flip of triangulations; the above example picture is for the case $n=4$}\label{Fig;mutation_sequence_for_flip}
\end{figure}

The sought-for mutation sequence of Fock, Goncharov and Shen consists of first mutations at all vertices of $V_{\lambda;e}^{(0)}$ in any order, then mutations at all vertices of $V_{\lambda;e}^{(1)}$ in any order, etc., then lastly mutations at all vertices of $V_{\lambda;e}^{(n-2)}$ in any order. Regardless of the orders chosen, this mutation sequence connects the coordinate chart $\mathcal{D}_\lambda$ of $\mathscr{P}_{{\rm PGL}_n,\fS}$ for $\lambda$ to the chart $\mathcal{D}_{\lambda'}$ for $\lambda'$. In particular, the length of this mutation sequence is $\sum_{i=0}^{n-2} (i+1)(n-1-i) = \sum_{j=1}^{n-1} j(n-j) = \frac{1}{6}(n^3-n) =:r$. Let us denote the corresponding sequence of vertices as
$$
v_1,v_2,\ldots,v_r.
$$
So, the first $n-1$ of them would be the elements of $V^{(0)}_{\lambda;e}$, and we have
\begin{align}
\label{two_seeds_connected_by_sequence_of_mutations}
\mathcal{D}_{\lambda'} = \mu_{v_r} \cdots \mu_{v_2} \mu_{v_1} (\mathcal{D}_\lambda)    
\end{align}
and in particular,
$$
\Gamma_{\lambda'} = \mu_{v_r} \cdots \mu_{v_2} \mu_{v_1} (\Gamma_\lambda).
$$

At the quantum level, we first define the quantum coordinate change isomorphism for the skew-fields of fractions of the usual Fock-Goncharov algebras
\begin{equation}\label{eq-Theta1}
\Phi_{
\lambda \lambda'}^q:=\mu_{v_1}^{q}\circ\cdots\circ\mu_{v_r}^{q}\colon\Fr(\mathcal X_q(\fS,\lambda'))\rightarrow 
\Fr(\mathcal X_q(\fS,\lambda)).
\end{equation}
Note that the order of the mutation sequence seems to be reversed because of the contravariant nature of the classical and quantum coordinate change maps. We then define the $n$-\red{th} root \red{mutable-}balanced version of the quantum coordinate change isomorphism to be
\begin{equation}\label{eq-Theta2}
\Theta_{
\lambda \lambda'}^{\omega}:=\nu_{v_1}^{\omega}\circ\cdots\circ\nu_{v_r}^{\omega}\colon\Fr(\mathcal Z_\omega^{\rm mbl}(\fS,\lambda'))\rightarrow 
\Fr(\mathcal Z_\omega^{\rm mbl}(\fS,\lambda)).
\end{equation}

Suppose now that $\lambda$ and $\lambda'$ are any two triangulations of $\fS$. A {\bf triangulation sweep} connecting 
$\lambda$ and $\lambda'$ is a sequence of triangulations
$\Lambda=(\lambda_1,\cdots,\lambda_m)$ such that 
$\lambda_1=\lambda$, $\lambda_m=\lambda'$, and $\lambda_{i+1}$ is obtained from $\lambda_i$ by a flip for each $1\leq i\leq m-1$. It is well known that for any $\lambda$ and $\lambda'$, there exists a triangulation sweep $\Lambda$ connecting $\lambda$ and $\lambda'$ (\cite{Lab09}). For any such triangulation sweep $\Lambda$, we define the corresponding quantum coordinate change isomorphisms
\begin{equation}\label{eq-Theta-change1}
\Phi_\Lambda^q :=\Phi_{
\lambda_1\lambda_2}^q\circ\cdots\circ\Phi_{
\lambda_{m-1}\lambda_m}^q\colon\Fr(
\mathcal{X}_q(\fS,\lambda'))\rightarrow 
\Fr(
\mathcal{X}_q(\fS,\lambda)),
\end{equation}
\begin{equation}\label{eq-Theta-change2}
\Theta_{\Lambda}^{\omega}  :=\Theta_{
\lambda_1\lambda_2}^{\omega}\circ\cdots\circ\Theta_{
\lambda_{m-1}\lambda_m}^{\omega}\colon\Fr(
\mathcal{Z}^{\rm mbl}_\omega(\fS,\lambda'))\rightarrow 
\Fr(
\mathcal{Z}^{\rm mbl}_\omega(\fS,\lambda)).
\end{equation}

\begin{proposition} [\cite{BZ,FG09a,KN}; see also \cite{Kim_phase,Kim21}]
    $\Phi^q_\Lambda$ depends only on $\lambda$ and $\lambda'$.
\end{proposition}

In order to prove this proposition, one should show that any relation satisfied by classical coordinate change maps for flips, say by the maps $\Phi^1_{\lambda \lambda'}$, such as the pentagon relation, is also satisfied by the quantum counterparts $\Phi^q_{\lambda\lambda'}$. It suffices to show that any relation satisfied by classical mutation maps $\mu^1_k$ is satisfied by the quantum counterparts $\mu^q_k$. One could use the decomposition $\mu^q_k = \mu^{\sharp q}_k \circ \mu'_k$ for each quantum mutation $\mu^q_k$. Then, for a sought-for equation among $\mu^q_k$, one can move all $\mu'_k$ to the left (or right), with the help of \eqref{eq-quantum-mutation}. By also exploiting a slight variant of the decomposition $\mu^q_k = \mu^{\sharp q}_k \circ \mu'_k$ (a `signed' decomposition), one can ensure that the problem boils down to proving an equation among $\mu'_k$'s only, and an equation among $\mu^{\sharp q}_k$'s only. We refer the readers to \cite{KN,Kim_phase} how this argument works. Essentially the same argument works for the balanced $n$-th root quantum mutations, yielding the following.

\begin{proposition}
\label{prop:Theta_omega_consistency}
    $\Theta^\omega_\Lambda$ depends only on $\lambda$ and $\lambda'$. \qed
\end{proposition}

Therefore, when we do not have to keep track of a specific triangulation sweep, we can write
$$
\Phi^q_{\lambda\lambda'} = \Phi^q_\Lambda \quad\mbox{and}\quad \Theta^\omega_{\lambda\lambda'} = \Theta^\omega_\Lambda.
$$
One obvious 
\red{observation}: 
\begin{lemma}\label{lem-extension}
For any two triangulations $\lambda$ and $\lambda'$, the \red{mutable-}balanced $n$-th root quantum coordinate change map $\Theta^\omega_{\lambda\lambda'}$ extends the quantum coordinate change map $\Phi^q_{\lambda\lambda'}$.
\end{lemma}

\subsection{\red{Compatibility of mutable-balanced quantum coordinate change map and splitting homomorphisms}}

A crucial property we shall develop and use about the \red{mutable-}balanced quantum coordinate change maps \red{$\Theta^\omega_{\lambda\lambda'}$ constructed in the previous subsection} is the compatibility with cutting and gluing of a pb surface along an ideal arc \red{(\S\ref{sub-splitting})}.

\def\cut{\mathsf{Cut}}
\def\pr{{\bf pr}}

\red{L}et $\fS$ be a triangulable pb surface with a triangulation $\lambda$. 
In principle, one can cut the surface $\fS$ along any ideal arc, but here let us assume that $e$ is a non-boundary edge of the triangulation $\lambda$. As in \S\ref{sub-splitting}, we denote by  
$\cut_e(\fS)$ 
the pb surface obtained from $\fS$ by cutting along edge $e$, and denote the projection map by
$$
\pr_e : \cut_e(\fS) \to \fS,
$$
so that $\pr_e^{-1}({e})$ consists of two ideal arcs $e',e''$ of $\cut_e\fS$. In \S\ref{sub-splitting} we studied the induced splitting homomorphism $\mathbb{S}_e$ between the stated ${\rm SL}_n$-skein algebras and that between the reduced stated ${\rm SL}_n$-skein algebras; here we will need the latter
$$
\mathbb{S}_e : \rS(\fS) \to \rS(\cut_e(\fS)).
$$
We further need to investigate a `splitting map' between the ($n$-th root)  Fock-Goncharov algebras. Note that  
$\lambda_e = (\lambda\setminus\{e\})\cup\{e',e''\}$ is a triangulation of
$\cut_e(\fS)$. We say $\lambda_e$ is induced from $\lambda$.
For a small
vertex $v$ in $e$, i.e. $v\in V_\lambda$ lying in $e$, we have $\pr^{-1}_{e}(v) = \{v',v''\}$, where $v'$ and $v''$ are small vertices of \red{$\lambda_e$}. 
There is an algebra embedding
\begin{align}\label{eq-splitting-torus}
    \mathcal S_e
    : \mathcal{Z}_\omega(\fS,\lambda)\rightarrow
    \mathcal{Z}_\omega(\cut_e(\fS),\lambda_e)
\end{align}
defined on the generators $Z_v$, $v\in V_\lambda$, by
\begin{align}
\label{cutting_homomorphism_for_Z_omega}
    \mathcal S_e
    (
    Z_v)=
    \begin{cases}
        Z_v & \text{ if $v$ is not contained in $e$},\\
        [
        Z_{v'} Z_{v''}] & \text{ if $v$ is contained in $e$ and 
        $\pr^{-1}_{e}(v) = \{v',v''\}$},
    \end{cases}
\end{align}
where $[\sim]$ is the Weyl-ordered product \eqref{Weyl_ordering_Z}; since we are not allowing self-folded triangles in this paper, in fact $Z_{v'}$ commutes with $Z_{v''}$ in the second case, and hence $\mathcal{S}_e(Z_v) = Z_{v'}Z_{v''}$. \red{Here, we are using the fact that ${\bf pr}_e$ naturally gives a map between the vertex sets
\begin{align}
\label{eq:pr_for_vertices}
{\bf pr}_e : V_{\lambda_e} \to V_\lambda    
\end{align}
of the quivers. For each vertex $w\in V_\lambda$ such that ${\bf pr}_e^{-1}(w)$ consists of one vertex of $V_{\lambda_e}$, we are denoting the unique preimage vertex also by the same symbol $w$. In particular, the symbol $Z_w$ may mean an element of $\mathcal{Z}_\omega(\fS,\lambda)$ or of $\mathcal{Z}_\omega(\cut_e(\fS),\lambda_e)$.}

\red{We recall} the compatibility of the quantum trace map $\overline{\rm tr}_\lambda 
   $ (Theorem \ref{thm.quantum_trace}) with the splitting homomorphisms.

\begin{theorem}[\cite{LY23}]\label{thm-trace-cut}
    The following diagram commutes
    \begin{equation*}
\begin{tikzcd}
\rdS \arrow[r, "\mathbb S_e"]
\arrow[d, "\tr"]  
&  \overline{\cS}_{\omega}(\cut_e(\fS)) \arrow[d, "\overline{\rm tr}_{\lambda_e}"] \\
 \mathcal{Z}_\omega(\fS,\lambda)
 \arrow[r, "\mathcal S_e"] 
&  
\mathcal{Z}_\omega(\cut_e(\fS),\lambda_e),
\end{tikzcd}
\end{equation*}
    where $\mathbb S_e$ is the splitting homomorphism defined in subsection \ref{sub-splitting}.
\end{theorem}

\red{One can further restrict the bottom horizontal arrow of the above diagram to the mutable-balanced subalgebras (eq.\eqref{eq.Z_algebras_for_lambda} and Definition \ref{def.mbl}), by \eqref{eq.tr_image_in_bl} of Theorem \ref{thm.quantum_trace}, Lemma \ref{lem:bl_in_mbl}, and the following lemma. It is straightforward to show this lemma using Definition \ref{def.mbl} and \eqref{cutting_homomorphism_for_Z_omega}. We omit its proof here, especially because we will show a more general version of this lemma shortly (we do this in the proof of Lemma~\ref{lem-mutation-cutiing}).}

\begin{lemma}\label{lem:S_e_restricts_to_mbl}
    The algebra embedding $\mathcal S_e\colon 
    \mathcal{Z}_\omega(\fS,\lambda)\rightarrow 
    \mathcal{Z}_\omega(\cut_e(\fS),\lambda_e)$ restricts to an algebra embedding from $
    \mathcal{Z}^{\rm mbl}_\omega(\fS,\lambda)$ to 
$\mathcal{Z}^{\rm mbl}_\omega(\cut_e(\fS),\lambda_e)$. 
\end{lemma}

Suppose that $\lambda$ and $\lambda'$ are two triangulations  
and that $\lambda'$ is obtained from $\lambda$ by performing a flip at an edge $e_1$. 
Let $e \in \lambda \setminus\{e_1\}$.
Suppose that $\lambda_e$ (resp. $\lambda'_e$) is the triangulation of $\cut_e(\fS)$ induced by $\lambda$ (resp. $\lambda'$). 
Then $\lambda_e'$ is obtained from $\lambda_e$ by performing a flip at $e_1$.  
Let $v_1,\ldots,v_r$ be a sequence of mutable vertices of $\Gamma_\lambda$ such that $\mathcal{D}_{\lambda'} = \mu_{v_r} \cdots \mu_{v_1}(\mathcal{D}_\lambda)$, in the sense of \eqref{two_seeds_connected_by_sequence_of_mutations}.
Since $e\neq e_1$, 
each $v_i = {\bf pr}_e^{-1}(v_i) \in V_{\lambda_e}$ is
contained in the interior of $\cut_e(\fS)$.
We claim that \begin{align}
\label{mutation_sequence_of_cut_surface}
\mathcal{D}_{\lambda'_e} = \mu_{v_r} \cdots \mu_{v_1} (\mathcal{D}_{\lambda_e})
\end{align} 
holds. Consider the subsurface of $\fS$ formed by the two ideal triangles of $\lambda$ having $e_1$ as a side. Say that the three edges of one triangle are $e_1,e_2,e_3$ in the clockwise order, while those of the other are $e_1,e_4,e_5$ in the clockwise order. If $e$ appears at most once 
\red{among the edges} $e_2,e_3,e_4,e_5$, then \eqref{mutation_sequence_of_cut_surface} is easy to see. In case $e=e_2=e_4$ or $e=e_3=e_5$, then one has to be careful about variables associated to vertices lying in $e$ when checking \eqref{mutation_sequence_of_cut_surface}, but one can still check the equality. The case $e_2=e_3$ or $e_4=e_5$ is excluded, because then $\lambda$ has a self-folded triangle. The case $e_2=e_5$ or $e_3=e_4$ is excluded, because then $\lambda'$ has a self-folded triangle. We thus have \eqref{mutation_sequence_of_cut_surface} indeed.

The following lemma can be viewed as a balanced $n$-th root quantum version of \eqref{mutation_sequence_of_cut_surface}. See \cite[Prop.3.34]{Kim21} for the case when $n=3$, the proof of which we closely follow in our proof below.

\begin{lemma}[\red{compatibility of mutable-balanced quantum coordinate change map and splitting homomorphisms}]\label{lem-mutation-cutiing}
    The following diagram commutes
    \begin{equation*}
\begin{tikzcd}
\Fr(
\mathcal{Z}^{\rm mbl}_\omega(\fS,\lambda')) \arrow[r, "\mathcal S_e"]
\arrow[d, "\Theta^{\omega}_{
\lambda \lambda'}"]  
& \Fr(
\mathcal{Z}^{\rm mbl}_\omega(\cut_e(\fS),\lambda_e')) \arrow[d, "\Theta^{\omega}_{
\lambda_e \lambda_e'}"] \\
 \Fr(
 \mathcal{Z}^{\rm mbl}_\omega(\fS,\lambda))
 \arrow[r, "\mathcal S_e"] 
&  \Fr(
\mathcal{Z}^{\rm mbl}_\omega(\cut_e(\fS),\lambda_e)).
\end{tikzcd}
\end{equation*}
\end{lemma}

\begin{proof}
    Let $\mathcal{D}_0 := \mathcal{D}_\lambda$, and for each $\red{\ell}=1,\ldots,r$ let $\mathcal{D}_\red{\ell} = \mu_{v_\red{\ell}}(\mathcal{D}_{\red{\ell}-1})$ be the $\red{\ell}$-th seed in the sequence of seeds connecting $\mathcal{D}_\lambda$ and $\mathcal{D}_{\lambda'}$. In particular, $\mathcal{D}_r =\mathcal{D}_{\lambda'}$. Likewise for the cut surface $\cut_e(\fS)$, let $\mathcal{D}_{0,e} := \mathcal{D}_{\lambda_e}$, and $\mathcal{D}_{\red{\ell},e} := \mu_{v_\red{\ell}}(\mathcal{D}_{\red{\ell}-1,e})$ for $\red{\ell}=1,\ldots,r$. For each $\red{0}\le \red{\ell}\le r$ we have the algebra embedding
  $$
  \mathcal{S}_e : \mathcal{Z}_\omega(\mathcal{D}_\red{\ell}) \to \mathcal{Z}_\omega(\mathcal{D}_{\red{\ell},e})
$$
from \eqref{eq-splitting-torus}. 

We first claim that this restricts to an algebra embedding between the mutable-balanced subalgebras
\begin{align}
\label{mathcal_S_e_for_mbl}
    \mathcal{S}_e : \mathcal{Z}_\omega^{\rm mbl} (\mathcal{D}_\red{\ell}) \to \mathcal{Z}_\omega^{\rm mbl}(\mathcal{D}_{\red{\ell},e}).
\end{align}
\red{This in particular proves Lemma \ref{lem:S_e_restricts_to_mbl}. By Definition \ref{def.mbl}, $\mathcal{Z}_\omega^{\rm mbl} (\mathcal{D}_\red{\ell})$ is spanned by Laurent monomials of the form $Z^{\bf t}$, with ${\bf t}=(t_v)_{v\in \mathcal{V}} \in \mathbb{Z}^\mathcal{V}$ lying in $\mathcal{B}_{\mathcal{D}_\ell}$. Here, ${\bf t} \in \mathcal{B}_{\mathcal{D}_\ell}$ means}  
that
\begin{align}
\label{eq:B_D_ell_condition}
\sum_{v\in \mathcal{V}} Q_\red{\ell}(u,v) t_v \in n\mathbb{Z}    
\end{align}
\red{holds} for all mutable vertices $u$ of $\mathcal{V}$, i.e. vertices of $V_\lambda$ not lying on the boundary of $\fS$, where $Q_\red{\ell}$ is the exchange matrix of the seed $\mathcal{D}_\red{\ell}$. Then\red{, in view of \eqref{cutting_homomorphism_for_Z_omega} we have}
\begin{align}
    \label{mathcal_S_e_on_Laurent_monomial}
    \mathcal{S}_e(Z^{\bf t}) = Z^{{\bf t}'}
\end{align}
\red{for} ${\bf t}' = (t'_v)_{v\in \mathcal{V}_e} \in \mathbb{Z}^{\mathcal{V}_e}$, where $\mathcal{V}_e = V_{\lambda_e}$\red{, given as follows:} if $v\in \mathcal{V}$ does not lie on $e$ so that ${\bf pr}_e^{-1}(v) = \{v\}$ then $t'_v = t_v$, and if $v\in \mathcal{V}$ lies on $e$ so that ${\bf pr}_e^{-1}(v) = \{v',v''\}$ then $t'_{v'} = t'_{v''} = t_v$. \red{A priori, \eqref{cutting_homomorphism_for_Z_omega} yields $\mathcal{S}_e(Z^{\bf t}) = \omega^{m} Z^{{\bf t}'}$ for some $m\in \frac{1}{2} \mathbb{Z}$, and it is not hard to verify that $m=0$.
}

\red{To prove our claim about well-definedness of \eqref{mathcal_S_e_for_mbl}, it remains to show that $Z^{\mathbf{t}'}$ belongs to $\mathcal{Z}^{\mathrm{mbl}}_\omega(\mathcal{D}_{\ell,e})$. By Definition \ref{def.mbl}, we should show that ${\bf t}'$ lies in $\mathcal{B}_{\mathcal{D}_{\ell,e}}$, i.e.} 
$$
\sum_{\red{w} \in \mathcal{V}_e} Q_{\red{\ell},e}(u,\red{w}) t'_\red{w} \in n\mathbb{Z}
$$
\red{holds for every mutable vertex $u$ of $\mathcal{V}_e$,} where $Q_{\red{\ell},e}$ is the exchange matrix of $\mathcal{D}_{\red{\ell},e}$. Here, $u\red{\in \mathcal{V}_e}$ does not lie in the boundary of $\cut_e(\fS)$, hence ${\bf pr}_e^{-1}(u)=\{u\}$\red{,} where $u$\red{, when viewed as a vertex of $\mathcal{V} = V_\lambda$,} is a mutable vertex of $\mathcal{V} = V_\lambda$ \red{which does} not \red{lie} on $e$. 
\red{Observe f}or each $v\in \mathcal{V}$ that
\begin{align}
    \label{Q_after_cutting}
    Q_{\red{\ell}}(u,v) = \sum_{w \in {\bf pr}_e^{-1}(v)} Q_{\red{\ell},e}(u,w).
\end{align}
Since ${\bf pr}_e : \mathcal{V}_e \to \mathcal{V}$ is surjective, and since $t'_w = t_{{\bf pr}_e(w)}$ for each $w\in \mathcal{V}_e$, it follows that 
\begin{align}
\label{Q-linear_combi_after_cutting}
\sum_{v\in \mathcal{V}} Q_\red{\ell}(u,v)t_v = \sum_{w\in \mathcal{V}_e} Q_{\red{\ell},e}(u,w) t_w'
\end{align}
\red{holds; namely, we multiplied by $t_v=t'_w$ (for $w\in {\bf p}_e^{-1}(v)$) on both sides of \eqref{Q_after_cutting} and summed over all $v\in \mathcal{V}$.} 
The left hand side \red{of \eqref{Q-linear_combi_after_cutting}} belongs to $n\mathbb{Z}$ \red{due to \eqref{eq:B_D_ell_condition}}, \red{and hence} it follows that the right hand side \red{of \eqref{Q-linear_combi_after_cutting} also belongs to $n\mathbb{Z}$}, as desired. \red{So, indeed \eqref{mathcal_S_e_for_mbl} is well-defined.}

We still denote by $\mathcal{S}_e$ the map between the skew-fields of fractions ${\rm Frac}(\mathcal{Z}^{\rm mbl}_\omega(\mathcal{D}_\red{\ell})) \to {\rm Frac}(\mathcal{Z}^{\rm mbl}_\omega(\mathcal{D}_{\red{\ell},e}))$ induced by $\mathcal{S}_e : \mathcal{Z}^{\rm mbl}_\omega(\mathcal{D}_\red{\ell}) \to \mathcal{Z}^{\rm mbl}_\omega(\mathcal{D}_{\red{\ell},e})$. We now claim that the following diagram commutes:
\begin{align}\label{eq-com-S-vi}
\raisebox{7mm}{\xymatrix{
{\rm Frac}(\mathcal{Z}^{\rm mbl}_\omega(\mathcal{D}_{\red{\ell}})) \ar[r]^{\mathcal{S}_e} \ar[d]_{\nu^\omega_{v_\red{\ell}}} & {\rm Frac}(\mathcal{Z}^{\rm mbl}_\omega(\mathcal{D}_{\red{\ell},e})) \ar[d]^{\nu^\omega_{v_\red{\ell}}} \\
{\rm Frac}(\mathcal{Z}^{\rm mbl}_\omega(\mathcal{D}_{\red{\ell}-1})) \ar[r]^{\mathcal{S}_e} & {\rm Frac}(\mathcal{Z}^{\rm mbl}_\omega(\mathcal{D}_{\red{\ell}-1,e})).
}}
\end{align}
To show that the diagram in \eqref{eq-com-S-vi} commutes, it suffices to show that the following two diagrams commute:
\begin{align}\label{eq-com-S-vi1}
\raisebox{7mm}{\xymatrix{
{\rm Frac}(\mathcal{Z}_\omega(\mathcal{D}_{\red{\ell}})) \ar[r]^{\mathcal{S}_e} \ar[d]_{\nu_{v_\red{\ell}}'} & {\rm Frac}(\mathcal{Z}_\omega(\mathcal{D}_{\red{\ell},e})) \ar[d]_{\nu_{v_\red{\ell}}'} \\
{\rm Frac}(\mathcal{Z}_\omega(\mathcal{D}_{\red{\ell}-1})) \ar[r]^{\mathcal{S}_e} & {\rm Frac}(\mathcal{Z}_\omega(\mathcal{D}_{\red{\ell}-1,e}))
}}
\end{align}
and 
\begin{align}\label{eq-com-S-vi2}
\raisebox{7mm}{\xymatrix{
{\rm Frac}(\mathcal{Z}^{\rm mbl}_\omega(\mathcal{D}_{\red{\ell}-1})) \ar[r]^{\mathcal{S}_e} \ar[d]_{\nu^{\sharp\omega}_{v_\red{\ell}}} & {\rm Frac}(\mathcal{Z}^{\rm mbl}_\omega(\mathcal{D}_{\red{\ell}-1,e})) \ar[d]^{\nu^{\sharp\omega}_{v_\red{\ell}}} \\
{\rm Frac}(\mathcal{Z}^{\rm mbl}_\omega(\mathcal{D}_{\red{\ell}-1})) \ar[r]^{\mathcal{S}_e} & {\rm Frac}(\mathcal{Z}^{\rm mbl}_\omega(\mathcal{D}_{\red{\ell}-1,e})).
}}
\end{align}
Note that the first diagram restricts to the maps between the mutable-balanced subalgebras (or the skew-fields of fractions thereof), since we showed that both $\nu'_{v_\red{\ell}}$ and $\mathcal{S}_e$   restrict to the mutable-balanced subalgebras.

\red{We first show that \eqref{eq-com-S-vi1} commutes.} Let us temporarily denote the generators of $\mathcal{Z}_\omega(\mathcal{D}_{\red{\ell}})$ or $\mathcal{Z}_\omega(\mathcal{D}_{\red{\ell},e})$ by $Z_v'$, and those of $\mathcal{Z}_\omega(\mathcal{D}_{\red{\ell}-1})$ or $\mathcal{Z}_\omega(\mathcal{D}_{\red{\ell}-1,e})$ by $Z_v$.
For any $v\in\mathcal V$ such that $v$ is not contained in $e$, we have ${\bf pr}_e^{-1}(v)=\{v\}$ (also ${\bf pr}_e^{-1}(v_\red{\ell})=\{v_\red{\ell}\}$), so $Q_\red{\ell}(v,v_\red{\ell})=Q_{\red{\ell},e}(v,v_\red{\ell})$ 
 (cf.\eqref{Q_after_cutting}), $\mathcal{S}_e(Z_v') = Z_v'$ and $\mathcal{S}_e(Z_{v_\red{\ell}}')=Z_{v_\red{\ell}}'$ (\eqref{cutting_homomorphism_for_Z_omega}). In view of \eqref{eq-quantum-mutation_Z}, we can observe
\begin{align*}
    \mathcal S_e(\nu_{v_\red{\ell}}'(Z_v'))  = \nu_{v_\red{\ell}}'(\mathcal S_e(Z_v')).
\end{align*}
Now let $v\in \mathcal{V}$ be a vertex that lies in $e$; then
$(\textbf{pr}_e)^{-1}(v) = \{v',v''\}$, for some distinct $v',v'' \in \mathcal{V}_e$. From \eqref{cutting_homomorphism_for_Z_omega} and \eqref{eq-quantum-mutation_Z}, together with Lemma \ref{lem.nu_k_prime_preserves_Weyl-ordering}, we
have 
\begin{align*}
    \nu_{v_\red{\ell}}'(\mathcal S_e(Z_v'))
    =&\nu_{v_\red{\ell}}'([Z_{v'}' Z_{v''}'])
    =[Z_{v'} Z_{v_\red{\ell}}^{[Q_{\red{\ell}-1,e}(v',v_\red{\ell})]_{+}} Z_{v''} Z_{v_\red{\ell}}^{[Q_{\red{\ell}-1,e}(v'',v_\red{\ell})]_{+}}],
     \\
    \mathcal S_e(\nu_{v_\red{\ell}}'(Z_v'))
    =&\mathcal S_e([Z_{v} Z_{v_\red{\ell}}^{[Q_{\red{\ell}-1}(v,v_\red{\ell})]_{+}}])
    =[Z_{v'} Z_{v''} Z_{v_\red{\ell}}^{[Q_{\red{\ell}-1}(v,v_\red{\ell})]_{+}}].
\end{align*} 
From \eqref{Q_after_cutting} one has
\begin{align*}
    Q_{\red{\ell}-1}(v,v_\red{\ell})
    =Q_{\red{\ell}-1,e}(v',v_\red{\ell}) + Q_{\red{\ell}-1,e}(v'',v_\red{\ell}).
\end{align*}
When $n>2$, we claim that at least one of $Q_{\red{\ell}-1,e}(v',v_\red{\ell}) $ and $Q_{\red{\ell}-1,e}(v'',v_\red{\ell})$ is zero. As observed in the proof of \cite[Prop.3.34]{Kim21}, the only possible case when both are nonzero is when $e$ and 
$e_1$ shares a puncture in the interior of the surface, such that the valence of $\lambda$ at this puncture is $2$. However, then flipping at $e$ would yield a self-folded triangle (in $\lambda'$), so this case is excluded. 
When $n=2$, it is easy to check that $Q_{\red{\ell}-1,e}(v',v_\red{\ell})Q_{\red{\ell}-1,e}(v'',v_\red{\ell})\geq 0$. 
This shows 
$[Q_{\red{\ell}-1}(v,v_\red{\ell})]_{+}
    =[Q_{\red{\ell}-1,e}(v',v_\red{\ell})]_{+} + [Q_{\red{\ell}-1,e}(v'',v_\red{\ell})]_{+}$
and therefore
$\nu_{v_\red{\ell}}'(\mathcal S_e(Z_v'))= \mathcal S_e(\nu_{v_\red{\ell}}'(Z_v')).$
Thus the diagram in \eqref{eq-com-S-vi1} commutes.

\red{To show that \eqref{eq-com-S-vi2} commutes, we consider a generating Laurent monomial of $\mathcal{Z}^{\mathrm{mbl}}_\omega(\mathcal{D}_{\ell-1})$, namely $Z^{\bf t}$ with} ${\bf t} = (t_v)_{v\in \mathcal{V}} \in \mathcal{B}_{\mathcal{D}_{\red{\ell}-1}}$. \red{From \eqref{mathcal_S_e_on_Laurent_monomial} and the well-definedness of \eqref{mathcal_S_e_for_mbl}, we have} $\mathcal S_e(Z^{\bf t}) = Z^{{\bf t}'}$, where ${\bf t}'=(t_v')_{v\in \mathcal V_e}\in \mathcal B_{\mathcal D_{\red{\ell}-1,e}}$.
\red{Thus} Lemma \ref{lem.nu_sharp_well-defined} \red{applied to the seeds $\mathcal{D}_{\ell-1,e}$ and $\mathcal{D}_{\ell-1}$ respectively} implies that
\begin{align*}
    \nu^{\sharp \omega}_{v_\red{\ell}}(\mathcal S_e(Z^{\bf t})) & \red{=\nu^{\sharp \omega}_{v_\red{\ell}}(Z^{{\bf t}'}) } = Z^{{\bf t}'} \, F^q(X_{v_\red{\ell}}, 
    \textstyle{ \frac{1}{n} } {\textstyle \sum}_{v\in \mathcal{V}_e} Q_{\red{\ell}-1,e}(v_\red{\ell},v) t_v'),\\
    \mathcal S_e(\nu^{\sharp \omega}_{v_\red{\ell}}(Z^{\bf t}))&=\mathcal S_e(Z^{\bf t} \, F^q(X_{v_\red{\ell}}, {\textstyle \frac{1}{n}} {\textstyle \sum}_{v\in \mathcal{V}} Q_{\red{\ell}-1}(v_\red{\ell},v) t_v))
    =Z^{{\bf t}'}\, F^q(X_{v_\red{\ell}}, {\textstyle \frac{1}{n}} {\textstyle \sum}_{v\in \mathcal{V}} Q_{\red{\ell}-1}(v_\red{\ell},v) t_v).
\end{align*}
From \eqref{Q-linear_combi_after_cutting} we have
$$\sum_{v\in \mathcal{V}_e} Q_{\red{\ell}-1,e}(v_\red{\ell},v) t_v'=
\sum_{v\in \mathcal{V}} Q_{\red{\ell}-1}(v_\red{\ell},v) t_v.
$$
Thus we have $\nu^{\sharp \omega}_{v_\red{\ell}}(\mathcal S_e(Z^{\bf t})) = \mathcal S_e(\nu^{\sharp \omega}_{v_\red{\ell}}(Z^{\bf t}))$, and hence 
the diagram in \eqref{eq-com-S-vi2} commutes, finishing the proof. 
\end{proof}

\def\wl{\widetilde\lambda}

\section{
\red{Naturality of} $\SL$-quantum trace maps 
\red{under} 
coordinate change isomorphisms}\label{sec:compatibility}

\red{We formulate and prove our main result. Theorem \ref{thm-main-compatibility}, the mutable-balanced version of the main theorem, asserts that the L\^e-Yu quantum trace maps $\tr : \rdS \to \mathcal{Z}^{\mathrm{bl}}_\omega(\fS,\lambda) \subset \mathcal{Z}^{\mathrm{mbl}}_\omega(\fS,\lambda)$ \cite{LY23} (Theorem \ref{thm.quantum_trace}) for different triangulations $\lambda$ of a given pb surface $\fS$ are compatible with each other via the mutable-balanced quantum coordinate change maps $\Theta^\omega_{\lambda\lambda'} : \mathrm{Frac}(\mathcal{Z}^{\mathrm{mbl}}_\omega(\fS,\lambda')) \to \mathrm{Frac}(\mathcal{Z}^{\mathrm{mbl}}_\omega(\fS,\lambda))$ which we constructed as compositions of mutable-balanced quantum cluster mutations in the previous section. We use the compatibility of $\Theta^\omega_{\lambda\lambda'}$ with the splitting homomorphisms established in Lemma \ref{lem-mutation-cutiing} to boil the situation down to the case when $\fS$ is an ideal quadrilateral. We prove the statement for the quadrilateral by exploiting the treatment by Schrader and Shapiro \cite{SS17} using suitable networks, which lets us avoid any serious computation. We also investigate the restriction from the mutable-balanced part to the balanced part, giving Theorem \ref{thm-main-compatibility_bl} which is the balanced version of the main theorem, and then make a direct contact with L\^e and Yu's `tautological' quantum coordinate change map \cite{LY23} which a priori lacks a cluster description (Theorem \ref{thm-main-comparision}).}

\subsection{
\red{Naturality} of the quantum trace map with the \red{mutable-}balanced quantum coordinate change maps $\Theta^\omega_{\lambda\lambda'}$}\label{subsec.compatibility_Theta}

Let $\fS$ be a triangulable pb surface. For each triangulation $\lambda$ of $\fS$, we have a quantum trace homomorphism
$$
\tr : \rdS \to \mathcal{Z}_\omega^{\rm bl}(\fS,\lambda) \subset \mathcal{Z}_\omega^{\rm mbl} (\fS,\lambda) \subset \mathcal{Z}_\omega(\fS,\lambda),
$$
from Theorem \ref{thm.quantum_trace}; see Lemma \ref{lem:bl_in_mbl} for $\mathcal{Z}_\omega^{\rm bl}(\fS,\lambda) \subset \mathcal{Z}_\omega^{\rm mbl} (\fS,\lambda)$. Meanwhile, for any pair of triangulations $\lambda$ and $\lambda'$, we developed the \red{mutable-}balanced $n$-th root quantum coordinate change isomorphism
$$
\Theta^\omega_{\lambda\lambda'} : {\rm Frac}(\mathcal{Z}^{\rm mbl}_\omega(\fS,\lambda')) \to {\rm Frac}(\mathcal{Z}^{\rm mbl}_\omega(\fS,\lambda))
$$
in the previous section, which extends the previously known quantum coordinate change isomorphism $\Phi^q_{\lambda\lambda'} : {\rm Frac}(\mathcal{X}_q(\fS,\lambda')) \to {\rm Frac}(\mathcal{X}_q(\fS,\lambda))$ coming from the theory of quantum cluster algebras. 

In the present subsection \red{and the next one,} we show that the quantum trace maps for different triangulations are compatible with each other via the \red{mutable-balanced} isomorphisms $\Theta^\omega_{\lambda\lambda'}$. This 
\red{mutable-balanced version of the naturality}, formulated as the following theorem, constitutes the core part of the main result of the present paper.
\begin{theorem}[
\red{naturality} of quantum trace 
\red{under} mutable-balanced $n$-root quantum coordinate change maps $\Theta^\omega_{\lambda\lambda'}$]\label{thm-main-compatibility}
    Let $\fS$ be a triangulable pb surface. For any two triangulations $\lambda$ and $\lambda'$ of $\fS$, the following diagram commutes:
    \begin{align}
        \label{eq-compability-tr-mutation}
        \raisebox{7mm}{\xymatrix{
        & \rdS \ar[dl]_{\overline{\rm tr}_{\lambda'}} \ar[dr]^{\tr} & \\
        {\rm Frac}(\mathcal{Z}^{\rm mbl}_\omega(\fS,\lambda')) \ar[rr]_-{\Theta^\omega_{\lambda\lambda'}} & & {\rm Frac}(\mathcal{Z}^{\rm mbl}_\omega(\fS,\lambda)).}}
    \end{align}
    That is, we have
    \begin{align}
    \label{compatibility_equation}
        \tr = \Theta^\omega_{\lambda\lambda'} \circ \overline{\rm tr}_{\lambda'}.
    \end{align}
\end{theorem}
Here is how we break down the proof. First, from Proposition \ref{prop:Theta_omega_consistency} we immediately obtain the consistency among the \red{mutable-balanced} quantum coordinate change isomorphisms; that is, for any triangulations $\lambda$, $\lambda'$ and $\lambda''$, we have
\begin{align}
\label{Theta_omega_lambda_lambda_prime_consistency}
\Theta^\omega_{\lambda\lambda'} \circ \Theta^\omega_{\lambda'\lambda''} = \Theta^\omega_{\lambda\lambda''}.
\end{align}
Now, since any two triangulations are connected by a sequence of flips, it suffices to prove Theorem \ref{thm-main-compatibility} for the case when $\lambda$ and $\lambda'$ are related by a single flip at an edge, say $e_1$. 

There are two triangles of $\lambda$ having $e_1$ as a side. Let one of them have sides $e_1,e_2,e_3$ in the clockwise order, and the other have sides $e_1,e_4,e_5$ in the clockwise order. Cutting the surface $\fS$ along all the edges in the set $\{e_2,e_3,e_4,e_5\}$ (whose cardinality may be less than 4) that do not lie in the boundary of $\fS$ yields a triangulable pb surface, the connected component of which that contains (the image of) $e_1$ is isomorphic to $\mathbb{P}_4$, the ideal quadrilateral, the 4-gon, i.e. the closed disc minus four punctures at the boundary. By a repeated application of Lemma \ref{lem-mutation-cutiing} about the compatibility between the coordinate change maps $\Theta^\omega_{\lambda\lambda'}$ and the splitting homomorphisms of skein algebras, the problem boils down to proving Theorem \ref{thm-main-compatibility} in the case when the surface $\fS$ is the 4-gon $\mathbb{P}_4$ (we will justify this `boiling down' process). So we shall first show the following:
\begin{proposition}\label{prop-P4-compatibility_new}
    Theorem \ref{thm-main-compatibility} holds when  
 $\fS$ is the 4-gon $\mathbb P_4$, where $\lambda,\lambda'$ are the triangulations illustrated in Figure \ref{P4_two_triangulations}.
\end{proposition}

\begin{figure}[h]
    \centering
\begingroup%
  \makeatletter%
  \providecommand\color[2][]{%
    \errmessage{(Inkscape) Color is used for the text in Inkscape, but the package 'color.sty' is not loaded}%
    \renewcommand\color[2][]{}%
  }%
  \providecommand\transparent[1]{%
    \errmessage{(Inkscape) Transparency is used (non-zero) for the text in Inkscape, but the package 'transparent.sty' is not loaded}%
    \renewcommand\transparent[1]{}%
  }%
  \providecommand\rotatebox[2]{#2}%
  \newcommand*\fsize{\dimexpr\f@size pt\relax}%
  \newcommand*\lineheight[1]{\fontsize{\fsize}{#1\fsize}\selectfont}%
  \ifx\svgwidth\undefined%
    \setlength{\unitlength}{198.42519685bp}%
    \ifx\svgscale\undefined%
      \relax%
    \else%
      \setlength{\unitlength}{\unitlength * \real{\svgscale}}%
    \fi%
  \else%
    \setlength{\unitlength}{\svgwidth}%
  \fi%
  \global\let\svgwidth\undefined%
  \global\let\svgscale\undefined%
  \makeatother%
  \begin{picture}(1,0.47142857)%
    \lineheight{1}%
    \setlength\tabcolsep{0pt}%
    \put(0,0){\includegraphics[width=\unitlength,page=1]{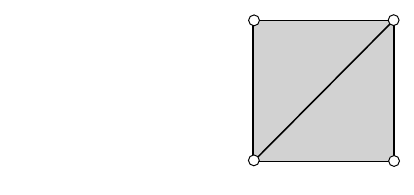}}%
    \put(0.19359755,0.02315805){\color[rgb]{0,0,0}\makebox(0,0)[lt]{\lineheight{1.25}\smash{\begin{tabular}[t]{l}$\lambda$\end{tabular}}}}%
    \put(0,0){\includegraphics[width=\unitlength,page=2]{P4_two_triangulations.pdf}}%
    \put(0.74544285,0.02315805){\color[rgb]{0,0,0}\makebox(0,0)[lt]{\lineheight{1.25}\smash{\begin{tabular}[t]{l}$\lambda'$\end{tabular}}}}%
  \end{picture}%
\endgroup%

    \caption{Two triangulations $\lambda$ and $\lambda'$ of $\mathbb{P}_4$}\label{P4_two_triangulations}
\end{figure}

In fact, \red{now that we are equipped with the tools we recalled or developed so far such as Lemma \ref{lem-mutation-cutiing}}, the core computational part of the proof of the entire Theorem \ref{thm-main-compatibility}, hence the main \red{(remaining)} difficulty thereof, lies in this case of $\mathbb{P}_4$, i.e. in the proof of Proposition \ref{prop-P4-compatibility_new}, which we postpone until the next subsection. At the moment, we explain how Proposition \ref{prop-P4-compatibility_new} implies Theorem \ref{thm-main-compatibility}, i.e. justify the `boiling down' process.

\begin{proof}[Proof of Theorem \ref{thm-main-compatibility}, using Proposition \ref{prop-P4-compatibility_new}]
    Let $\fS$ be a triangulable pb surface, and $\lambda$ and $\lambda'$ be triangulations of $\fS$. As mentioned above, it suffices to show the statement, \eqref{compatibility_equation}, in the case when $\lambda$ and $\lambda'$ are related by a flip at an edge, say $e_1$.

    Define 
    \begin{align}
        \label{lambda_0}
        E =\{e\in\lambda\mid e\neq e_1\text{ and $e$ is not a boundary edge}\}.
    \end{align}
    We use $
    \fS_E:=\mathsf{Cut}_E(\fS)$ to denote the pb surface obtained from $\fS$ by cutting $\fS$ along all edges in $
    E$. Then the connected components of $
    \fS_E$ consist of exactly one $\mathbb P_4$ and some number of $\mathbb P_3$. We use $\lambda_E$ (resp. $
    \lambda'_E$) to denote the triangulation of $
    \fS_E$ induced by $\lambda$
    (resp. $\lambda'$). 

    Then
    $\lambda_E'$ is obtained from $
    \lambda_E$ by a flip at the edge $e_1$. 
Consider the following diagram
\begin{equation}\label{eq-cutting-maps}
\begin{tikzcd}[column sep=small]
\Fr(\mathcal Z_\omega^{\rm mbl}(\fS,\lambda')) \arrow[rr,"\Theta^{\omega}_{\lambda \lambda'}"] 
\arrow[ddd,hook,"\mathcal{S}_{
E}"] & & \Fr(\mathcal Z_\omega^{\rm mbl}(\fS,\lambda)) 
\arrow[ddd,hook,"\mathcal{S}_{
E}"] \\
& \rdS  \arrow[d,"\mathbb{S}_{
E}"] \arrow[lu, "\overline{\rm tr}_{\lambda'}"] \arrow[ru,"
\overline{\rm tr}_{\lambda}"]& \\
& \overline{\cS}_{\omega}(
\fS_E)  \arrow[rd,"
\overline{\rm tr}_{
\lambda_E}"]\arrow[ld,"
\overline{\rm tr}_{
\lambda'_E}"] & \\
 \Fr(\mathcal Z_\omega^{\rm mbl}(
 \fS_E,
 \lambda'_E)) \arrow[rr,"\Theta^{\omega}_{
 \lambda_E \lambda'_E}"] 
 & &  \Fr(\mathcal Z_\omega^{\rm mbl}(
 \fS_E,
 \lambda_E)),
\end{tikzcd}
\end{equation}
where all the vertical arrows are the algebra homomorphisms induced by the 
splitting maps (for cutting along 
all the edges in $
E$), as in Theorem \ref{thm-trace-cut}; note that one can cut the edges in any order, as the splitting maps for different edges of a triangulation commute with each other.  
Observe that the two long vertical arrows in the above diagram are injective. 
Theorem \ref{thm-trace-cut} implies that each of the left and the right trapezoids of this diagram commutes. 
Proposition \ref{prop-P4-compatibility_new} implies that the lower triangle in this diagram commutes. 
Lemma \ref{lem-mutation-cutiing} implies that the outermost square of this diagram commutes. 

One can now show that the upper triangle also commutes. Let us elaborate. Pick any $u \in \overline{\cS}_\omega(\fS)$. Then
\begin{align*}
    \mathcal{S}_{
    E}(\Theta^\omega_{\lambda\lambda'}(
    \overline{\rm tr}_{\lambda'}(u)))
    & = \Theta^\omega_{
    \red{\lambda_E \lambda'_E}}(\mathcal{S}_{
    E}(
    \overline{\rm tr}_{\lambda'}(u))) \quad (\because\mbox{outermost square}) \\
    & = \Theta^\omega_{
    \red{\lambda_E \lambda'_E}}(
    \overline{\rm tr}_{
    \lambda'_E}(\mathbb{S}_{
    E}(u))) \quad (\because\mbox{left trapezoid}) \\
    & = 
    \overline{\rm tr}_{
    \lambda_E}(\mathbb{S}_{
    E}(u)) \quad (\because \mbox{lower triangle}) \\
    & = \mathcal{S}_{
    E}(
    \overline{\rm tr}_{
    \lambda}(u)) \quad (\because \mbox{right trapezoid}).
\end{align*}
Since $\mathcal{S}_{
E}$ (for the right vertical arrow) is injective, we get $\Theta^\omega_{\lambda\lambda'}(
\overline{\rm tr}_
{\lambda'}(u)) = 
\overline{\rm tr}_{
\lambda}(u)$, as desired.
\end{proof}

\subsection{Proof of 
\red{mutable-balanced naturality} for the case of quadrilateral $\mathbb{P}_4$}

In the present subsection we prove Proposition \ref{prop-P4-compatibility_new}. We shall verify that \eqref{compatibility_equation} holds when applied to generating elements of the reduced stated ${\rm SL}_n$-skein algebra $\rdP$. Consider the three corner arcs $a,b,c$ in $\mathbb P_4$ as in Figure \ref{P4_corner_arcs}. 
For $i,j\in\mathbb \{1,2,\ldots,n\}$, we use 
$a_{ij}$ to denote the stated web diagram such that the state of the starting point (resp. endpoint) of $a$ is $i$ (resp. $j$). Similarly, we define $b_{ij}$ and $c_{ij}$. 
We have the following.

\begin{lemma}\cite[Theorem 6.1(c)]{LY23}\label{lem-satuared}
    The stated ${\rm SL}_n$-skein algebra $\mathscr{S}_\omega(\mathbb{P}_4)$ is generated by
    $a_{ij},b_{ij},c_{ij}$ for $i,j\in\mathbb \{1,\ldots,n\}$.
\end{lemma}

\begin{figure}[h]
    \centering
\begingroup%
  \makeatletter%
  \providecommand\color[2][]{%
    \errmessage{(Inkscape) Color is used for the text in Inkscape, but the package 'color.sty' is not loaded}%
    \renewcommand\color[2][]{}%
  }%
  \providecommand\transparent[1]{%
    \errmessage{(Inkscape) Transparency is used (non-zero) for the text in Inkscape, but the package 'transparent.sty' is not loaded}%
    \renewcommand\transparent[1]{}%
  }%
  \providecommand\rotatebox[2]{#2}%
  \newcommand*\fsize{\dimexpr\f@size pt\relax}%
  \newcommand*\lineheight[1]{\fontsize{\fsize}{#1\fsize}\selectfont}%
  \ifx\svgwidth\undefined%
    \setlength{\unitlength}{96.37795276bp}%
    \ifx\svgscale\undefined%
      \relax%
    \else%
      \setlength{\unitlength}{\unitlength * \real{\svgscale}}%
    \fi%
  \else%
    \setlength{\unitlength}{\svgwidth}%
  \fi%
  \global\let\svgwidth\undefined%
  \global\let\svgscale\undefined%
  \makeatother%
  \begin{picture}(1,1)%
    \lineheight{1}%
    \setlength\tabcolsep{0pt}%
    \put(0,0){\includegraphics[width=\unitlength,page=1]{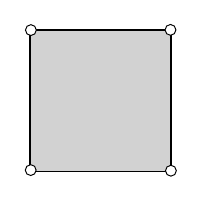}}%
    \put(0.03084271,0.49939851){\color[rgb]{0,0,0}\makebox(0,0)[lt]{\lineheight{1.25}\smash{\begin{tabular}[t]{l}$e_3$\end{tabular}}}}%
    \put(0.40414908,0.27322287){\color[rgb]{0,0,0}\makebox(0,0)[lt]{\lineheight{1.25}\smash{\begin{tabular}[t]{l}$a$\end{tabular}}}}%
    \put(0.33709571,0.56961649){\color[rgb]{0,0,0}\makebox(0,0)[lt]{\lineheight{1.25}\smash{\begin{tabular}[t]{l}$b$\end{tabular}}}}%
    \put(0.57763257,0.60308036){\color[rgb]{0,0,0}\makebox(0,0)[lt]{\lineheight{1.25}\smash{\begin{tabular}[t]{l}$c$\end{tabular}}}}%
    \put(0.87128392,0.49939851){\color[rgb]{0,0,0}\makebox(0,0)[lt]{\lineheight{1.25}\smash{\begin{tabular}[t]{l}$e_5$\end{tabular}}}}%
    \put(0.46050854,0.89131787){\color[rgb]{0,0,0}\makebox(0,0)[lt]{\lineheight{1.25}\smash{\begin{tabular}[t]{l}$e_4$\end{tabular}}}}%
    \put(0.46050854,0.05087684){\color[rgb]{0,0,0}\makebox(0,0)[lt]{\lineheight{1.25}\smash{\begin{tabular}[t]{l}$e_2$\end{tabular}}}}%
    \put(0,0){\includegraphics[width=\unitlength,page=2]{P4_corner_arcs.pdf}}%
  \end{picture}%
\endgroup%

    \caption{Three corner arcs $a,b,c$ in $\mathbb P_4$. Four edges of $\mathbb P_4$ are labeled by $e_2,e_3,e_4,e_5$}\label{P4_corner_arcs}
\end{figure}

In view of the definition of the bad arcs \red{(\S\ref{subsec.reduced_stated_SLn-skein_algebras})} and reduced stated ${\rm SL}_n$-skein algebra \red{(\eqref{eq:cS_red})}, we get:

\begin{lemma}\label{lem-satuared-rd}
    The reduced stated ${\rm SL}_n$-skein algebra $\rdP$ is generated by
    $a_{ij},b_{ij},c_{ij}$ for $i,j \in \{1,\ldots,n\}$ with $i\geq j$.
\end{lemma}

Thus it suffices to check \eqref{compatibility_equation} when applied to these generators. See \cite[Lem.4.2]{Kim21} for another argument that allows us to focus only on these corner arcs in $\mathbb{P}_4$.

We now explain our proof of \eqref{compatibility_equation} for the generator $a_{ij} \in \rdP$, with $i\ge j$. Proof by a direct computation would be quite involved; see \cite{Kim21} for such a proof, for the case $n=3$. Instead, we adapt the methodology using paths lying in a certain `network', which is developed initially by Schrader and Shapiro in \cite{SS17} and further investigated in \cite{CS23,LY23}. This methodology, which we shall now briefly recall, will let us bypass difficult computations. 

Choose any triangulation of $\mathbb{P}_4$. For each triangle $t$ of this triangulation, consider its $n$-triangulation \red{(like in \S\ref{subsec.quantum_trace_maps}; see Figure \ref{Fig;coord_ijk})}, which subdivides the triangle $t$ into $n^2$ small triangles\red{; we may call $t$ a big triangle}. Consider a graph dual to this $n$-triangulation of $t$, such that the set of all edges of this graph is in bijection with the set of all edges of the $n$-triangulation. In particular, for each small triangle there is a vertex of this graph, \red{for each edge of the $n$-triangulation there is an edge of this graph,} and for each boundary edge of the $n$-triangulation there is a \red{1-valent} vertex \red{of this graph which the corresponding edge ends at}; see the blue graph in Figure \ref{P4_networks}. These graphs for different \red{(big)} triangles can naturally be glued together \red{along the 1-valent vertices lying in the sides of the triangles; after gluing, these vertices become 2-valent}.

Now we give orientations to the edges of this graph in a particular way that depends on the choice of a distinguished side of each (big) triangle; we call the resulting oriented graph a {\bf network}. There is a left-turn version and a right-turn version; here we only need to consider the left-turn version. See \cite{SS17,CS23,LY23} for the precise recipe. In particular, in each (big) triangle, the orientations of the edges meeting the distinguished side of the triangle are `outward', while those of the edges meeting the other sides are `inward'. Figure \ref{P4_networks} presents examples of (left-turn) networks $\mathcal{N}$ and $\mathcal{N}'$ for the two triangulations $\lambda$ and $\lambda'$ of $\mathbb{P}_4$, where $e_3$
is the \red{distinguished} `outward' side, and 
$e_2,e_4,e_5$ are the `inward' sides, where we are using the same labels for the four boundary arcs of $\mathbb{P}_4$ as in Figure 
\red{\ref{P4_corner_arcs}}. Although Figure \ref{P4_networks} is only about the case when $n=4$, we believe that the readers can easily see how it would work for general $n$. We label the vertices of these networks lying on the bottom side 
$e_2$ as $\alpha_1,\ldots,\alpha_n$, and those lying on the left side 
$e_3$ as $\beta_1,\ldots,\beta_n$, as appearing in Figure \ref{P4_networks}.

\begin{figure}[h]
    \centering
    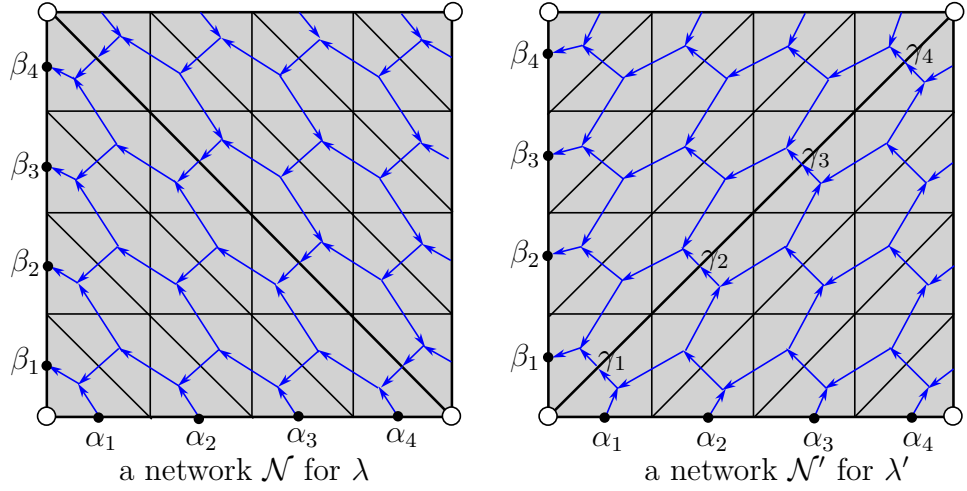
    \caption{(left-turn) Networks for triangulations of $\mathbb{P}_4$; when $n=4$}\label{P4_networks}
\end{figure}

A {\bf path} in a network is a concatenation of oriented edges of the network. For $1\le i,j\le n$, let $\mathcal{N}(ij)$ (resp. $\mathcal{N}'(ij)$) be the set of all paths in the network $\mathcal{N}$ (resp. $\mathcal{N}'$) from $\alpha_i$ to $\beta_j$. For example, $\mathcal{N}(11)$ consists of a single path, which is made of 
\red{two} oriented edges. As can be verified by inspection in Figure \ref{P4_networks}, $\mathcal{N}\red{'}(21)$ consists of two paths, if $n\ge 2$; one path is made of six oriented edges, and the other path of eight oriented edges.

We shall soon see that the role of paths is that the value of a stated arc under the quantum trace map can be presented as a sum over all paths starting and ending with specified boundary vertices of a suitable network. The summand corresponding to a path $p$ is a balanced Weyl-ordered Laurent monomial $Z^{\bf t}$ (see \eqref{eq-weyl-order-Z}). To describe the exponent ${\bf t}$, we need some definitions.

For any path $p$ in $\mathcal{N}(ij)$ (resp. $\mathcal{N}'(ij)$), define the element 
${\bf k}(p,\lambda) = (k_v(p))_{v \in V_\lambda} \in \mathbb Z^{V_{\lambda}}$ (resp. ${\bf k}(p,\lambda') = (k_v(p))_{v\in V_{\lambda'}} \in \mathbb Z^{V_{\lambda'}}$) as
\begin{align}
    \label{eq:k_v_p}
k_v(p)=
\begin{cases}
    1 &\text{ if $v \in V_\lambda$ (resp. $v\in V_{\lambda'}$) lies strictly on the left of the path $p$,}\\
    0 & \text{otherwise},
\end{cases}
\end{align}
where the elements of $V_\lambda$ (resp. $V_{\lambda'}$) are naturally thought of as vertices of the $n$-triangulation of the triangles of $\lambda$ (resp. $\lambda'$) used in the definition of the network $\mathcal{N}$ (resp. $\mathcal{N}'$); see Figure \ref{P4_networks}.

Note that for each $i=1,2,\ldots,n$, $\mathcal{N}(ii)$ (resp. $\mathcal{N}'(ii)$) consists of a single path, say $p_i$ (resp. $p_i'$). 
\red{D}efine
\begin{align}
    \label{eq:k_and_k_prime}
{\bf k} := \sum_{1\leq i\leq n} {\bf k}(p_i,\lambda) \in \mathbb{Z}^{V_\lambda},\qquad
{\bf k}' := \sum_{1\leq i\leq n} {\bf k}(p_i',\lambda') \in \mathbb{Z}^{V_{\lambda'}}.
\end{align}
\red{The following lemma describes the image of a stated corner arc in $\mathbb{P}_4$ under the quantum trace map as a sum over paths in the network $\mathcal{N}$ or $\mathcal{N}'$ with fixed endpoints, using the quantities defined above. A corresponding statement for $\mathbb{P}_3$ is proved in \cite{LY23}, and we generalize it here to $\mathbb{P}_4$.}
\begin{lemma}[quantum trace as sum over paths in a network]\label{lem.quantum_trace_as_sum_over_paths}
    One has
    \begin{enumerate}[label={\rm (\arabic*)}]
        \item $\displaystyle 
        \overline{\rm tr}_\lambda(a_{ij})$ equals $\sum_{p \in \mathcal{N}(ij)} [X^{{\bf k}(p,\lambda)} Z^{\bf k}] \in \mathcal{Z}^{\rm bl}_\omega(\mathbb{P}_4, \lambda)$.

        \item $\displaystyle 
        \overline{\rm tr}_{\lambda'}(a_{ij})$ equals $\sum_{p' \in \mathcal{N}'(ij)} [X^{{\bf k}(p',\lambda')} Z^{{\bf k}'}] \in \mathcal{Z}^{\rm bl}_\omega(\mathbb{P}_4, \lambda').$
    \end{enumerate}
\end{lemma}

\begin{proof}
    (1) The edges of the network $\mathcal{N}$ meeting the diagonal ideal arc of $\lambda$ are oriented toward the lower-left direction as seen in Figure \ref{P4_networks}; so a path in $\mathcal{N}$ starting at a vertex in the left triangle \raisebox{-0,3\height}{\scalebox{0.25}{
\begingroup%
  \makeatletter%
  \providecommand\color[2][]{%
    \errmessage{(Inkscape) Color is used for the text in Inkscape, but the package 'color.sty' is not loaded}%
    \renewcommand\color[2][]{}%
  }%
  \providecommand\transparent[1]{%
    \errmessage{(Inkscape) Transparency is used (non-zero) for the text in Inkscape, but the package 'transparent.sty' is not loaded}%
    \renewcommand\transparent[1]{}%
  }%
  \providecommand\rotatebox[2]{#2}%
  \newcommand*\fsize{\dimexpr\f@size pt\relax}%
  \newcommand*\lineheight[1]{\fontsize{\fsize}{#1\fsize}\selectfont}%
  \ifx\svgwidth\undefined%
    \setlength{\unitlength}{73.7007874bp}%
    \ifx\svgscale\undefined%
      \relax%
    \else%
      \setlength{\unitlength}{\unitlength * \real{\svgscale}}%
    \fi%
  \else%
    \setlength{\unitlength}{\svgwidth}%
  \fi%
  \global\let\svgwidth\undefined%
  \global\let\svgscale\undefined%
  \makeatother%
  \begin{picture}(1,1)%
    \lineheight{1}%
    \setlength\tabcolsep{0pt}%
    \put(0,0){\includegraphics[width=\unitlength,page=1]{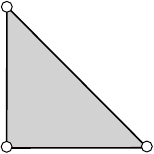}}%
  \end{picture}%
\endgroup%
}} can never leave this triangle. Therefore each path $p \in \mathcal{N}(ij)$ is completely contained in the left triangle. Then the sought-for result follows directly from \cite[Theorem 10.5]{LY23}, which expresses the quantum trace for a triangle $\mathbb{P}_3$ as a sum over paths in a network.

    (2) A path $p \in \mathcal{N}'(ij)$ necessarily meets both triangles of $\lambda'$, so \cite[Theorem 10.5]{LY23} does not directly apply. Denote by 
    $e_1'$ the diagonal ideal arc of $\lambda'$. For each $i=1,\ldots,n$, recall that there is only one path in $\mathcal{N}'(ii)$, which we denoted by $p_i'$; this path $p_i'$ intersects 
    $e_1'$ at a single point, which we denote by $\gamma_i$ as in Figure \ref{P4_networks}.

    We now cut $\mathbb{P}_4$ along 
    $e_1'$. Note that  $\mathsf{Cut}_{
    e_1'}(\mathbb P_4)$ consists of two copies of triangles $\mathbb P_3$. The triangle \raisebox{-0,3\height}{\scalebox{0.25}{
\begingroup%
  \makeatletter%
  \providecommand\color[2][]{%
    \errmessage{(Inkscape) Color is used for the text in Inkscape, but the package 'color.sty' is not loaded}%
    \renewcommand\color[2][]{}%
  }%
  \providecommand\transparent[1]{%
    \errmessage{(Inkscape) Transparency is used (non-zero) for the text in Inkscape, but the package 'transparent.sty' is not loaded}%
    \renewcommand\transparent[1]{}%
  }%
  \providecommand\rotatebox[2]{#2}%
  \newcommand*\fsize{\dimexpr\f@size pt\relax}%
  \newcommand*\lineheight[1]{\fontsize{\fsize}{#1\fsize}\selectfont}%
  \ifx\svgwidth\undefined%
    \setlength{\unitlength}{73.7007874bp}%
    \ifx\svgscale\undefined%
      \relax%
    \else%
      \setlength{\unitlength}{\unitlength * \real{\svgscale}}%
    \fi%
  \else%
    \setlength{\unitlength}{\svgwidth}%
  \fi%
  \global\let\svgwidth\undefined%
  \global\let\svgscale\undefined%
  \makeatother%
  \begin{picture}(1,1)%
    \lineheight{1}%
    \setlength\tabcolsep{0pt}%
    \put(0,0){\includegraphics[width=\unitlength,page=1]{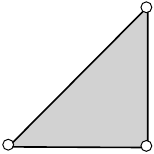}}%
  \end{picture}%
\endgroup%
}} containing $\alpha_i$ (resp. the triangle \raisebox{-0,3\height}{\scalebox{0.25}{
\begingroup%
  \makeatletter%
  \providecommand\color[2][]{%
    \errmessage{(Inkscape) Color is used for the text in Inkscape, but the package 'color.sty' is not loaded}%
    \renewcommand\color[2][]{}%
  }%
  \providecommand\transparent[1]{%
    \errmessage{(Inkscape) Transparency is used (non-zero) for the text in Inkscape, but the package 'transparent.sty' is not loaded}%
    \renewcommand\transparent[1]{}%
  }%
  \providecommand\rotatebox[2]{#2}%
  \newcommand*\fsize{\dimexpr\f@size pt\relax}%
  \newcommand*\lineheight[1]{\fontsize{\fsize}{#1\fsize}\selectfont}%
  \ifx\svgwidth\undefined%
    \setlength{\unitlength}{73.7007874bp}%
    \ifx\svgscale\undefined%
      \relax%
    \else%
      \setlength{\unitlength}{\unitlength * \real{\svgscale}}%
    \fi%
  \else%
    \setlength{\unitlength}{\svgwidth}%
  \fi%
  \global\let\svgwidth\undefined%
  \global\let\svgscale\undefined%
  \makeatother%
  \begin{picture}(1,1)%
    \lineheight{1}%
    \setlength\tabcolsep{0pt}%
    \put(0,0){\includegraphics[width=\unitlength,page=1]{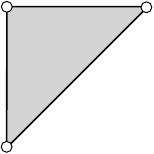}}%
  \end{picture}%
\endgroup%
}} containing $\beta_j$) is denoted by $\mathbb P_3^\alpha$ (resp. $\mathbb P_3^\beta$). We use ${\bf pr}_{
    e_1'}$ to denote the projection from 
    $\mathsf{Cut}_{
    e_1'}(\mathbb P_4)$ to $\mathbb P_4$. For each $i=1,\ldots,n$, note that 
    $({\bf pr}_{
    e_1'})^{-1}(\gamma_i)$ consists of two elements. The one contained in $\mathbb P_{\red{3}}^\alpha$ (resp. $\mathbb P_{\red{3}}^\beta$) is denoted by $\red{\gamma}_i^\alpha$ (resp. $\red{\gamma}_i^\beta$). The network $\mathcal N'$ for $\mathbb P_4$ induces a network $({\bf pr}_{
    e_1'})^{-1}(\mathcal N')$ for $\mathsf{Cut}_{
    e_1'}(\mathbb P_4)$, which has two components. The component contained in $\mathbb P_3^\alpha$ (resp. $\mathbb P_3^\beta$) is denoted by $\mathcal N^\alpha$ (resp. $\mathcal N^\beta$).
    Then $\mathcal N^\alpha$ (resp. $\mathcal N^\beta$) contains vertices $\red{\gamma}_i^\alpha$ (resp. $\red{\gamma}_i^\beta$) for $i=1,\ldots,n$.
    The preimage $({\bf pr}_{
    e_1'})^{-1}(a)$ in $\mathsf{Cut}_{
    e_1'}(\mathbb{P}_4)$ of the oriented arc $a$ in $\mathbb{P}_4$ consists of two oriented arcs. The one contained in $\mathbb P_3^\alpha$ (resp. $\mathbb P_3^\beta$) is denoted by $a^\alpha$ (resp. $a^\beta$). For each  $i,j \in \{1,\ldots,n\}$, similarly as $a_{ij}$, we can define stated arcs $a_{ij}^\alpha$ (resp. $a_{ij}^\beta$) in $\mathbb P_3^\alpha$ (resp. $\mathbb P_3^\beta$), whose state values at the initial and the terminal endpoints are $i$ and $j$ respectively. 

    The definition of the splitting homomorphism $\mathbb{S}_{
    e_1'}$ in equation \eqref{eq-def-splitting} shows that
    $$
    \mathbb S_{
    e_1'}(a_{ij})
    =\sum_{1\leq t\leq n} a^\alpha_{it}\otimes a^\beta_{tj}\in \overline{\cS}_{\omega}(\mathbb P_3^\alpha)\otimes_R
    \overline{\cS}_{\omega}(\mathbb P_3^\beta)=\overline{\cS}_{\omega}(\mathsf{Cut}_{
    e_1'}(\mathbb P_4)).
    $$

    As before, let $\lambda'_{
    e_1'}$ denote the triangulation of $\mathsf{Cut}_{
    e_1'}(\mathbb{P}_4)=\mathbb{P}_3^\alpha \sqcup \mathbb{P}_3^\beta$ induced by $\lambda'$; that is, $\lambda'_{
    e_1'} = \lambda^\alpha \sqcup \lambda^\beta$, where $\lambda^\alpha$ (resp. $\lambda^\beta$) is the unique triangulation of $\mathbb{P}_3^\alpha$ (resp. $\mathbb{P}_3^\beta$). Hence one can express $
    \overline{\rm tr}_{\lambda'_{
    e_1'}} : \overline{\cS}_\omega(\mathsf{Cut}_{
    e_1'}(\mathbb{P}_4)) \to \mathcal{Z}_\omega(\mathsf{Cut}_{
    e_1'}(\mathbb{P}_4),\lambda'_{
    e_1'})$ as $
    \overline{\rm tr}_{\lambda^\alpha} \otimes 
    \overline{\rm tr}_{\lambda^\beta} : \overline{\cS}_\omega(\mathbb{P}_3^\alpha) \otimes_R \overline{\cS}_\omega(\mathbb{P}_3^\beta) \to \mathcal{Z}_\omega(\mathbb{P}_3^\alpha,\lambda^\alpha) \otimes_R \mathcal{Z}_\omega(\mathbb{P}_3^\beta,\lambda^\beta)$. For each of $
    \overline{\rm tr}_{\lambda^\alpha}$ and $
    \overline{\rm tr}_{\lambda^\beta}$ we apply \cite[Theorem 10.5]{LY23}. That is, note that
    \begin{align}
    \nonumber
        \overline{\rm tr}_{\lambda'_{
        e_1'}}(\mathbb{S}_{
        e_1'}(a_{ij})) & = (
        \overline{\rm tr}_{\lambda^\alpha} \otimes 
        \overline{\rm tr}_{\lambda^\beta}) (\sum_{1\le t\le n} a_{it}^\alpha \otimes a^\beta_{tj}) \\
        \nonumber
        & = \sum_{1\le t\le n} 
        \overline{\rm tr}_{\lambda^\alpha}(a_{it}^\alpha) \otimes 
        \overline{\rm tr}_{\lambda^\beta}(a_{tj}^\beta) \\
        \label{quantum_trace_for_lambda_e_prime}
        & = \sum_{1\le t \le n} (\sum_{p \in \mathcal{N}^\alpha(it)} [X^{{\bf k}(p,\lambda^\alpha)} Z^{{\bf k}^\alpha}] \otimes \sum_{p \in \mathcal{N}^\beta(tj)} [X^{{\bf k}(p,\lambda^\beta)} Z^{{\bf k}^\beta}] ), \\
        \nonumber
        & \hspace{55mm} (\because \mbox{\cite[Theorem 10.5]{LY23}})
    \end{align}
where 
\begin{align*}
 \mathcal{N}^\alpha(it) & = \{\mbox{paths in $\mathcal{N}^\alpha$ from $\alpha_i$ to $\gamma_t^\alpha$}\}, \\
\mathcal{N}^\beta(tj) & = \{\mbox{paths in $\mathcal{N}^\beta$ from $\gamma_t^\beta$ to $\beta_j$}\}, \\
{\bf k}(p,\lambda^\alpha) & = (k_v(p))_{v\in V_{\lambda^\alpha}}, \quad k_v(p) = \left\{
\begin{array}{ll}
1 & \mbox{if $v \in 
V_{\lambda^\alpha}$ lies on the left of $p$}, \\
0 & \mbox{otherwise},
\end{array}
 \right. \\
{\bf k}(p,\lambda^\beta) & = (k_v(p))_{v\in V_{\lambda^\beta}}, \quad k_v(p) = \left\{
\begin{array}{ll}
1 & \mbox{if $v \in 
V_{\lambda^\beta}$ lies on the left of $p$}, \\
0 & \mbox{otherwise},
\end{array}
 \right. \\
{\bf k}^\alpha & = \sum_{i=1}^n \red{\mathbf{k}((p_i)^\alpha,\lambda^\alpha)}, \quad \mbox{where $(p_i)^\alpha$ is the only constituent path of $\mathcal{N}^\alpha(ii)$}, \\
{\bf k}^\beta & = \sum_{i=1}^n \red{\mathbf{k}((p_i)^\beta,\lambda^\beta)}, \quad \mbox{where $(p_i)^\beta$ is the only constituent path of $\mathcal{N}^\beta(ii)$}.
\end{align*}

The natural concatenation yields a map
\begin{align}
\label{concatenation_bijection}
\bigsqcup_{1\le t \le n} (\mathcal{N}^\alpha(it) \times \mathcal{N}^\beta(tj)) \to \mathcal{N}'({ij}),
\end{align}
which is easily seen to be a bijection as follows. Any path $p'$ in $\mathcal{N}'(ij)$ must meet the diagonal ideal arc $e_1'$ of $\lambda'$ exactly once, say at $\gamma_t$; then one can see that $p'$ is the concatenation of a uniquely determined path in $\mathcal{N}^\alpha(it)$, say $(p')^\alpha$, and a uniquely determined path in $\mathcal{N}^\beta(tj)$, say $(p')^\beta$.

Recall from \eqref{eq-splitting-torus} the \red{splitting map for the $n$-th root Fock-Goncharov algebras, namely the} algebra embedding
$$
\mathcal{S}_{e_1'} : \mathcal{Z}_\omega(\mathbb{P}_4,\lambda') \to \mathcal{Z}_\omega(\mathsf{Cut}_{e_1'}(\mathbb{P}_4), \lambda'_{e_1'})
\cong \mathcal{Z}_\omega(\mathbb{P}_3^\alpha,\lambda^\alpha) \otimes_R \mathcal{Z}_\omega(\mathbb{P}_3^\beta,\lambda^\beta).
$$
For any path $p'$ in $\mathcal{N}'(ij)$ we can observe
\begin{align}
\label{mathcal_S_e_prime_applied_to_P4}
    \mathcal{S}_{e_1'}[ X^{{\bf k}(p',\lambda')} Z^{{\bf k}'}] = [X^{{\bf k}((p')^\alpha,\lambda^\alpha)} Z^{ {\bf k}^\alpha }] \otimes [X^{{\bf k}((p')^\beta,\lambda^\beta)} Z^{ {\bf k}^\beta }],
\end{align}
which one can see from the definition of $\mathcal{S}_{e_1'}$ in \eqref{cutting_homomorphism_for_Z_omega} to hold up to a multiplicative constant that is an integer power of $\omega^{1/2}$. One can further observe that the multiplicative constant is $1$, by showing that $\mathcal{S}_{e_1'}$ sends a Weyl-ordered Laurent monomial to a Weyl-ordered Laurent monomial, which in turn is straightforward to prove, e.g. by the same argument as in the proof of \cite[Lem.3.19(3)--(5)]{Kim21}. \red{A brief summary of this argument is that we first show that $\mathcal{S}_{e_1'}$ commutes with the $*$-map (see \S\ref{subsec:quantum_mutations_for_mutable-balanced_n-th_root_algebras}) and then use the fact that the power of $\omega$ in the definition \eqref{Weyl_ordering_Z}--\eqref{eq-weyl-order-Z} of a Weyl-ordered Laurent monomial  is the unique one that makes the Weyl-ordered Laurent monomial $*$-invariant}.

Now, for $i,j \in \{1,\ldots,n\}$, take the sum of \eqref{mathcal_S_e_prime_applied_to_P4} over all paths $p'$ in $\mathcal{N}'(ij)$:
\begin{align*}
    \mathcal{S}_{e_1'} (\sum_{p' \in \mathcal{N}'(ij)} [ X^{{\bf k}(p',\lambda')} Z^{{\bf k}'}]) & = \sum_{p' \in \mathcal{N}'(ij)} ( [X^{{\bf k}((p')^\alpha,\lambda^\alpha)} Z^{ {\bf k}^\alpha }] \otimes [X^{{\bf k}((p')^\beta,\lambda^\beta)} Z^{ {\bf k}^\beta }] ) \\
    & = \sum_{1\le t\le n} (\sum_{p \in \mathcal{N}^\alpha(it)} [X^{{\bf k}(p,\lambda^\alpha)} Z^{ {\bf k}^\alpha }]) \otimes (\sum_{p \in \mathcal{N}^\beta(tj)} [X^{{\bf k}(p,\lambda^\beta)} Z^{ {\bf k}^\beta }]),
\end{align*}
where the latter equality is due to the bijection in \eqref{concatenation_bijection}. We thus have
\begin{align*}
    \mathcal{S}_{e_1'} (\sum_{p' \in \mathcal{N}'(ij)} [ X^{{\bf k}(p',\lambda')} Z^{{\bf k}'}]) & = 
    \overline{\rm tr}_{\lambda'_{e_1'}}(\mathbb{S}_{e_1'}(a_{ij})) \qquad (\because \mbox{\eqref{quantum_trace_for_lambda_e_prime}}) \\
    & = \mathcal{S}_{e_1'} (
    \overline{\rm tr}_{\lambda'}(a_{ij})) \qquad (\because \mbox{Theorem \ref{thm-trace-cut}}).
\end{align*}
Since $\mathcal{S}_{e_1'}$ is injective, it follows that $\sum_{p' \in \mathcal{N}'(ij)} [ X^{{\bf k}(p',\lambda')} Z^{{\bf k}'}]$ equals $\overline{\rm tr}_{\lambda'}(a_{ij})$, as desired. 
\end{proof}

\begin{remark}
   Using the same technique, Lemma \ref{lem.quantum_trace_as_sum_over_paths} can be generalized to calculate the image of the stated corner arcs in any triangulable pb surface under the quantum trace map.
\end{remark}

Keep in mind that our goal is to show the following form of 
\red{naturality}:
$$
\Theta^\omega_{\lambda\lambda'} (\overline{\rm tr}_{\lambda'}(a_{ij})) = \overline{\rm tr}_\lambda(a_{ij}).
$$
The reason why we invoked the description of the value of the quantum trace in terms of sum over paths in a network is because Schrader and Shapiro showed a form of 
\red{naturality} of such sums over paths 
\red{under} mutations, namely \cite[Proposition 4.2]{SS17}. Each cluster mutation corresponds to mutation of a network (see \cite[Figure 9]{SS17}), and the sequence of cluster mutations realizing the flip $\lambda \leadsto \lambda'$ along the ideal arc $e_1$ in $\mathbb{P}_4$ (as in \eqref{two_seeds_connected_by_sequence_of_mutations}) transforms the network $\mathcal{N}$ into the network $\mathcal{N}'$. Applying \cite[Proposition 4.2]{SS17} to this mutation sequence, we obtain the following:
\begin{lemma}\label{lem.SS_core}
    For each $i,j\in\{1,\ldots,n\}$, one has
    \begin{align}
    \nonumber
    \Phi^q_{\lambda \lambda'} (\sum_{p' \in \mathcal{N}'(ij)} X^{{\bf k}(p',\lambda')}) = \sum_{p \in \mathcal{N}(ij)} X^{{\bf k}(p,\lambda)}. \qed
\end{align}
\end{lemma}
We stress that this lemma settles the most involved part of the proof of our sought-for Proposition \ref{prop-P4-compatibility_new}. That is, a direct computational proof of this lemma, for each fixed $n$, would require an enormous effort; see \cite{Kim21} for a version of such a proof for $n=3$. In a sense, the advantage of using the network description of quantum trace mainly lies in this lemma.

We now perform the remaining steps to prove Proposition \ref{prop-P4-compatibility_new}.
\begin{lemma}\label{lem.Theta_on_normalizer_monomial}
    \red{For ${\bf k}$ and ${\bf k}'$ defined in \eqref{eq:k_and_k_prime}, o}ne has
    \begin{align}
        \nonumber
        \Theta^\omega_{\lambda\lambda'}(Z^{{\bf k}'}) = Z^{\bf k}.
    \end{align}
\end{lemma}

Note that $\mathcal{N}'(ii) = \{p'_i\}$ and $\mathcal{N}(ii) = \{p_i\}$, so from Lemma \ref{lem.SS_core} applied to $i=j$ we get \begin{align}
\label{Phi_q_sends_k_p_i_prime_to_k_p_i}
    \Phi^q_{\lambda\lambda'}(X^{{\bf k}(p'_i,\lambda')}) = X^{{\bf k}(p_i,\lambda)},
\end{align}
hence $\Phi^q_{\lambda\lambda'}(X^{ {\bf k}'}) = X^{\bf k}$ holds up to a multiplicative constant that is an integer power of $q^{1/2}$. One can show that this constant is 1, as we shall see later. However, $\Phi^q_{\lambda\lambda'}(X^{ {\bf k}'}) = X^{\bf k}$ does not directly imply the sought-for $\Theta^\omega_{\lambda\lambda'}(Z^{{\bf k}'}) = Z^{\bf k}$, although these two are consistent with each other, in view of 
\red{Lemma \ref{lem-extension}}. We give a direct proof of Lemma \ref{lem.Theta_on_normalizer_monomial} below.

Note that any ground ring $R$ is an algebra over $\mathbb Z[\omega^{\pm\frac{1}{2}}]$. Then we can use the functor $
- \otimes_{\mathbb Z[\omega^{\pm\frac{1}{2}}]} R$ to change the ground ring from 
$\mathbb Z[\omega^{\pm\frac{1}{2}}]$ to $R$.
Due to this functor, we assume that the ground ring is $\mathbb Z[\omega^{\pm\frac{1}{2}}]$ in the following proof.

\begin{proof}[Proof of Lemma \ref{lem.Theta_on_normalizer_monomial}]
\red{We will first show a counterpart of \eqref{Phi_q_sends_k_p_i_prime_to_k_p_i} for $\Theta^\omega_{\lambda\lambda'}$:
\begin{align}
    \label{eq:Theta_omega_sends_k_p_i_prime_to_k_p_i}
    \Theta^\omega_{\lambda\lambda'}(Z^{{\bf k}(p_i',\lambda')}) = Z^{{\bf k}(p_i,\lambda)}, \quad \forall i =1,\ldots,n.
\end{align}
} Denote by $\mathcal{D}_\lambda$ and $\mathcal{D}_{\lambda'}$ the cluster seeds corresponding to $\lambda$ and $\lambda'$, and that the sequence of mutations at the vertices $v_1,v_2,\ldots,v_r$ turns $\mathcal{D}_\lambda$ to $\mathcal{D}_{\lambda'}$, as in \eqref{two_seeds_connected_by_sequence_of_mutations}. Like we did in the proof of Lemma \ref{lem-mutation-cutiing}, let $\mathcal{D}_0 = \mathcal{D}_\lambda$, and let $\mathcal{D}_\red{\ell} = \mu_{v_\red{\ell}}(\mathcal{D}_{\red{\ell}-1})$ for each $\red{\ell}=1,\ldots,r$, so that $\mathcal{D}_r = \mathcal{D}_{\lambda'}$ in particular. Denote by $\mathcal{V}$ the underlying vertex sets of all these seeds $\mathcal{D}_\red{\ell}$, $\red{\ell}=0,1,\ldots,r$, which are naturally identified with each other; in particular, $\mathcal{V} = V_\lambda = V_{\lambda'}$. 
\red{Then} $\mathcal{V}_{\rm mut}$ 
\red{is the set of all} vertices \red{of $\mathcal{V}$} lying in the interior of $\mathbb{P}_4$. Let $Q_\red{\ell}$ be the exchange matrix of $\mathcal{D}_\red{\ell}$. 

Let $i \in \{1,\ldots,n\}$. Define 
\begin{align}
    \label{eq:k_i_ell}
    \red{{\bf k}_i^{(\ell)} = (k^{(\ell)}_{i,v})_{v\in \mathcal{V}}} \in \mathbb{Z}^\mathcal{V}
\end{align}
by the following equations:
$$
\red{{\bf k}^{(r)}_i} = {\bf k}(p_i',\lambda'), \qquad
\nu'_{v_\red{\ell}}(Z^{\red{{\bf k}_i^{(\ell)}}}) = Z^{\red{{\bf k}_i^{(\ell-1)}}}, \quad \red{\ell}=r,r-1,\ldots,2,1.
$$
More explicitly, in view of \eqref{eq-quantum-mutation_Z}, the second equation says that
\begin{align}
    \label{nu_prime_v_k_on_exponents}
    \red{k^{(\ell-1)}_{i,v_\ell}} = - \red{k^{(\ell)}_{i,v_\ell}} + \sum_{v\in \mathcal{V}} [Q_{\red{\ell}-1}(v,v_\red{\ell})]_+ \red{k^{(\ell)}_{i,v}}, \quad\mbox{and}\quad \red{k^{(\ell-1)}_{i,v}} = \red{k^{(\ell)}_{i,v}} \quad \mbox{for all $v\neq v_\red{\ell}$}.
\end{align}
For the seed $\mathcal{D}_{\lambda'}$, it is straightforward to observe that 
\begin{align}
    \label{Q_r_sum_zero}
    \sum_{v \in \mathcal{V}} Q_r(u,v) \red{k^{(r)}_{i,v}} \red{= \sum_{v \in \mathcal{V}} Q_{\lambda'}(u,v)} \red{k^{(r)}_{i,v}} = 0 \quad\mbox{for all $u \in \mathcal{V}_{\rm mut}$.}
\end{align}
One can translate this condition into $Z_u Z^{\red{{\bf k}^{(r)}_i}} = Z^{\red{{\bf k}^{(r)}_i}} Z_u \in \mathcal{Z}_\omega(\mathcal{D}_\red{r})\red{=\mathcal{Z}_\omega(\mathcal{D}_{\lambda'})}$, $\forall u\in \mathcal{V}_{\rm mut}$. By applying Lemma \ref{lem-commute} to \eqref{Q_r_sum_zero} inductively, one obtains:
\begin{align}
\label{Q_k_minus_one_sum_zero}
    \mbox{For each $\red{\ell}=r,r-1,\ldots,1$, one has $\sum_{v\in \mathcal{V}} Q_{\red{\ell}-1} (u,v) \red{k_{i,v}^{(\ell-1)}} = 0$, $\forall u \in \mathcal{V}_{\rm mut}$.}
\end{align}
Since $\nu^\omega_{v_\red{\ell}} = \nu^{\sharp \omega}_{v_\red{\ell}} \circ \nu'_{v_\red{\ell}}$, in view of the action of $\nu^{\sharp \omega}_{v_\red{\ell}}$ as written in Lemma \ref{lem.nu_sharp_well-defined}, it is easy to see that \eqref{Q_k_minus_one_sum_zero} implies:
$$
\mbox{For each $\red{\ell}=r,r-1,\ldots,1$, one has $\nu^\omega_{v_\red{\ell}}(Z^{\red{{\bf k}^{(\ell)}_i}}) = Z^{\red{{\bf k}^{(\ell-1)}_i}}$.}
$$
Thus $\Theta^\omega_{\lambda\lambda'}(Z^{\red{{\bf k}^{(r)}_i}}) = \nu^\omega_{v_r} \cdots \nu^\omega_{v_1}(Z^{\red{{\bf k}^{(r)}_i}}) = Z^{\red{{\bf k}^{(0)}_i}}$. 

On the other hand, a verbatim proof, with all $\nu^?_?$ replaced by $\mu^?_?$ and $Z^?$ by $X^?$, shows $\Phi^q_{\lambda\lambda'}(X^{\red{{\bf k}^{(r)}_i}}) = X^{\red{{\bf k}^{(0)}_i}}$. Since 
\begin{align}
    \label{eq:k_i_r}
    \red{{\bf k}^{(r)}_i} = {\bf k}(p_i',\lambda'),
\end{align}
now \eqref{Phi_q_sends_k_p_i_prime_to_k_p_i} (together with injectivity of $\Phi^q_{\lambda\lambda'}$) tells us that 
\begin{align}
    \label{eq:k_i_0}
\red{{\bf k}^{(0)}_i} = {\bf k}(p_i,\lambda).
\end{align}
\red{Hence we showed \eqref{eq:Theta_omega_sends_k_p_i_prime_to_k_p_i}.}

\red{Note
\begin{align*}
& Z^{\bf k} \stackrel{{\rm eq.}\eqref{eq:k_and_k_prime}}{=} Z^{{\bf k}(p_1,\lambda) + \cdots {\bf k}(p_n,\lambda)} \stackrel{{\rm eq.}\eqref{Weyl_ordering_Z},\eqref{eq-weyl-order-Z}}{=} 
\omega^{-m} Z^{{\bf k}(p_1,\lambda)} \cdots Z^{{\bf k}(p_n,\lambda)}, \\
& Z^{{\bf k}'} \stackrel{{\rm eq.}\eqref{eq:k_and_k_prime}}{=} Z^{{\bf k}(p_1',\lambda') + \cdots {\bf k}(p_n',\lambda')} \stackrel{{\rm eq.}\eqref{Weyl_ordering_Z},\eqref{eq-weyl-order-Z}}{=} 
\omega^{-m'} Z^{{\bf k}(p_1',\lambda')} \cdots Z^{{\bf k}(p_n',\lambda')},
\end{align*}
for some numbers $m, m' \in \frac{1}{2}\mathbb{Z}$. Since $\Theta^\omega_{\lambda\lambda'}$ is a homomorphism, from \eqref{eq:Theta_omega_sends_k_p_i_prime_to_k_p_i} it follows that $\Theta^\omega_{\lambda\lambda'}(Z^{{\bf k}'}) = \omega^{m-m'} Z^{{\bf k}}$. From Lemma \ref{lem:nu_omega_preserves_star} it follows that $\Theta^\omega_{\lambda\lambda'}$ preserves $*$-structures, hence sends a $*$-invariant element to a $*$-invariant element. Since $Z^{{\bf k}'}$ and $Z^{\bf k}$ are $*$-invariant, it follows that $m-m'=0$, so $\Theta^\omega_{\lambda\lambda'}(Z^{{\bf k}'}) = Z^{{\bf k}}$ holds, as desired. \quad {\it [End of Proof of Lemma \ref{lem.Theta_on_normalizer_monomial}]}}
\end{proof}

\red{By assembling the results obtained so far, we will try to prove the sought-for equality $\Theta^\omega_{\lambda\lambda'} (\overline{\rm tr}_{\lambda'}(a_{ij})) = \overline{\rm tr}_\lambda(a_{ij})$. From Lemma \ref{lem.quantum_trace_as_sum_over_paths} we have
\begin{align}
    \label{eq:tr_of_a_ij_lambda_and_lambda_prime}
    \overline{\rm tr}_\lambda(a_{ij}) = \sum_{p \in \mathcal{N}(ij)} [X^{{\bf k}(p,\lambda)} Z^{{\bf k}}], \quad\mbox{and}\quad
\overline{\rm tr}_{\lambda'}(a_{ij})=\sum_{p' \in \mathcal{N}'(ij)} [X^{{\bf k}(p',\lambda')} Z^{{\bf k}'}]
\end{align}
To proceed, for each of the paths $p \in \mathcal{N}(ij)$ and $p' \in \mathcal{N}'(ij)$ we need to compute the numbers $m_p, m_{p'} \in \frac{1}{2}\mathbb{Z}$ such that 
\begin{align}
    \label{eq:m_p_and_m_prime_def}
    [X^{{\bf k}(p,\lambda)} Z^{{\bf k}}] = \omega^{-m_p} X^{{\bf k}(p,\lambda)} Z^{{\bf k}},\quad\mbox{and}\quad [X^{{\bf k}(p',\lambda')} Z^{{\bf k}'}] = \omega^{-m_{p'}} X^{{\bf k}(p',\lambda')} Z^{{\bf k}'}. 
\end{align}
Write ${\bf k} = (k_v)_{v\in V_\lambda}$ and ${\bf k}' = (k'_v)_{v\in V_{\lambda'}}$; recall ${\bf k}(p,\lambda)=(k_v(p))_{v\in V_\lambda}$ and ${\bf k}(p',\lambda') = (k_v(p'))_{v\in V_{\lambda'}}$. It is a standard and straightforward exercise from the definition of the Weyl-ordered product (\eqref{Weyl_ordering_Z}, \eqref{eq-weyl-order-Z}) that
\begin{align}
\label{eq:m_p_and_m_p_prime}
    m_p = n \sum_{u,v \in V_\lambda} k_u(p) k_v Q_\lambda(u,v), \qquad
    m_{p'} = n \sum_{u,v \in V_{\lambda'}} k_u(p') k'_v Q_{\lambda'}(u,v).
\end{align}
(note that the factor $n$ comes from $X_v = Z_v^n$) Now, note
\begin{align*}
    & {\bf k} \stackrel{{\rm eq}.\eqref{eq:k_and_k_prime}}{=} {\bf k}(p_1,\lambda)+\cdots+{\bf k}(p_n,\lambda)
    \stackrel{{\rm eq}.\eqref{eq:k_i_0}}{=} {\bf k}_1^{(0)}+ \cdots +{\bf k}_n^{(0)}, \\
    & {\bf k}' \stackrel{{\rm eq}.\eqref{eq:k_and_k_prime}}{=} {\bf k}(p_1',\lambda')+\cdots+{\bf k}(p_n',\lambda')
    \stackrel{{\rm eq}.\eqref{eq:k_i_r}}{=} {\bf k}_1^{(r)}+ \cdots +{\bf k}_n^{(r)}.
\end{align*}
Recall from \eqref{eq:k_i_ell} that ${\bf k}^{(\ell)}_i = (k_{i,v}^{(\ell)})_{v\in \mathcal{V}}$. So, for each $v\in \mathcal{V}$ we have $k_v = \sum_{i=1}^n k^{(0)}_{i,v}$ and $k'_v = \sum_{i=1}^n k^{(r)}_{i,v}$. From \eqref{Q_k_minus_one_sum_zero} for $\ell=1$, for each mutable vertex $u \in \mathcal{V}_{\mathrm{mut}}$ and for each $i=1,\ldots,n$ we have $\sum_{v\in \mathcal{V}} Q_\lambda(u,v) k_{i,v}^{(0)}=0$. Similarly, from \eqref{Q_r_sum_zero}, for each $u \in \mathcal{V}_{\mathrm{mut}}$ and each $i=1,\ldots,n$ we have $\sum_{v\in \mathcal{V}} Q_{\lambda'}(u,v) k_{i,v}^{(r)}=0$. So, by summing over $i$ we get
\begin{align}
\label{eq:sums_for_u_mut_zero}
\sum_{v\in \mathcal{V}} Q_\lambda(u,v) k_v=0 \quad\mbox{and}\quad
\sum_{v\in \mathcal{V}} Q_{\lambda'}(u,v) k'_v=0, \qquad \mbox{for each $u\in \mathcal{V}_{\mathrm{mut}}$.}    
\end{align}
Hence, from \eqref{eq:m_p_and_m_p_prime}, keeping in mind the identification $\mathcal{V} = V_\lambda = V_{\lambda'}$, we have
\begin{align*}
    m_p = n \sum_{u \in \mathcal{V}} \sum_{v\in \mathcal{V}} k_u(p) k_v Q_\lambda(u,v) & = n \sum_{u \in \mathcal{V}} \left( k_u(p) \sum_{v\in \mathcal{V}}  k_v Q_\lambda(u,v) \right) \\
    & \stackrel{{\rm eq}.\eqref{eq:sums_for_u_mut_zero}}{=}
    n \sum_{u \in \mathcal{V}\setminus \mathcal{V}_{\mathrm{mut}}} \left( k_u(p) \sum_{v\in \mathcal{V}}  k_v Q_\lambda(u,v) \right), \\
    m_{p'} = n \sum_{u \in \mathcal{V}} \sum_{v\in \mathcal{V}} k_u(p') k_v' Q_{\lambda'}(u,v)
    & \stackrel{{\rm eq}.\eqref{eq:sums_for_u_mut_zero}}{=}
    n \sum_{u \in \mathcal{V}\setminus \mathcal{V}_{\mathrm{mut}}} \left( k_u(p') \sum_{v\in \mathcal{V}}  k'_v Q_{\lambda'}(u,v) \right).
\end{align*}
For each $u \in \mathcal{V}\setminus \mathcal{V}_{\mathrm{mut}}$, i.e. a frozen vertex of $V_\lambda$ or $V_{\lambda'}$, i.e. a vertex lying on the boundary of $\fS$, we can easily compute the numbers $k_u(p)$ and $k_u(p')$ directly from their definition \eqref{eq:k_v_p}. Namely, for $p\in \mathcal{N}(ij)$ and $p' \in \mathcal{N}'(ij)$ we have
$$
k_u(p) = k_u(p') = \begin{cases}
    1 & \mbox{if $u = \alpha_h$ for $1\le h \le i$ or $u = \beta_g$ for $1\le g \le j$,} \\
    0 & \mbox{if $u$ is any other frozen vertex.}
\end{cases}
$$
So, if we let
$$
\mathcal{V}_{ij} := \{ \alpha_h, \beta_g \, | \, 1\le h \le i, ~ 1\le g \le j\},
$$
we get
$$
m_p = n \sum_{u \in \mathcal{V}_{ij}} \sum_{v\in \mathcal{V}} k_v Q_\lambda(u,v), \qquad
m_{p'} = n \sum_{u \in \mathcal{V}_{ij}} \sum_{v\in \mathcal{V}} k_v' Q_{\lambda'}(u,v).
$$
All we need to know for our purposes is that these numbers $m_p$ and $m_{p'}$ depend only on $i$ and $j$, but not on the paths $p\in \mathcal{N}(ij)$ and $p' \in \mathcal{N}'(ij)$. So, we could let
\begin{align}
    \label{eq:m_ij_and_m_ij_prime}
    m_{ij} = m_p, \qquad
m_{ij}' = m_{p'}.
\end{align}
}

\red{
Putting things together, observe
\begin{align*}
\Theta^\omega_{\lambda\lambda'} (\overline{\rm tr}_{\lambda'}(a_{ij}))
& = \Theta^\omega_{\lambda\lambda'}({\textstyle \sum}_{p' \in \mathcal{N}'(ij)} [X^{{\bf k}(p',\lambda')} Z^{{\bf k}'}]) \qquad (\because \mbox{eq.\eqref{eq:tr_of_a_ij_lambda_and_lambda_prime}}) \\
& = \Theta^\omega_{\lambda\lambda'}(\omega^{-m_{ij}'} {\textstyle \sum}_{p' \in \mathcal{N}'(ij)}  X^{{\bf k}(p',\lambda')} Z^{{\bf k}'}) \qquad (\because \mbox{eq.\eqref{eq:m_p_and_m_prime_def}, eq.\eqref{eq:m_ij_and_m_ij_prime}}) \\
& = \omega^{-m_{ij}'} {\textstyle \sum}_{p' \in \mathcal{N}'(ij)} \Theta^\omega_{\lambda\lambda'}(X^{{\bf k}(p',\lambda')}) \Theta^\omega_{\lambda\lambda'}(Z^{{\bf k}'}) \\
& = \omega^{-m_{ij}'} \left( {\textstyle \sum}_{p' \in \mathcal{N}'(ij)} \Phi^q_{\lambda\lambda'}(X^{{\bf k}(p',\lambda')}) \right) \Theta^\omega_{\lambda\lambda'}(Z^{{\bf k}'}) \quad (\because \mbox{Lemma \ref{lem-extension}}) \\
& = \omega^{-m_{ij}'} \left( {\textstyle \sum}_{p \in \mathcal{N}(ij)} X^{{\bf k}(p,\lambda)} \right) Z^{\bf k} \quad (\because \mbox{Lemmas \ref{lem.SS_core} and \ref{lem.Theta_on_normalizer_monomial}}) \\
& = \omega^{-m_{ij}'} {\textstyle \sum}_{p \in \mathcal{N}(ij)} X^{{\bf k}(p,\lambda)}Z^{\bf k} \\
& = \omega^{m_{ij} - m_{ij}'} {\textstyle \sum}_{p \in \mathcal{N}(ij)} [X^{{\bf k}(p,\lambda)}Z^{\bf k}] \quad (\because \mbox{eq.\eqref{eq:m_p_and_m_prime_def}, eq.\eqref{eq:m_ij_and_m_ij_prime}}) \\
& = \omega^{m_{ij} - m_{ij}'} \, \overline{\rm tr}_\lambda(a_{ij}) \qquad (\because \mbox{eq.\eqref{eq:tr_of_a_ij_lambda_and_lambda_prime}})
\end{align*}

}

From Lemma \ref{lem.quantum_trace_as_sum_over_paths}, both $\overline{\rm tr}_{\lambda'}(a_{ij}) \in \mathcal{Z}_\omega(\mathbb{P}_4,\lambda')$ and $\overline{\rm tr}_\lambda(a_{ij}) \in \mathcal{Z}_\omega(\mathbb{P}_4,\lambda)$ are sums of Weyl-ordered Laurent monomials, hence they are invariant under the $*$-maps. By Lemma \ref{lem:nu_omega_preserves_star} it follows that $\Theta^\omega_{\lambda\lambda'}$ preserves the $*$-structures, hence sends $*$-invariant elements to $*$-invariant elements. It then follows that $\omega^{m_{ij} - m'_{ij}}=1$. We thus obtain the sought-for compatibility $\Theta^\omega_{\lambda\lambda'}( \overline{\rm tr}_{\lambda'}(a_{ij})) = \overline{\rm tr}_\lambda(a_{ij})$ for the stated arcs $a_{ij}$.

The proofs for the remaining generators $b_{ij}$ and $c_{ij}$ are similar, and we omit them. In view of Lemma \ref{lem-satuared-rd}, the 
\red{naturality} $\Theta^\omega_{\lambda\lambda'}(\overline{\rm tr}_{\lambda'}(u)) = \overline{\rm tr}_\lambda(u)$ holds for all elements $u \in \overline{\cS}_\omega(\mathbb{P}_4)$, hence Proposition \ref{prop-P4-compatibility_new} is proved. As explained in \S\ref{subsec.compatibility_Theta}, this finishes the proof of Theorem \ref{thm-main-compatibility}.

\subsection{The restriction of
 $\Theta^\omega_{\lambda \lambda'}$ to the balanced part}\label{subsec:restriction_of_Theta_to_balanced} Theorem \ref{thm-main-compatibility} states that the quantum trace maps $
 \overline{\rm tr}_\lambda$ and $
 \overline{\rm tr}_{\lambda'}$ for different triangulations are compatible with each other through the \red{mutable-balanced} quantum coordinate change map $\Theta^\omega_{\lambda\lambda'} : {\rm Frac}(\mathcal{Z}_\omega^{\rm mbl}(\fS,\lambda')) \to {\rm Frac}(\mathcal{Z}_\omega^{\rm mbl}(\fS,\lambda))$ between the skew-fields of fractions of the mutable-balanced algebras. Meanwhile, in view of Theorem \ref{thm.quantum_trace}(b), the image of the quantum trace maps $
 \overline{\rm tr}_{\lambda'}$ and $
 \overline{\rm tr}_\lambda$ lie in ${\rm Frac}(\mathcal{Z}^{\rm bl}_\omega(\fS,\lambda'))$ and ${\rm Frac}(\mathcal{Z}^{\rm bl}_\omega(\fS,\lambda))$, the skew-fields of fractions of the balanced algebras. In the present subsection we show that $\Theta^\omega_{\lambda\lambda'}$ restricts to the map between these balanced fraction skew-fields.

\begin{proposition}[\red{balanced quantum coordinate change map}]\label{prop.Theta_restricts_to_balanced}
    Let $\fS$ be a triangulable pb surface, and $\lambda$ and $\lambda'$ be triangulations of $\fS$. The \red{mutable-balanced} quantum coordinate change isomorphism $\Theta^\omega_{\lambda\lambda'}:{\rm Frac}(\mathcal{Z}_\omega^{\rm mbl}(\fS,\lambda')) \to {\rm Frac}(\mathcal{Z}_\omega^{\rm mbl}(\fS,\lambda))$ established in \S\ref{subsec:coordinate_change_isomorphisms_for_change_of_triangulations} restricts to an isomorphism
    $$
    (\Theta^{\omega}_{\lambda\lambda'})^{\rm bl} \colon {\rm Frac}(\mathcal{Z}_\omega^{\rm bl}(\fS,\lambda')) \to {\rm Frac}(\mathcal{Z}_\omega^{\rm bl}(\fS,\lambda)).
    $$
\end{proposition}
\red{We note that it is this balanced quantum coordinate change map $(\Theta^{\omega}_{\lambda\lambda'})^{\rm bl}$, rather than the mutation-balanced one $\Theta^\omega_{\lambda\lambda'}$ of \S\ref{subsec:coordinate_change_isomorphisms_for_change_of_triangulations}, that is constructed in \cite{BW11, Hiatt} for $n=2$ and in \cite{Kim21} for $n=3$, and that appears in our Theorem \ref{thm.main.intro} in the introduction section under the symbol $\Theta^{\omega;{\rm SL}_n}_{\lambda\lambda'}$.}

Since $\Theta^\omega_{\lambda'\lambda} : {\rm Frac}(\mathcal{Z}^{\rm mbl}_\omega(\fS,\lambda)) \to {\rm Frac}(\mathcal{Z}^{\rm mbl}_\omega(\fS,\lambda'))$ is the inverse of $\Theta^\omega_{\lambda\lambda'}$ (which follows from \eqref{Theta_omega_lambda_lambda_prime_consistency}, $\Theta^\omega_{\lambda\lambda}={\rm Id}$ and $\Theta^\omega_{\lambda'\lambda'}={\rm Id}$), \red{in order to prove Proposition \ref{prop.Theta_restricts_to_balanced}} it suffices to show that $\Theta^\omega_{\lambda\lambda'}$ does restrict to a homomorphism ${\rm Frac}(\mathcal{Z}_\omega^{\rm bl}(\fS,\lambda')) \to {\rm Frac}(\mathcal{Z}_\omega^{\rm bl}(\fS,\lambda))$.

Our strategy to show this is to use the splitting homomorphisms $\mathcal{S}_e$ for the skew-fields of fractions of the mutable-balanced algebras, as in \eqref{mathcal_S_e_for_mbl}. We first show \red{that they further restrict to the balanced subalgebras}:
\begin{lemma}\label{lem:cutting_homomorphism_for_mbl_FG_algebra_restritcs_to_bl}
    Let $\fS$ be a triangulable pb surface, and $\lambda$ be a triangulation of $\fS$. Let $e$ be an edge of $\lambda$ that is not a boundary edge. Let $\mathsf{Cut}_e(\fS)$ be the pb surface obtained from $\fS$ by cutting along $e$, and let $\lambda_e$ be the triangulation of $\mathsf{Cut}_e(\fS)$ induced by $\lambda$. Then, the splitting homomorphism $\mathcal{S}_e : {\rm Frac}(\mathcal{Z}^{\rm mbl}_\omega(\fS,\lambda)) \to {\rm Frac}(\mathcal{Z}^{\rm mbl}_\omega(\mathsf{Cut}_e(\fS),\lambda_e))$ in \eqref{mathcal_S_e_for_mbl} restricts to the homomorphism
    $$
    \mathcal{S}_e^{\rm bl}\colon {\rm Frac}(\mathcal{Z}^{\rm bl}_\omega(\fS,\lambda)) \to {\rm Frac}(\mathcal{Z}^{\rm bl}_\omega(\mathsf{Cut}_e(\fS),\lambda_e)).
    $$
\end{lemma}

\begin{proof}
    Let $Z^{\bf t} \in \mathcal{Z}^{\rm bl}_\omega(\fS,\lambda) \subset \mathcal{Z}^{\rm mbl}_\omega(\fS,\lambda) \subset \mathcal{Z}_\omega(\fS,\lambda)$, with ${\bf t} = (t_v)_{v\in V_\lambda} \in \mathcal{B}_\lambda$, where $\mathcal{B}_\lambda \subset \mathbb{Z}^{V_\lambda}$ is as defined in \eqref{B_lambda}. From \eqref{mathcal_S_e_on_Laurent_monomial} we have $\mathcal{S}_e(Z^{\bf t}) = Z^{{\bf t}'}$, with ${\bf t}' = (t'_v)_{v\in V_{\lambda_e}} \in \mathbb{Z}^{V_{\lambda_e}}$ given by the equation $t'_v = t_{{\bf pr}_e(v)}$, where ${\bf pr}_e$ stands for the gluing projection map $\mathsf{Cut}_e(\fS) \to \fS$ and also the induced map $V_{\lambda_e} \to V_\lambda$. Thus, for each ideal triangle $\tau$ of $\lambda$, if we denote by $\tau'$ the corresponding unique ideal triangle of $\lambda_e$, then ${\bf pr}_e$ induces a bijection from $V_{\lambda_e} \cap \tau'$ to $V_\lambda \cap \tau$, and we have $t'_v = t_{{\bf pr}_e(v)}$ for all $v \in V_{\lambda_e}\cap \tau'$. In view of \eqref{B_lambda}, the balancedness condition is described for each ideal triangle, and hence one can observe that the balancedness of ${\bf t}$ (with respect to $\lambda$) implies that of ${\bf t}'$ (with respect to $\lambda_e$). So $\mathcal{S}_e(Z^{\bf t}) = Z^{{\bf t}'} \in \mathcal{Z}^{\rm mbl}_\omega(\mathsf{Cut}_e(\fS),\lambda_e)$ lies in $\mathcal{Z}^{\rm bl}_\omega(\mathsf{Cut}_e(\fS),\lambda_e)$.
\end{proof}

By using the splitting homomorphisms, we will see that proving Proposition \ref{prop.Theta_restricts_to_balanced} boils down to the case when $\fS$ is a disjoint union of polygons. To settle the polygon case, we invoke the following lemma:
\begin{lemma}\cite[Theorem 11.7 and section 14.2]{LY23}\label{lem.quantum_trace_on_fractions}
    Suppose that the pb surface $\fS$ is a polygon $\mathbb{P}_k$ with $k\ge 3$, and $\lambda$ is a triangulation of $\fS$. 
    
    Then the quantum trace map $
    \overline{\rm tr}_\lambda : \overline{\cS}_\omega(\fS) \to \mathcal{Z}^{\rm bl}_\omega(\fS,\lambda)$ induces an isomorphism
    \begin{align*}
    {\rm Frac}(
    \overline{\rm tr}_\lambda)\colon {\rm Frac}(\overline{\cS}_\omega(\fS)) \to {\rm Frac}(\mathcal{Z}^{\rm bl}_\omega(\fS,\lambda))
\end{align*}
between the corresponding skew-fields of fractions.
\end{lemma}

\begin{proposition}\label{prop:Theta_restricts_to_bl_polygons}
    Proposition \ref{prop.Theta_restricts_to_balanced} holds when $\fS$ is a disjoint union of polygons.
\end{proposition}

\begin{proof}
    It is straightforward to see that it suffices to show the statement when $\fS$ is a (single) polygon $\mathbb{P}_k$, with $k\ge 3$. \red{Lemma \ref{lem.quantum_trace_on_fractions} applied to $\lambda'$ says that any element of ${\rm Frac}(\mathcal{Z}^{\rm bl}_\omega(\fS,\lambda'))$, say $U$, can be written as the ratio $U = (\overline{\mathrm{tr}}_{\lambda'}(U_1))^{-1}(\overline{\mathrm{tr}}_{\lambda'}(U_2))$, for some $U_1,U_2 \in \overline{\cS}_\omega(\fS)$. Applying $\Theta^\omega_{\lambda\lambda'}$ we get $\Theta^\omega_{\lambda\lambda'}(U) = (\Theta^\omega_{\lambda\lambda'}\overline{\mathrm{tr}}_{\lambda'}(U_1))^{-1}(\Theta^\omega_{\lambda\lambda'}\overline{\mathrm{tr}}_{\lambda'}(U_2))= (\overline{\mathrm{tr}}_{\lambda}(U_1))^{-1} (\overline{\mathrm{tr}}_{\lambda}(U_2))$ by our Theorem \ref{thm-main-compatibility}. The last ratio belongs to ${\rm Frac}(\mathcal{Z}^{\rm bl}_\omega(\fS,\lambda))$ by Lemma \ref{lem.quantum_trace_on_fractions} applied to $\lambda$. This means that} 
    $\Theta^\omega_{\lambda\lambda'}$ restricts to a homomorphism ${\rm Frac}(\mathcal{Z}^{\rm bl}_\omega(\fS,\lambda')) \to {\rm Frac}(\mathcal{Z}^{\rm bl}_\omega(\fS,\lambda))$. \quad \red{\it [End of Proof of Proposition \ref{prop:Theta_restricts_to_bl_polygons}]} 
\end{proof}

We are now ready to prove Proposition \ref{prop.Theta_restricts_to_balanced}.

\begin{proof}[Proof of Proposition \ref{prop.Theta_restricts_to_balanced}]
By the consistency \eqref{Theta_omega_lambda_lambda_prime_consistency}, it suffices to prove the case when $\lambda$ and $\lambda'$ are related by a flip at an edge, say $e$. As in the proof of Theorem \ref{thm-main-compatibility}, we cut along all edges in $
E$ defined as in \eqref{lambda_0}, i.e. along all non-boundary edges of $\lambda$ not equal to $e$. The resulting pb surface is denoted by $
\fS_E:=\mathsf{Cut}_E(\fS)$, which is a disjoint union of a single copy of $\mathbb{P}_4$ which contains $e$ and some number of copies of triangles $\mathbb{P}_3$. Let $
\lambda_E$ and $
\lambda'_E$ be the triangulations of $
\fS_E$ induced by $\lambda$ and $\lambda'$. Consider the following diagram
\begin{align}
    \label{compatibility_diagram_for_mbl_and_bl}
    \hspace*{-7mm} \xymatrix@C-1,5mm{
    {\rm Frac}(\mathcal{Z}^{\rm bl}_\omega(\fS,\lambda')) \ar[rd]^{\iota} \ar[ddd]^{
    \mathcal{S}_E^{\rm bl}} \ar@{-->}[rrr]^-{\exists 
    (\Theta^\omega_{\lambda\lambda'})^{\rm bl}?} & & & {\rm Frac}(\mathcal{Z}^{\rm bl}_\omega(\fS,\lambda)) \ar[ld]_{\iota} \ar[ddd]^{
    \mathcal{S}^{\rm bl}_E} \\
    & {\rm Frac}(\mathcal{Z}^{\rm mbl}_\omega(\fS,\lambda')) \ar[r]^-{\Theta^\omega_{\lambda\lambda'}} \ar[d]^{\mathcal{S}_{
    E}} & {\rm Frac}(\mathcal{Z}^{\rm mbl}_\omega(\fS,\lambda)) \ar[d]^{\mathcal{S}_{
    E}} & \\
    & {\rm Frac}(\mathcal{Z}^{\rm mbl}_\omega(
    \fS_E,
    \lambda'_E)) \ar[r]^-{\Theta^\omega_{
    \lambda_E \lambda'_E}} & {\rm Frac}(\mathcal{Z}^{\rm mbl}_\omega(
    \fS_E,
    \lambda_E)) & \\
    {\rm Frac}(\mathcal{Z}^{\rm bl}_\omega(
    \fS_E,
    \lambda'_E)) \ar[ur]_{\iota} \ar[rrr]^-{
    (\Theta^\omega_{
    \lambda_E \lambda'_E})^{\rm bl}} & & & {\rm Frac}(\mathcal{Z}^{\rm bl}_\omega(
    \fS_E,
    \lambda_E)) \ar[ul]^{\iota}
    }
\end{align}
The inner square of the above diagram is the outermost square of the diagram in \eqref{eq-cutting-maps}, hence is commutative. The four diagonal arrows denoted by $\iota$ are the natural embeddings. So the left and the right trapezoids are commutative, by Lemma \ref{lem:cutting_homomorphism_for_mbl_FG_algebra_restritcs_to_bl}. The bottom horizontal map comes from Proposition \ref{prop:Theta_restricts_to_bl_polygons}; so the bottom trapezoid commutes. The goal is to show that there exists the top horizontal map $
(\Theta^\omega_{\lambda\lambda'})^{\rm bl}$ that makes the top trapezoid commutes; since $\iota$ are injective, such a map $
(\Theta^\omega_{\lambda\lambda'})^{\rm bl}$ is unique if it exists. From now on, we may omit writing $\iota$ if it's clear from the context.

Let $Z^{{\bf t}'} \in \mathcal{Z}^{\rm bl}_\omega(\fS,\lambda') \subset {\rm Frac}(\mathcal{Z}^{\rm bl}_\omega(\fS,\lambda'))$, belonging to the upper left corner of the diagram; that is, ${\bf t}' = (t'_v)_{v\in V_{\lambda'}} \in \mathcal{B}_{\lambda'} \subset \mathbb{Z}^{V_{\lambda'}}$, where $\mathcal{B}_{\lambda'}$ is defined in \eqref{B_lambda}. As $Z^{{\bf t}'} = \iota(Z^{{\bf t}'}) \in {\rm Frac}(\mathcal{Z}^{\rm mbl}_\omega(\fS,\lambda'))$, we may apply $\Theta^\omega_{\lambda\lambda'}$ and get an element $\Theta^\omega_{\lambda\lambda'}(Z^{{\bf t}'}) \in {\rm Frac}(\mathcal{Z}^{\rm mbl}_\omega(\fS,\lambda))$. All we need to check is whether $\Theta^\omega_{\lambda\lambda'}(Z^{{\bf t}'})$ lies in ${\rm Frac}(\mathcal{Z}^{\rm bl}_\omega(\fS,\lambda))$.

Recall from \eqref{two_seeds_connected_by_sequence_of_mutations} again that the cluster seeds $\mathcal{D}_\lambda$ and $\mathcal{D}_{\lambda'}$ are connected by the sequence of mutations at the vertices $v_1,\ldots,v_r$. As in the proof of Lemma \ref{lem.Theta_on_normalizer_monomial} (also Lemma \ref{lem-mutation-cutiing}), let $\mathcal{D}_0 = \mathcal{D}_\lambda$, and $\mathcal{D}_k = \mu_{v_k}(\mathcal{D}_{k-1})$ for $k=1,\ldots,r$. In particular, $\mathcal{D}_r = \mathcal{D}_{\lambda'}.$ Let ${\bf t}^{(r)} := {\bf t}'$, and define ${\bf t}^{(k)}$, $k=r-1,\ldots,2,1,\red{0}$, recursively by the equation $\nu'_{v_k}(Z^{{\bf t}^{(k)}}) = Z^{{\bf t}^{(k-1)}}$; see \eqref{nu_prime_v_k_on_exponents} for the formula of ${\bf t}^{(k-1)}$ in terms of ${\bf t}^{(k)}$. Note that $\nu^\omega_{v_k}(Z^{{\bf t}^{(k)}}) = \nu^{\sharp \omega}_{v_k}(\nu'_{v_k}(Z^{{\bf t}^{(k)}})) = \nu^{\sharp \omega}_{v_k}(Z^{{\bf t}^{(k-1)}}) = Z^{{\bf t}^{(k-1)}}  \cdot (*)_{k-1}$ for some nonzero element $(*)_{k-1}$ of ${\rm Frac}(\mathcal{X}_q(\mathcal{D}_{k-1}))$, due to Lemma \ref{lem.nu_sharp_well-defined}. Note also that the isomorphism $\nu^\omega_{v_k} : {\rm Frac}(\mathcal{Z}^{\rm mbl}_\omega(\mathcal{D}_k)) \to {\rm Frac}(\mathcal{Z}^{\rm mbl}_\omega(\mathcal{D}_{k-1}))$ restricts to an isomorphism $\Phi^q_{v_k} : {\rm Frac}(\mathcal{X}_q(\mathcal{D}_k)) \to {\rm Frac}(\mathcal{X}_q(\mathcal{D}_{k-1}))$ (Lemma \ref{lem:nu_extends_mu}). Thus we arrive at
$$
\Theta^\omega_{\lambda\lambda'}(Z^{{\bf t}^{(r)}}) = \nu^\omega_{v_r} \cdots \nu^\omega_{v_1}(Z^{{\bf t}^{(r)}}) =
\nu^\omega_{v_r} \cdots \nu^\omega_{v_2}(Z^{{\bf t}^{(r-1)}} \cdot (*)_{r-1}) = \cdots = Z^{{\bf t}^{(0)}} \cdot (*),
$$
where ${\bf t}^{(0)} \in \mathbb{Z}^{V_\lambda}$ and $(*)$ is a nonzero element of ${\rm Frac}(\mathcal{X}_q(\mathcal{D}_0)) = {\rm Frac}(\mathcal{X}_q(\fS,\lambda))$. It remains to show that ${\bf t}^{(0)}$ is balanced, i.e. lies in $\mathcal{B}_\lambda$.

We now apply $\mathcal{S}_{\red{E}}$ to $Z^{{\bf t}^{(0)}} \cdot (*)\red{=\Theta^\omega_{\lambda\lambda'}(Z^{{\bf t}^{(r)}})=\Theta^\omega_{\lambda\lambda'}(\iota(Z^{{\bf t}'}))}$. Then, from the diagram,
\begin{align*}
\mathcal{S}_{E}( \Theta^\omega_{\lambda\lambda'}(\iota(Z^{{\bf t}'})))
& = \Theta^\omega_{
\lambda_E \lambda'_E} (\mathcal{S}_{
E} (\iota(Z^{{\bf t}'}))) \quad (\because \mbox{inner square}) \\
& = \Theta^\omega_{
\lambda_E \lambda'_E}( \iota( 
\mathcal{S}^{\rm bl}_E (Z^{{\bf t}'}))) \quad (\because \mbox{left trapezoid}) \\
& = \iota(
(\Theta^\omega_{
\lambda_E \lambda'_E})^{\rm bl} (
\mathcal{S}^{\rm bl}_E (Z^{{\bf t}'}))) \quad (\because \mbox{bottom trapezoid})
\end{align*}
Note that $
\mathcal{S}^{\rm bl}_E(Z^{{\bf t}'})$ is of the form $Z^{{\bf t}''}$ with ${\bf t}'' \in \mathcal{B}_{
\red{\lambda'_E}} \subset \mathbb{Z}^{V_{
\red{\lambda'_E}}}$, so by the similar arguments as $\Theta^\omega_{\lambda\lambda'}$ we see that $\iota(
(\Theta^\omega_{\lambda_E \lambda'_E})^{\rm bl} (
\mathcal{S}^{\rm bl}_E (Z^{{\bf t}'})))$ is of the form $Z^{{\bf t}'''} \cdot (**)$, with ${\bf t}''' \in \mathbb{Z}_{V_{
\red{\lambda_E}}}$ and $0 \neq (**) \in {\rm Frac}(\mathcal{X}_q(\fS_{\red{E}}, 
\red{\lambda_E}))$. As the image of $
(\Theta^\omega_{\lambda_E \lambda'_E})^{\rm bl}$ lies in ${\rm Frac}(\mathcal{Z}^{\rm bl}_\omega(\fS_{\red{E}}, 
\red{\lambda_E}))$, and since ${\rm Frac}(\mathcal{X}_q(\fS_{\red{E}}, 
\lambda_{\red{E}}))$ is contained in ${\rm Frac}(\mathcal{Z}^{\rm bl}_\omega(\fS_{\red{E}}, 
\red{\lambda_E}))$, it follows that $Z^{{\bf t}'''}$ lies in $\mathcal{Z}^{\rm bl}_\omega(\fS_{\red{E}}, 
\red{\lambda_E})$, i.e. ${\bf t}'''$ lies in $\mathcal{B}_{
\red{\lambda_E}}$. On the other hand, $\mathcal{S}_{
E}( \Theta^\omega_{\lambda\lambda'}(\iota(Z^{{\bf t}'}))) = \mathcal{S}_{
E}(Z^{{\bf t}^{(0)}}) \cdot \mathcal{S}_{
E}((*))$; here $\mathcal{S}_{
E}(Z^{{\bf t}^{(0)}}) = Z^{{\bf t}''''}$ for some ${\bf t}'''' \in \mathbb{Z}^{V_{
\red{\lambda_E}}}$ and $\mathcal{S}_{
E}((*))$ is a nonzero element of ${\rm Frac}(\mathcal{X}_q(\fS_{\red{E}}, 
\red{\lambda_E}))$. By the equality $Z^{{\bf t}''''} \cdot \mathcal{S}_{E}((*)) = Z^{{\bf t}'''} \cdot (**)$, it follows that ${\bf t}''''$ also belongs to $\mathcal{B}_{
\red{\lambda_E}}$. We use the arguments in the proof of Lemma \ref{lem:cutting_homomorphism_for_mbl_FG_algebra_restritcs_to_bl}. Since the balancedness condition for a Laurent monomial is described for each triangle in a triangulation, and since the splitting homomorphism $\mathcal{S}_{E}$ preserves this condition, it follows that the balancedness of $Z^{{\bf t}''''} = \mathcal{S}_{
E}(Z^{{\bf t}^{(0)}})$ implies that of $Z^{{\bf t}^{(0)}}$, i.e. we get ${\bf t}^{(0)} \in \mathcal{B}_\lambda$, as desired. This finishes the proof of Proposition \ref{prop.Theta_restricts_to_balanced}.
\end{proof}

The two 
\red{observations} below easily follow.
\begin{lemma}
    The balanced quantum coordinate change isomorphism \\ $
    (\Theta^\omega_{\lambda\lambda'})^{\rm bl} : {\rm Frac}(\mathcal{Z}^{\rm bl}_\omega(\fS,\lambda')) \to {\rm Frac}(\mathcal{Z}^{\rm bl}_\omega(\fS,\lambda))$ obtained in Proposition \ref{prop.Theta_restricts_to_balanced} extends the map $\Phi^q_{\lambda\lambda'} : {\rm Frac}(\mathcal{X}_q(\fS,\lambda')) \to {\rm Frac}(\mathcal{X}_q(\fS,\lambda))$. (see 
    \red{Lemma \ref{lem-extension}}.)
\end{lemma}

\begin{lemma}\label{lem-consistence-Thata}
    The consistency equation holds (see \eqref{Theta_omega_lambda_lambda_prime_consistency}):
    $$
    (\Theta^\omega_{\lambda\lambda'})^{\rm bl} \circ (\Theta^\omega_{\lambda'\lambda''})^{\rm bl} = (\Theta^\omega_{\lambda\lambda''})^{\rm bl}
    $$
\end{lemma}

By collecting what we have proved, we obtain the following version of Theorem \ref{thm-main-compatibility} written this time with respect to the coordinate change maps $
(\Theta^\omega_{\lambda\lambda'})^{\rm bl}$ between the skew-fields of fractions of balanced algebras. \red{This is the balanced version of our main theorem.}

\begin{theorem}[
\red{naturality} of quantum trace 
\red{under} balanced $n$-root quantum coordinate change maps $
(\Theta^\omega_{\lambda\lambda'})^{\rm bl}$]\label{thm-main-compatibility_bl}
    Let $\fS$ be a triangulable pb surface. For any two triangulations $\lambda$ and $\lambda'$ of $\fS$, the following diagram commutes:
    \begin{align}
        \label{eq-compability-tr-mutation_bl}
        \raisebox{7mm}{\xymatrix{
        & \rdS \ar[dl]_{
        \overline{\rm tr}_{\lambda'}} \ar[dr]^{
        \overline{\rm tr}_\lambda} & \\
        {\rm Frac}(\mathcal{Z}^{\rm bl}_\omega(\fS,\lambda')) \ar[rr]_-{(\Theta^\omega_{\lambda\lambda'})^{\rm bl}} & & {\rm Frac}(\mathcal{Z}^{\rm bl}_\omega(\fS,\lambda)).}}
    \end{align}
    That is, we have
    \begin{align}
    \label{compatibility_equation_bl}
        \overline{\rm tr}_\lambda = (\Theta^\omega_{\lambda\lambda'})^{\rm bl} \circ 
        \overline{\rm tr}_{\lambda'}.
    \end{align}
\end{theorem}

\red{We note that it is this theorem, rather than Theorem \ref{thm-main-compatibility}, that is proved in \cite{BW11, Hiatt} for $n=2$ and in \cite{Kim21} for $n=3$.}

\subsection{A comparison with coordinate change isomorphisms defined by 
L\^e and Yu}

Suppose that $\fS$ is a triangulable pb surface and $\lambda,\lambda'$ are two triangulations of $\fS$.
L{\^ e} and Yu  also constructed 
a coordinate change isomorphism \cite[Theorem 14.2]{LY23}
 $$\overline{\Psi}_{\lambda\lambda'}^X \, \colon \,
 \Fr(\mathcal Z_\omega^{\rm bl}(\fS,\lambda'))\rightarrow
 \Fr(\mathcal Z_\omega^{\rm bl}(\fS,\lambda))$$
 such that the following diagram commutes:
 \begin{align}
        \label{eq-compability-tr-mutation_bl_LY}
        \raisebox{7mm}{\xymatrix{
        & \rdS \ar[dl]_{\overline{\rm tr}_{\lambda'}} \ar[dr]^{\overline{\rm tr}_\lambda} & \\
        {\rm Frac}(\mathcal{Z}^{\rm bl}_\omega(\fS,\lambda')) \ar[rr]_-{\overline{\Psi}_{\lambda\lambda'}^X} & & {\rm Frac}(\mathcal{Z}^{\rm bl}_\omega(\fS,\lambda)).}}
    \end{align}
For any three triangulations $\lambda,\lambda',\lambda'$ of $\fS$,  
the consistency equation holds:
\begin{align}\label{eq-consistence-phi}
    \overline{\Psi}^X_{\lambda\lambda'} \circ \overline{\Psi}^X_{\lambda'\lambda''} = \overline{\Psi}^X_{\lambda\lambda''}.
\end{align}

Here we briefly recall the construction of 
$\overline{\Psi}_{\lambda\lambda'}^X$ of L\^e and Yu.
As we mentioned in Lemma \ref{lem.quantum_trace_on_fractions},
 in \cite{LY23} L{\^ e} and Yu proved that the quantum trace map
 $$\tr\colon\rdS\rightarrow \mathcal Z_\omega^{\rm bl}(\fS,\lambda)$$
 induces an isomorphism 
$$\Fr(\tr)
\colon\Fr(\rdS)\rightarrow \Fr(\mathcal Z_\omega^{\rm bl}(\fS,\lambda))$$
when $\fS$ is a polygon.
They define
 $\overline{\Psi}_{\lambda\lambda'}^X$
 to be $\Fr(\overline{\rm tr}_{\lambda})\circ\Fr(\overline{\rm tr}_{\lambda'})^{-1}$ when $\fS$ is a polygon.
 Then they generalize the definition of $\overline{\Psi}_{\lambda\lambda'}^X$ by cutting $\fS$ into a collection of polygons
 (\cite[
 \S14.3]{LY23}).

We show that the L\^e-Yu coordinate change map $\overline{\Psi}^X_{\lambda\lambda'}$ coincides with the balanced quantum cluster coordinate change map $(\Theta^\omega_{\lambda\lambda'})^{\rm bl}$ which we studied in the present paper. This answers the expectation stated in \cite[\S14.4]{LY23}, which provided the very motivation for the present paper.

 \begin{theorem}[L\^e-Yu quantum coordinate change equals balanced quantum cluster coordinate change]\label{thm-main-comparision}
     Suppose that $\fS$ is a triangulable pb surface and $\lambda,\lambda'$ are two triangulations of $\fS$.
     We have 
     $$\overline{\Psi}_{
     \lambda\lambda'}^X = 
     (\Theta^\omega_{\lambda \lambda'})^{\rm bl}.$$
 \end{theorem}
 \begin{proof}
     Because of 
     \red{Lemma \ref{lem-consistence-Thata}} and equation \eqref{eq-consistence-phi},
     we can assume that $\lambda'$ is obtained from $\lambda$ by a flip on an edge $e\in\lambda$.
      When $\fS$ is a polygon, Lemma \ref{lem.quantum_trace_on_fractions} implies the theorem. That is, Lemma \ref{lem.quantum_trace_on_fractions} implies that an isomorphism $\overline{\Psi}^X_{\lambda\lambda'} : {\rm Frac}(\mathcal{Z}^{\rm bl}_\omega(\fS,\lambda')) \to {\rm Frac}(\mathcal{Z}^{\rm bl}_\omega(\fS,\lambda))$ making the diagram in \eqref{eq-compability-tr-mutation_bl_LY} is unique (see \cite[Cor.14.3]{LY23} for such a uniqueness), hence equals $(\Theta^\omega_{\lambda\lambda'})^{\rm bl}$, in view of Theorem \ref{thm-main-compatibility_bl}.
      
      When $\fS$ is a general triangulable pb surface, we can use the trick as in the proofs of Theorem \ref{thm-main-compatibility} and Proposition \ref{prop.Theta_restricts_to_balanced}. We cut $\fS$ into a $\mathbb P_4$ containing $e$ and some $\mathbb P_3$. 
Consider the outermost square of the diagram in \eqref{compatibility_diagram_for_mbl_and_bl}, which is commutative. This diagram still commutes if we replace the top horizontal map by $\overline{\Psi}^X_{\lambda\lambda'}$ and the bottom horizontal map by $\overline{\Psi}^X_{\lambda_E\lambda'_E}$, because the definition of the L\^e-Yu maps $\overline{\Psi}^X_{\lambda\lambda'}$ (see \cite[equation (253)]{LY23}) implies that they are compatible with the splitting homomorphisms. From the polygon case, we see that the original bottom horizontal map $(\Theta^\omega_{\lambda_E\lambda'_E})^{\rm bl}$ coincides with $\overline{\Psi}^X_{\lambda_E\lambda'_E}$.

Since the right vertical map $\mathcal{S}^{\rm bl}_E$ in this original square diagram are injective, it follows that the top horizontal map that makes this diagram commutative is unique, if exists. By this uniqueness, it follows that $(\Theta^\omega_{\lambda\lambda'})^{\rm bl}$ coincides with $\overline{\Psi}^X_{\lambda\lambda'}$, as desired.

 \end{proof}

 \begin{remark}
     The proof of Theorem \ref{thm-main-comparision} implies that $\overline{
     \Psi}_{
     \lambda \lambda'}^X
     = 
     (\Theta^\omega_{\lambda\lambda'})^{\rm bl}$ as long as $\overline{
     \Psi}_{\lambda'\lambda}^X$ makes the diagram \eqref{eq-compability-tr-mutation_bl_LY} commute and satisfies the equation \eqref{eq-consistence-phi}. 
     This shows the uniqueness  of 
     $\overline{\Theta}_{
     \lambda\lambda'}^\omega$ with properties stated in Theorem \ref{thm-main-compatibility}
and 
\red{Lemma \ref{lem-consistence-Thata}}.
\end{remark}

\end{document}